\documentclass{article}
\usepackage{mathtools,amssymb,amsthm,xcolor,tikz-cd,dsfont}
\usepackage{multirow}
\usepackage{graphicx}
\usepackage{float}
\usepackage[caption=false]{subfig}
\usepackage[font=small,labelsep=none]{caption}

\usepackage{cleveref}
\usepackage[T1]{fontenc}
\usepackage{CJKutf8}
\usepackage[english]{babel}
\usepackage{url}
\usepackage{BOONDOX-calo}

\DeclareMathOperator{\image}{im}

\newenvironment{Japanese}{%
  \CJKfamily{min}%
  \CJKtilde
  \CJKnospace}{}

\usepackage[mathscr]{euscript}
\usepackage{todonotes}
\newcommand{\CC}{\mathbb{C}}
\newcommand{\FF}{\mathbb{F}}
\newcommand{\NN}{\mathbb{N}}
\newcommand{\QQ}{\mathbb{Q}}
\newcommand{\RR}{\mathbb{R}}
\newcommand{\ZZ}{\mathbb{Z}}

\newcommand{\NNN}{{\NN^*}}

\newcommand{\NNs}{{\NN^*_s}}
\newcommand{\NNodd}{{\NN^*_2}}

\newcommand{\floor}[1]{{\lfloor{#1}\rfloor}}

\newcommand{\Set}{{\mathscr{S}\textnormal{et}}}

\newcommand{\Kei}{{\mathscr{K}\textnormal{ei}}}
\newcommand{\AKei}{{\mathscr{A}\textnormal{ug}\Kei}}
\newcommand{\Grp}{{\mathscr{G}\textnormal{rp}}}

\newcommand{\Hom}{\textnormal{Hom}}

\newcommand{\fin}{\textnormal{fin}}
\newcommand{\Fin}{{\Set^\fin}}
\newcommand{\Col}{\textnormal{Col}}
\newcommand{\col}{\textnormal{col}}
\newcommand{\ccol}{{\textnormal{c}\widehat{\textnormal{o}} \textnormal{l}}}
\newcommand{\dom}{\textnormal{dom}}

\newcommand{\KeiFunc}[1]{{\textnormal{Kei}_{#1}}}
\newcommand{\AKeiFunc}[1]{{\mathscr{A}_{#1}}}

\newcommand{\m}{\mathfrak{m}}
\newcommand{\cont}{\textnormal{cont}}

\newcommand{\p}{\mathfrak{p}}
\newcommand{\q}{\mathfrak{q}}

\newcommand{\Inn}{\textnormal{Inn}}
\newcommand{\Top}{{\mathscr{T}\textnormal{op}}}

\newcommand{\qq}[2]{{{}^{#1}{#2}}}
\newcommand{\Aut}{\textnormal{Aut}}
\newcommand{\surj}{\textnormal{surj}}
\newcommand{\concise}{\textnormal{cs}}
\newcommand{\Stab}{\textnormal{Stab}}
\newcommand{\Pro}{\textnormal{Pro-}}

\newcommand{\op}{\textnormal{op}}

\newcommand{\GenType}[2]{{\mathscr{W}(#1,#2)}}
\newcommand{\counter}{\mathcal N}
\newcommand{\ccounter}{\mathcal M}
\newcommand{\avg}{\mathcal E}

\newcommand{\D}{{\mathscr D}}
\renewcommand{\a}{\mathscr{A}}

\newcommand{\Spec}{\textnormal{Spec}}

\newcommand{\Id}{\textnormal{Id}}
\newcommand{\iso}{{\xrightarrow{\sim}}}
\newcommand{\ab}{{ab}}

\newcommand{\Disc}{\textnormal{Disc}}

\newcommand{\F}{\mathfrak{F}}

\renewcommand{\L}{\mathfrak L}
\newcommand{\K}{\mathfrak{K}}
\newcommand{\G}{\mathfrak{G}}
\newcommand{\A}{\mathfrak{A}}

\renewcommand{\k}{\mathscr{K}}
\newcommand{\triv}{\mathscr{T}}
\newcommand{\kk}{{\k'}}
\newcommand{\kl}{\mathscr{L}}

\newcommand{\joyce}{\mathscr{J}}
\newcommand{\RGB}{\mathscr{R}_3}

\newcommand{\uno}{\mathds 1}
\newcommand{\ley}{\mathds 1_{(s)}}
\newcommand{\gy}{\mathds 1_{s}}

\newcommand{\muy}{\mu_{(s)}}
\newcommand{\fya}{f_s^{(\powa)}}

\newcommand{\powa}{{ \mathcal{a}}}
\newcommand{\powb}{{ \mathcal{b}}}

\usepackage[utf8]{inputenc}
\usepackage[english]{babel}

\newcommand{\kei}{
\text{\begin{CJK}{UTF8}{}
\!\!\!\!\!\!
\begin{Japanese}
圭
\end{Japanese}
\!\!\!\!
\end{CJK}}
}

\newcommand{\Jac}[2]{\left(\frac{#1}{#2}\right)}

\newcommand{\Gal}{\textnormal{Gal}}
\renewcommand{\O}{\mathcal O}

\newtheorem{theorem}{Theorem}[section]
\newtheorem{proposition}[theorem]{Proposition}
\newtheorem{corollary}[theorem]{Corollary}
\newtheorem{lemma}[theorem]{Lemma}
\newtheorem{conjecture}[theorem]{Conjecture}

\theoremstyle{definition}
\newtheorem{definition}[theorem]{Definition}
\newtheorem{example}[theorem]{Example}

\newtheorem{remark}[theorem]{Remark}

\newcommand{\SQ}{\mathscr{Q}}
\newcommand{\SG}{\mathscr{G}}

\title{Arithmetic Kei Theory}
\author{Ariel Davis, Tomer M. Schlank}
\date{}
\begin{document}
\maketitle

\abstract{A $\kei$ (kei), or 2-quandle, is an algebraic structure one can use to produce a numerical invariant of links, known as coloring invariants. Motivated by Mazur's analogy between prime numbers and knots, we define for every finite kei $\k$ an analogous coloring invariant $\col_\k(n)$ of square-free integers. This is achieved by defining a fundamental kei $\kei_n$ for every such $n$. We conjecture that the asymptotic average order of $\col_\k$ can be predicted to some extent by the colorings of random braid closures. This conjecture is fleshed out in general, building on previous work, and then proven for several cases.}
\begin{center}  
\includegraphics[width=0.86\textwidth]{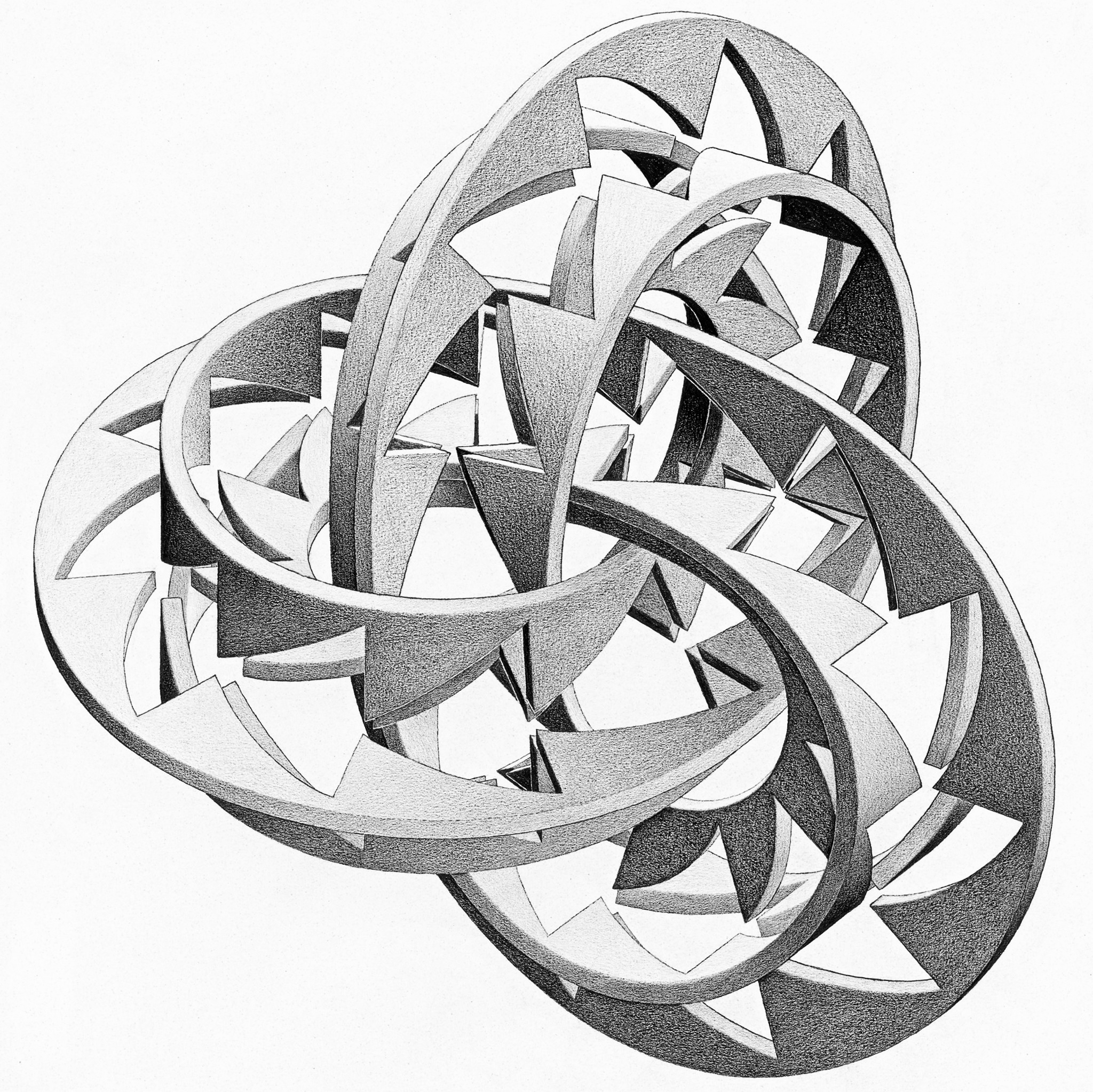}\\
M. C. Escher, \emph{Knot}, 1966
\end{center}
\pagebreak

\tableofcontents

\section{Introduction}
There is an analogy that likens prime numbers $p$ to knots in the sphere $S^3$.
This analogy is the observation of Barry Mazur regarding the embedding
\begin{equation}
\label{Eqn: Fp in ZZ}
\Spec(\FF_p)\hookrightarrow\Spec(\ZZ)\end{equation}
of schemes, as seen through the eyes of \'etale topology and homotopy. \emph{\'Etale topology} is a tool in algebraic geometry developed by Grothendieck et al in \cite{BOOK:SGA4}. It enriches the theory of schemes with topological notions such as well-behaved cohomology groups for finite locally-constant sheaves. While the concept was originally conceived with smooth varieties over fields in mind, one can attach an \'etale topology (or more precisely, an \'etale \emph{site}) to any scheme. In the case of spectra of rings of integers in number fields, this provides the framework for Artin-Verdier duality. Through Artin-Verdier duality we are justified in thinking about the geometric objects attached to rings of integers as smooth threefolds - see \cite{ARTICLE:ArtinVerdier1964} or \cite{ARTICLE:Mazur1973} for details. The \emph{\'etale homotopy type} defined by M. Artin and Mazur in \cite{BOOK:ArtinMazur}
takes this idea even further, in granting access to a wider array of hompotopical tools by producing for any scheme $X$ a profinite approximation $\acute Et(X)$ of a homotopy type. The homotopy type $\acute Et(\Spec(\ZZ))$ is simply-connected, and $\acute Et(\Spec(\FF_p))$ is a profinite circle.
Therefore \'etale homotopy sees (\ref{Eqn: Fp in ZZ}) as the embedding of a circle into a simply-connected threefold - that is, a knot - and  for square-free $n\in\NN$, the embedding
$$\Spec(\ZZ/n\ZZ)\simeq\coprod_{p|n}\Spec(\FF_p)\hookrightarrow \Spec(\ZZ)$$
as a disjoint union of knots - a link. One is motivated to understand how far we can take this analogy. Various notions from knot theory have been interpreted in number theory based on this analogy, many of which can be found in Morishita's book \cite{BOOK:Morishita}. One salient example is the linking number $link(K_1,K_2)$ of two oriented knots $K_1,K_2$, which corresponds to the homology class
$$[K_1]\in H_1(S^3\setminus K_2;\ZZ)\simeq\ZZ.$$
A theorem by Gauss famously states that the linking number is symmetric in $K_1,K_2$.
The Jacobi symbol $\Jac pq$ of odd primes $p\neq q$ is seen as analogous to
$$(-1)^{link(K_1,K_2)},$$
and quadratic reciprocity as a manifestation the linking number's symmetry. A more sophisticated example is the Alexander module of a knot $K$, closely related to the Alexander polynomial of $K$. This has an arithmetic interpretation as the Iwasawa module of a prime $p$. The analogy between knots and primes does not produce de facto links in $S^3$ for individual square-free integers $n\in\NN$. That is, while certain knot-theoretical notions are emulated in number theory, the values that they take are not recovered knot-theoretically from any knot or link. Indeed if that were the case, we would expect actual symmetry in quadratic reciprocity. 

There is however a way in which one can say that \emph{random} square-free integers resemble \emph{random} links. Alexander's theorem states in \cite{ARTICLE:Alexander1923} that every link is obtained as the closure $\overline \sigma$ of a braid $\sigma$. For $k\in\NN$, the set of all braids on $k$ strands, together with concatenation, form the Artin braid group $B_k$ - see \cite{ARTICLE:Artin1947}. The naive notion of a random braid on $k$ strands is ill-defined because $B_k$ is infinite discrete for $k\ge 2$. 
That said, if a function $f\colon B_k\to \RR$ factors through a finite quotient of $B_k$, then we consider the locally-constant extension $\widehat f$ of $f$ to $\widehat B_k$, the profinite completion of $B_k$. In this case we interpret ``the distribution of $f$''
to mean the distribution of $\widehat f$, where $\widehat B_k$ has normalized Haar measure. 
Such is the case with the number of connected components in the closure $\overline\sigma$. For $\sigma\in B_k$, the number of connected components of $\overline\sigma$ equals
$$|\pi_0(\overline \sigma)|=|\{1,\dots, k\}/\sigma|.$$
This is determined by how $\sigma$ permutes the endpoints of the $k$ strands. Hence
$$|\pi_0(\overline \sigma)|\colon B_k\to\NN$$
factors through the symmetric group $S_k$.
As per the above, the closure of a random braid $\sigma\in B_k$ is a knot with probability $\frac 1k$. For square-free $n\in\NN$, the connected components of $\acute Et(\Spec(\ZZ/n\ZZ))$ correspond to the prime factors of $n$. A square-free number $n\in\NN$ of magnitude $X$ is prime with probability roughly $\frac 1{\log X}$. From here one draws the analogy that 
\begin{equation}
\label{Eqn: braid philosophy}
\begin{matrix}
	\textnormal{Random square-free integers $n\in\NN$ of magnitude $X$ resemble}\\
	\textnormal{closures of random braids on approximately $\log X$ strands.}
\end{matrix}
\end{equation}

Another numerical invariant of links, introduced by Fox in \cite{BOOK:CrowellFox1963}\footnote{See \cite{ARTICLE:Przytycki1998} for a more comprehensive treatment of the subject.}, is the number of \emph{$3$-colorings} of a link $L$. A $3$-coloring of a link diagram $D$ constitutes a choice of color $x\in \ZZ/3\ZZ$ for each arc in $D$ such that at each crossing, the colors $x,y,z$ as in \cref{Fig: crossing coloring} must satisfy $y+z=2x$. This condition ensures that the number of $3$-colorings is invariant under Reidemeister moves, and by \cite{BOOK:Reidemeister} is therefore an isotopy invariant of the link.

\begin{figure}[H]
    $$\begin{tikzpicture}
			\draw[-] (2,2)--(0,0);
			\draw[-,line width=10pt, draw=white] (0,2)--(2,0);
			\draw[-] (0,2)--(2,0);
			\draw (0.5,1.5)node [anchor=south]{$\;\;x$};
			\draw (1.5,1.5)node[anchor = west]{$\;y$};
			\draw (0.5,0.5)node[anchor=east]{$z\;$};
	\end{tikzpicture}$$
    \caption{}
    \label{Fig: crossing coloring}
\end{figure}

The fact that $z$ in \cref{Fig: crossing coloring} is determined by $x,y$ is necessary in order to transport colorings across a Reidemeister $2$ move, as depicted in \cref{fig: R2}.
The concept of $3$-colorings thus 
generalizes to coloring links using a set of colors $\k$ equipped with a binary operator
$$(x,y)\mapsto \qq xy.$$
Here a $\k$-coloring of a crossing as in \cref{Fig: crossing coloring} is valid if $z=\qq xy$. 
For $\k$-colorings to be transported bijectively across all Reidemeister moves,

\begin{figure}[H]
\begin{minipage}{.3\textwidth}
    \subfloat[R1]{
    $$\begin{tikzpicture}
    	\begin{scope}[scale=0.75]
			\draw[-] (0,0)--(0,2);
			\draw (0,1)node [anchor=east, style={scale=0.7}]{$\;\;x$};
			\draw (0.5,1)node{$\rightsquigarrow$};
			\begin{scope}[xshift=35pt]
			\draw[-] (0,0).. controls (0,1) and (0.2,1.4) ..(0.6,1.4) .. controls (0.8,1.4) and (1,1.25) .. (1,1);
			\draw[-, line width=5pt, white] (0,2).. controls (0,1) and (0.2,0.6) ..(0.6,0.6) .. controls (0.8,0.6) and (1,0.75) .. (1,1);			\draw[-] (0,2).. controls (0,1) and (0.2,0.6) ..(0.6,0.6) .. controls (0.8,0.6) and (1,0.75) .. (1,1);
			\draw (0,2)node [anchor=north east, style={scale=0.7}]{$\;\;x$};
			\draw (0,0)node [anchor=south east, style={scale=0.7}]{$\;\;x$};
			\end{scope}
			\end{scope}
	\end{tikzpicture}$$
    \label{Fig: R1}}
\end{minipage}
\hfill    
\begin{minipage}{.3\textwidth}
    \subfloat[R2]{
    $$\begin{tikzpicture}
    	\begin{scope}[scale=0.75]
			\draw (-0.15,1)node [anchor=east, style={scale=0.7}]{$x$};
			\draw[-] (-0.15,0)--(-0.15,2);
			\draw[-] (0.5,0)--(0.5,2);
			\draw (0.5,1)node [anchor=west, style={scale=0.7}]{$y$};
			\draw (1.35,1)node{$\rightsquigarrow$};
			\begin{scope}[xshift=60]
			\draw[-] (0.5,0) .. controls (-0.25,0.75) and (-0.25,1.25) .. (0.5,2);
			\draw[-, line width = 5pt, white] (-0.15,0).. controls (0.6,0.75) and (0.6,1.25) .. (-0.15,2);
			\draw[-] (-0.15,0).. controls (0.6,0.75) and (0.6,1.25) .. (-0.15,2);
			\draw (0.85,1)node [anchor=east, style={scale=0.7}]{$x$};
			\draw (-0.05,1)node [anchor=east, style={scale=0.7}]{$z$};
			\draw (0.85,2)node [anchor=east, style={scale=0.7}]{$y$};
			\draw (0.85,0)node [anchor=east, style={scale=0.7}]{$y$};
			\end{scope}
			\end{scope}
	\end{tikzpicture}$$
    \label{fig: R2}}
\end{minipage}
	\hfill    
\begin{minipage}{.3\textwidth}
    \subfloat[R3]{
    $$\begin{tikzpicture}
    \begin{scope}[scale=0.75]
			\draw[-] (0,0.6)--(1.6,1.72);
			\draw (1.6,1.72)node[anchor=south, style={scale=0.7}] {$z$};
			\draw (-0.1,0.6)node[anchor=north, style={scale=0.7}] {$\qq x(\qq yz)$};
			\draw[-, line width = 5pt, white][-] (1,0)--(1,2);
			\draw[-] (1,0)--(1,2);
			\draw (1,2)node[anchor=east, style={scale=0.7}] {$y$};
			\draw (1,0)node[anchor=east, style={scale=0.7}] {$\qq xy$};
			\draw[-, line width = 5pt, white] (2,0)--(-0.5,1.75);
			\draw[-] (1.6,0.28)--(0,1.4);
			\draw (0,1.4)node[anchor=south, style={scale=0.7}] {$x$};
			\draw (1.6,0.28)node[anchor=north, style={scale=0.7}] {$x$};
			\draw (2.2,1)node{$\rightsquigarrow$};
			
			\begin{scope}[xshift=75]
			\draw[-] (2,1.4)--(0.4,0.28);
			\draw (2,1.4)node[anchor=south, style={scale=0.7}] {$z$};
			\draw[-, line width = 5pt, white][-] (1,2)--(1,0);
			\draw[-] (1,2)--(1,0);
			\draw (1,0)node[anchor=west, style={scale=0.7}] {$\qq xy $};
			\draw (1,2)node[anchor=west, style={scale=0.7}] {$y$};
			\draw[-, line width = 5pt, white] (0,2)--(2.5,0.25);
			\draw[-] (0.4,1.72)--(2,0.6);
			\draw (2,0.6)node[anchor=north, style={scale=0.7}] {$x$};
			\draw (0.4,1.72)node[anchor=south, style={scale=0.7}] {$x$};
			\draw (0.2,0.38)node[anchor=north, style={scale=0.7}] {${\qq{(\qq xy)}{(\qq xz)} }$};
			\end{scope}
			\end{scope}
	\end{tikzpicture}$$
    \label{fig: R3}}
\end{minipage}
    \label{fig:sub}
    \caption{}
\end{figure}

this operator must satisfy the following corresponding axioms:
\begin{definition}[\cite{ARTICLE:Takasaki1943}\footnote{The terminology and Kanji notation $\kei$ are taken from \cite{ARTICLE:Takasaki1943}, where the notion was first defined. See \cite{ARTICLE:Joyce1982} for keis in relation to knot theory, where they are called \emph{$2$-quandles}.}]
A $\kei$, or \emph{kei}, is a set $\k$ with binary operator
$$\k\times\k\to\k \;\;,\;\;\;\;x,y\mapsto \qq xy$$
satisfying
	\begin{enumerate}
	\item $\qq xx=x$ for all $x\in\k$.
	\item $\qq x{(\qq xy)}=y$ for all $x,y\in\k$.
	\item $\qq x{(\qq yz)}= \qq{(\qq xy)} {(\qq xz)} $ for all $x,y,z\in\k$.
\end{enumerate}
\end{definition}

Keis are therefore an abstaction of Fox's three colors, and kei-colorings a generalization of $3$-colorings. For finite kei $\k$, the number $\col_\k(L)$ 
of $\k$-colorings of a link $L$ is therefore an isotopy invariant of $L$.

The first goal of this paper is to define a number-theoretical analogue of $\k$-colorings.
As is turns out, the kei-colorings of a link $L$ are goverened by  a universal kei $\kei_L$ associated to $L$.
That is, the colorings of $L$ are in natural bijection with representations of $\kei_L$:
$$\Col_\k(L)\simeq \Hom_{\Kei}(\kei_L,\k),$$
where $\Kei$ denotes the category of keis and suitable morphisms.
The kei $\kei_L$ is constructed explicitly with generators and relations extracted from a diagram $D$ of $L$: to every arc in $D$ there is a corresponding generator, and to every crossing a corresponding relation. It is not difficult to see that $\k$-colorings of a diagram $D$ correspond to morphisms $\kei_L\to \k$. The road to $\k$-coloring squarefree integers $n$ rests on the construction of an analogous $\kei_n$ for squarefree $n\in\NN$. For this, a presentation extracted from a diagram is unsuitable. While an instrinsically topological description of $\kei_L$ is already given in \cite{ARTICLE:Joyce1982}, it is the following presentation, based on the work of Winker (\cite{WEBSITE:Winker1984}), that we make the most of (for details see \cref{Subsection: Fund kei of link}).
We take $M_L$ to be the unique branched double cover of $S^3$ with ramification locus equal to $L$, and $\widetilde M_L$ to be the universal cover of $M_L$. Then as a set,  
$$\kei_L=\pi_0(\widetilde M_L\times_{S^3}L).$$
The arithmetic analogue of $\kei_L$ follows immediately. For simplicity assume $n$ is odd:

\begin{definition}[\cref{Def: Ln Quad Field}]
	Let $n\in\NN$ be odd squarefree. By $\L_n$ we denote the  field
$$\L_n=\QQ(\sqrt{n^*})\;\;,\;\;\;\;n^*:=(-1)^{\frac{n-1}2}n.$$
\end{definition}

The field $\L_n$ is the unique\footnote{This assuming $n$ is odd. For even $n$ there are three such fields. In the body of the paper we extend this definition for even $n$. Of the three quadratic fields that ramify precisely at $n$ we choose $\QQ(\sqrt{n^*})$ with $n^*= \pm n$ s.t. $n^*=2\mod 8$. While one can consider the other options, this is the one we find most suitable to our needs.} quadratic number field ramified at precisely $n$. As such,
the morphism
$$\Spec(\O_{\L_n})\to\Spec(\ZZ)$$
is analogous to the branched double cover $M_L\to S^3$.

\begin{definition}[\cref{Def: kei_n as set}]
	Let $n\in\NN$ be squarefree. By $\F_n$ we denote the maximal unramified extension of $\L_n$. We then define $\kei_n$ to be the profinite set
	$$\kei_n:= \Spec (\O_{\F_n})\times_{\Spec(\ZZ)}\Spec(\ZZ/n\ZZ)\simeq\{\p \textnormal{ prime in }\O_{\F_n}\textnormal{ s.t. } \p|n\}.$$
\end{definition}

\begin{theorem}[\cref {Prop: arith fun qdl small} + \cref{Prop: An is Top AugQdl} + \cref{Def: AKei_n and Kei_n}]
${}$
\begin{enumerate}
	\item The set $\kei_n$ has a natural structure of a (profinite) kei.
	\item For finite $\k\in\Kei$, the set of continuous morphisms
	$$\Hom_\Kei^\cont(\kei_n,\k)$$
	is finite.
\end{enumerate}
\end{theorem}

This allows us to define $\k$-colorings of squarefree integers $n\in\NN$ using $\kei_n$ by mimicking the coloring of links:

\begin{definition}[\cref {Def: k-col arith}]
	Let $n\in\NN$ be squarefree and let $\k$ be a finite kei. We define
	$$\col_\k(n):=
\left| \Hom_\Kei^\cont(\kei_n,\k)\right|
\in\NN.$$
\end{definition}

Next, it is our goal to understand the behavior of this function. In particular, we wish to understand the average asymptotic order of $\col_\k(n)$.

\begin{definition}[\cref {Def: curly N and E}]
	Let $\k$ be a finite kei and let $X\ge 1$. We denote by
	$$\textnormal{Avg}_\k(X)=\frac{\sum\limits_{\substack{1\le n\le X\\n\;\textrm{sqr-free}}}\col_\k(n)}{\sum\limits_{\substack{1\le n\le X\\n\;\textrm{sqr-free}}}1}\in\QQ$$
	the average of $\col_\k(n)$ over the set of squarefree integers $1\le n\le X$.
\end{definition}

Recall the idea in (\ref{Eqn: braid philosophy}), as pertaining to certain random variables on $\widehat B_k$. In similar fashion we may interpret the distribution of $\k$-colorings of closures of random braids as above.
In \cite{WEBSITE:DavisSchlankHilbert2023} we compute the average number of $\k$-colorings of braid closures $\overline\sigma$ of braids $\sigma$ in $B_k$, and prove the following theorem:
\begin{theorem}[{\cite[Theorem 1.1]{WEBSITE:DavisSchlankHilbert2023}}]
	Let $\k$ be a finite kei. There exists an integer-valued polynomial $P_\k\in\QQ[x]$ such that for all $k\gg 0$,
	$$\int\limits_{\widehat B_k}\ccol_\k(\sigma)d\mu=P_\k(k).$$
\end{theorem}

This, and (\ref{Eqn: braid philosophy}) motivate us to form a conjecture about the growth of $\textnormal{Avg}_\k(X)$ as $X$ tends to infinity. In simplified form:

\begin{conjecture}[\cref{Def: gen sum type} + \cref{Conj: Main conjecture}]
\label{Conj: main conj in intro}
Let $\k\in\Kei$ be finite. Then there exists $a_\k>0$ such that
$$\textnormal{Avg}_\k(X)= a_\k \log^{\deg P_\k}(X)\cdot (1+o(1)) $$
as $X\to\infty$.
\end{conjecture}

The second goal of this paper is to verify the conjecture for several examples of finite keis. In particular we verify the conjecture for all keis $\k$ of size $|\k|=3$. One can easily show that every such kei is isomorphic to one of three keis, encoded in the following table by
$$\varphi\colon \{1,2,3\}\to S_3\;\;,\;\;\;\;\varphi_i(j)=\qq ij.$$

$$\begin{tabular}{|c|c|c|c|}
	\hline
	$\k$	&$\varphi_1$	&$\varphi_2$&	$\varphi_3$\\
	\hline
	\hline
	$\triv_3$&	$\Id$&		$\Id$&		$\Id$\\
	\hline
	$\joyce$	&	$\Id$&		$\Id$&		$(1,2)$\\
	\hline
	$\RGB$		&	$(2,3)$&	$(1,3)$&	$(1,2)$\\
	\hline
\end{tabular}$$

We note that one recovers Fox's $3$-colorings mentioned above as $\RGB$-colorings of links. We begin by identifying the function $\col_\k$ for each $\k\in\{\triv_3,\joyce,\RGB\}$:
\begin{proposition}
For $1<n\in\NN$ squarefree we have:
\begin{enumerate}
	\item \Cref{Prop: Trivial colorings factor through pi_0}: $$\col_{\triv_3}(n)=3^{\omega(n)},$$
	where $\omega(n)$ is the number of prime divisors of $n$.
	\item \Cref{Prop: arithmetic computation: Joyce colorings}:
	$$\col_{\joyce}(n)=\sum_{abc=n}\Jac ab.$$
	\item \Cref{Prop: computation of RGB for n>1}:
	$$\col_{\RGB}(n)=3\left|Cl\left({\L_n}\right) \otimes_\ZZ\FF_3\right|.$$
\end{enumerate}	
\end{proposition}

The polynomials
$$P_{\triv_3}(x)=\frac{x^2+3x+2}2\;\;,\;\;\;\;
P_{\joyce}(x)=2x+1\;\;,\;\;\;\;
P_{\RGB}(x)=6$$
are computed in \cite[\textsection 7]{WEBSITE:DavisSchlankHilbert2023}. Using methods of analytic number theory, we prove \cref{Conj: main conj in intro} for these three cases, and for additional related keis - see propositions
\ref{Prop: Main Conj for T_a},
\ref{Prop: Main Conj for J_+} and
\ref{Prop: Main Conj for RGB}.

\subsection{Structure of the Paper}
In \cref{Section: profinite keis} we lay the groundwork for defining $\kei_n$. Here you'll find definitions and statements for the algebraic theory of keis, some old, some new. Then in \cref{Section: fundamental arithmetic kei} we define $\kei_n$ itself, and $\k$-colorings of $n$ for finite kei $\k$.
In \cref{Section: main conjecture} we formulate the main conjecture of this paper. Subsequent sections are devoted to proving the conjecture in for several examples of keis: \cref{Section: preliminary computations} contains computations of some preliminary estimations we'll need for the actual results. Sections \ref{Section: trivial kei proof} through \ref{Section: dihedral 3 kei proof} contain concrete examples:
\begin{itemize}
	\item \cref{Section: trivial kei proof} - for trivial keis $\triv_\powa$
	\item \cref{Section: Joyce kei proof} - for a kei $\joyce$ introduced in \cite{ARTICLE:Joyce1982} and variants.
	\item \cref{Section: dihedral 3 kei proof} - for the dihedral kei $\RGB$.
\end{itemize}

\subsection{Relation to Other Works}
Our work is grounded in arithmetic statistics, an active field of research with a long, rich history history. For example, our proof of the main conjecture in the case of $\k=\RGB$ relies heavily on the Davenport-Heilbronn theorem \cite{ARTICLE:DavenportHeilbronn1971} and further refinements
\cite{WEBSITE:BhargavaShankTsim2012} on counting cubic fields.
The Cohen-Lenstra-Martinet heuristics
\cite{ARTICLE:CohenLenstra1983} on the distribution of class groups is very much an ongoing endeavor \cite{WEBSITE:SawinWood2023}.
In the case where one replaces $\Spec(\ZZ)$ with $\Spec(\FF_p[t])$, the relation with \'etale topology has been studied in works such as \cite{WEBSITE:ElVenWest2015} and \cite{ARTICLE:FarbWolfson2018}.
Others have considered keis and related structures in a number-theoretical setting, with different motivation (\cite{WEBSITE:Takahashi2017}). The notions obtained there are different, and unrelated to arithmetic statistics as far as we know.

\subsection{Acknowledgements}
We thank Ohad Feldheim, Ofir Gorodetski, Aaron Landesman, Lior Yanovski, David Yetter and Tamar Ziegler for useful discussion. TMS was supported by the US-Israel Binational Science Foundation under grant 2018389. TMS was supported by ISF1588/18 and the ERC under the European Union’s Horizon 2020 research and innovation program (grant agreement No. 101125896).

\section{Keis and Profinite Keis}
\label{Section: profinite keis}
\subsection{Keis and Augmented Keis}
\begin{definition}
	\label{Def: quandle, Qdl}
	A \emph{kei} is a set $\k$ with a binary operator
	$$\k\times \k\to \k \;\;,\;\;\;\;(x,y)\mapsto \qq xy$$
	satisfying the following three axioms:
	\begin{enumerate}
		\item[(K1)] $\forall x\in \k$, $\qq xx=x$.
		\item[(K2)] $\forall x,y\in \k$, $\qq{x}{(\qq xy)}=y$.
		\item[(K3)] $\forall x,y,z\in \k$, $\qq{x}{(\qq yz)}=\qq{(\qq xy)}{(\qq xz)}$.
	\end{enumerate}
	A \emph{morphism} $f\colon \k\to \kk$
	 of keis is a map satisfying
	$$\forall x,y\in \k\;\;,\;\;\;\;f(\qq xy)=\qq{f(x)}{f(y)}.$$
	Together with morphisms, keis form a category, denoted by $\Kei$.
\end{definition}

\begin{example}
Let $G\in\Grp$ be a group. Then the set $S=\{g\in G\;|\;g^2=1\}\subseteq G$ of involutions in $G$ is a kei with structure given by conjugation:
	$$\qq gh=ghg^{-1}.$$
\end{example}

\begin{definition}
\label{Dfn: Trivial T_a}
	A kei $\triv\in\Kei$ is \emph{trivial} if $\qq xy=y$ for all $x,y\in \triv$.
	For $\powa\in\NN$, we denote by $\triv_\powa$ the trivial kei with  $|\triv_\powa|=\powa$. This is well-defined up to isomorphism.
\end{definition}

\begin{definition}
	\label{Def: augmented quandle}
	An \emph{augmented kei} $(\k,G,\alpha)$ consists of
	a set $\k\in\Set$, a group $G\in\Grp$ acting on $\k$ from the left,
	and an \emph{augmentation map} - a function
	$$\alpha\colon \k\to G\;,\;\;x\mapsto\alpha_x$$
	satisfying:
	\begin{enumerate}
		\item For all $x\in \k$, $\alpha_x(x)=x$.
		\item For all $x\in\k$, $\alpha_x^2=1_G.$
		\item For all $g\in G$ and $x\in \k$, $\alpha_{g(x)}=g\alpha_xg^{-1}\in G$.
	\end{enumerate}
	A morphism of augmented keis $(\k,G,\alpha)\to (\k', G',\alpha')$ consists of a function $f\colon \k\to \k'$ and group homomorphism $\widetilde f\colon G\to G'$ satisfying
	\begin{enumerate}
		\item For every $x\in \k$, $\alpha'_{f(x)}=\widetilde f(\alpha_x)\in G'$.
		\item For every $x\in \k$ and $g\in G$, $f(g(x))=\widetilde f(g)(f(x))\in\k'$.
	\end{enumerate}
	In other words, the following diagram commutes:
	\begin{equation}\label{Eqn: AKei morphism}
	\begin{tikzcd}
	G\times \k\ar[r]\ar[d,"(\widetilde f{,}f)"'] &\k\ar[r,"\alpha"]\ar[d,"f"']&G\ar[d,"\widetilde f"]\\
	G'\times \k'\ar[r]&\k'\ar[r,"\alpha'"']&G'
	\end{tikzcd}\;\;. \end{equation}
	Together with morphisms, augmented keis form a category $\AKei$.
\end{definition}

\begin{example}
	Let $G\in\Grp$ group and let $X\subseteq G$ be a union of conjugacy classes in $G$ s.t. $x^2=1_G$ for all $x\in X$. Let $\iota\colon X\hookrightarrow G$ denote the inclusion. Then $(X,G,\iota)$ is an augmented kei, where for all $x\in X$ and $g\in G$,
	$$g(x)=gxg^{-1}.$$
\end{example}

As the name suggests, for $(\k,G,\alpha)\in\AKei$, the set $\k$ has a canonical kei structure. We omit the proof of the following easily verifiable proposition:

\begin{proposition}
\label{Prp: AKei to Kei functorial}
For every $(\k,G,\alpha)\in\AKei$, the operator
	$$\k\times \k\to \k\;\;\;,\;\;\;\;\;\;(x,y)\mapsto \alpha_x(y)$$
	defines a kei structure on $\k$. This extends to a functor
$$\KeiFunc\bullet\colon\AKei\to\Kei\;\;\;,\;\;\;\;\;\;\a=(\k,G,\alpha)\mapsto \k.$$

\end{proposition}

\begin{definition}
	\label{Def: Inn}
	Let $\k\in\Kei$ and let $x\in \k$. We denote by 
	$$\varphi_x=\varphi_{\k,x}\colon \k\to \k\;\;:\;\;\;\; \forall y\in\k\;,\;\varphi_x(y)=\qq xy $$
	By (K2), $\varphi_x$ is a bijection and by (K3), $\varphi_x\in\Aut_\Kei(\k)$. By $\Inn(\k)$ we denote the subgroup
	$$\Inn(\k)=\langle\{\varphi_x\;|\;x\in \k\}\rangle\le \Aut_\Kei(\k),$$
	and by $\varphi_\k\colon \k\to\Inn(\k)$ the map
	$$\varphi_{\k}\colon x\longmapsto \varphi_{\k,x}.$$
\end{definition}

The following notion of \emph{conciseness} is not new - see \cite[Thm 10.2]{ARTICLE:Joyce1982} for one of several examples of prior usage. We rely heavily on this property, and therefore see fit to give it a name:
\begin{definition}
\label{Dfn: concise AKei}
	An augmented kei $\a=(\k,G,\alpha)\in\AKei$ is \emph{concise} if the image of $\alpha$ generates $G$, that is,
	$$\langle \{\alpha_x\;|\;x\in \k\}\rangle=G.$$
\end{definition}

\begin{proposition}
	Let $\k\in\Kei$. Then $(\k,\Inn(\k),\varphi_\k)$ is a concise augmented kei.
\end{proposition}
\begin{proof}
	Let $x\in\k$. Then
	$$\varphi_x(x)=\qq xx=x,$$
	and for all $y\in \k$,
	$$\varphi_x^2(y)=\qq x{(\qq xy)}=y.$$
	Therefore $\varphi_{x}^2=\Id_\k$. For $g=\varphi_y\in\Inn(\k)$ we have for all $z\in\k$,
	$$(\varphi_{g(x)}\circ g)(z)=(\varphi_{\varphi_y(x)}\circ\varphi_y)(z)=\qq{(\qq yx)}{(\qq yz)}=\qq{y}{(\qq xz)}=(g\circ \varphi_{x})(z).$$
	That is,
	$$\varphi_{g(x)}=g\circ\varphi_{x}\circ g^{-1}.$$
	This equality holds for all $g\in\langle\{\varphi_y\;|\;y\in\k\}\rangle=\Inn(\k)$. Hence $(\k,\Inn(\k),\varphi_\k)$ is a concise augmented kei. 
\end{proof}

\begin{definition}
	Let $\k\in\Kei$. By $\a_\k\in\AKei$ we denote the augmented kei
	$$\a_\k:=(\k,\Inn(\k),\varphi_\k).$$
\end{definition}

\begin{proposition}
\label{Prop: concise augkei lift surj kei morphism}
	Let $\a=(\k,G,\alpha)\in\AKei$ be concise, and let $f\colon\k\twoheadrightarrow\k'$ be a surjective morphism in $\Kei$. Then there is a unique $\rho\colon G\to\Inn(\k')$ s.t. 
	$$(f,\rho)\colon(\k,G,\alpha)\to
	\a_{\k'}
	$$
	is a morphism in $\AKei$. Furthermore, $\rho$ is surjective.
\end{proposition}
\begin{proof}
	The desired $\rho\colon G\to\Inn(\k')$ must satisfy, for all $x\in\k$ and $g\in G$:
	$$(i):\;\rho(\alpha_x)=\varphi_{f(x)}\;\;\;\;\textrm{and}\;\;\;\;(ii):\;\rho(g)(f(x))=f(g(x)).$$
	Let $F_\k\in\Grp$ denote the free group on the set $\k$. The maps $\alpha\colon\k\to G$ and
	$$\varphi_f\colon \k\to\Inn(\k')\;\;,\;\;\;\;x\mapsto\varphi_{f(x)}.$$
	extend to respective group homomorphisms
	$$\widetilde\alpha\colon F_\k\to G\;\;,\;\;\;\;\widetilde{\varphi_f}\colon F_\k\to \Inn(\k').$$
	Condition $(i)$ is equivalent to the commuting of the diagram
	$$\begin{tikzcd}[row sep=small]
	&&G\ar[dr,dashed,"\rho"]\\
	F_\k\ar[urr,"\widetilde\alpha"]\ar[rrr,"\widetilde{\varphi_f}"']&&&\Inn(\k')
	\end{tikzcd}.$$
	Since $\a$ is concise, $\widetilde \alpha$ is surjective, then $\rho$ - should it exist - is unique.
	Let $x_1,\dots,x_n\in\k$ and let $\varepsilon_1\dots,\varepsilon_n\in\{\pm 1\}$. Then for all $y\in\k$,
	$$\widetilde{\varphi_f} \left(x_1^{\varepsilon_1}\cdots x_n^{\varepsilon_n}\right)(f(y))=
	(\varphi_{f(x_1)}\circ\cdots\circ\varphi_{f(x_n)})(f(y))= $$
	$$=\qq{f(x_1)}{(\dots{\qq{f(x_n)}{f(y)}}\dots)}=f\left(\qq{x_1}{\dots(\qq{x_n}{y})\dots }\right)=f(\alpha_{x_1}^{\varepsilon_1}\cdots \alpha_{x_n}^{\varepsilon_n} (y)).$$
	If $x_1,\dots ,x_n\in\k$ are s.t. $\widetilde\alpha\left(x_1^{\varepsilon_1}\cdots x_n^{\varepsilon_n}\right)= \alpha_{x_1}^{\varepsilon_1}\cdots \alpha_{x_n}^{\varepsilon_n}=1$, then for all $y\in\k$,
	$$\widetilde{\varphi_f} \left(x_1^{\varepsilon_1}\cdots x_n^{\varepsilon_n}\right)(f(y))=f(\alpha_{x_1}^{\varepsilon_1}\cdots \alpha_{x_n}^{\varepsilon_n} (y))=f(y).$$
	Since $f$ is surjective, we have
	$$\widetilde\alpha\left(x_1^{\varepsilon_1}\cdots x_n^{\varepsilon_n}\right)=1\;\;\Longrightarrow\;\;\widetilde{\varphi_f} \left(x_1^{\varepsilon_1}\cdots x_n^{\varepsilon_n}\right) =\Id_{\k'}.$$
	Hence there exists $\rho\colon G\to \Inn(\k')$ satisfying condition $(i)$:
	$$\rho\circ\widetilde\alpha=\widetilde{\varphi_f}.$$ Furthermore, since $\a$ is concise, for all $g= \alpha_{x_1}^{\varepsilon_1}\cdots \alpha_{x_n}^{\varepsilon_n} \in G$  we have
	$$\rho(g)\circ f=
	\widetilde{\varphi_f} \left(x_1^{\varepsilon_1}\cdots x_n^{\varepsilon_n}\right) \circ f =f\circ \alpha_{x_1}^{\varepsilon_1}\cdots \alpha_{x_n}^{\varepsilon_n}=f\circ g.$$
Therefore $\rho$ satsifies condition $(ii)$ as well, hence
	$$(f,\rho)\colon (\k,G,\alpha)\to(\k',\Inn(\k'),\varphi_{\k'})$$
	is a morphism in $\AKei$. Also, $\rho$ is surjective:
	$$\rho(G)=\rho(\langle\{\alpha_x\;|\; x\in\k\}\rangle)=\langle\{\varphi_{x'}\;|\; x'\in\k'\}\rangle=\Inn(\k').$$
\end{proof}

\begin{definition}
	Let $\a=(\k,G,\alpha)\in\AKei$ be concise and let $f\colon\k\to\k'$ be a morphism in $\Kei$. For $\k'':=f(\k)\le \k'$ We denote by
	$$\rho_{\a,f}\colon G\twoheadrightarrow\Inn(\k'')$$
	the homomorphism $\rho$ from \cref{Prop: concise augkei lift surj kei morphism}. If $\a=\AKeiFunc\k$ or if $f=\Id_\k$, we respectively denote by:
	$$\rho_f:=\rho_{\AKeiFunc\k,f}\colon \Inn(\k)\twoheadrightarrow\Inn(\k'')\;\;,\;\;\;\;\rho_\a:=\rho_{\a,\Id_\k}\colon G \twoheadrightarrow\Inn(\k).$$
\end{definition}

The following proposition is a straightforward consequence of \cref{Prop: concise augkei lift surj kei morphism}. The proof, a routine check, is omitted.

\begin{proposition}
\label{Prp: kei Inn surj functorial to AugKei}
	The map
	$$\AKeiFunc\bullet\colon\k\mapsto\AKeiFunc\k		= (\k,\Inn(\k),\varphi_\k)\in\AKei$$
	is functorial in surjective morphisms in $\Kei$, mapping $f \colon \k\twoheadrightarrow\k'$ to
	$$ (f,\rho_f)\in\Hom_{\AKei}(\a_\k,\a_{\k'}).$$
	Accordingly, the map $\k\mapsto \Inn(\k)$ that sends $f\colon\k\twoheadrightarrow\k'$ to
	$$\rho_{f}\in\Hom_\Grp(\Inn(\k),\Inn(\k'))$$
	is functorial in surjective morphisms.
\end{proposition}

\begin{proposition}
\label{Prop: AKei has images}
	Let $\a=(\k,G,\alpha),\a'=(\k',G',\alpha')\in\AKei$ and let $(f,\widetilde f)\colon \a\to\a'$ be a morphism. Then
	$$(f(\k),\widetilde f(G),\alpha'_{|f(\k)})\in\AKei$$
	is a sub-augmented kei of $\a'$.
\end{proposition}
\begin{proof}
	Let $y=f(x)\in f(\k)$ and let $h=\widetilde f(g)\in \widetilde f(G)$. Then
	$$\alpha'_y=\alpha'_{f(x)}=\widetilde f(\alpha_x)\in \widetilde f(G)$$
	and
	$$h(y)=\widetilde f(g)(f(x))=f(g(x))\in f(\k).$$
	Therefore $(f(\k),\widetilde f(G),\alpha'_{|f(\k)})$ is a sub-augmented kei of $\a'$.
\end{proof}

\begin{definition}
\label{Def: image and surjections in AKei}
Let $\a=(\k,G,\alpha),\a'=(\k',G',\alpha')\in\AKei$ and let $(f,\widetilde f)\colon \a\to\a'$ be a morphism. We define the \emph{image} of $(f,\widetilde f)$ to be the augmented kei
$$\image(f,\widetilde f):= (f(\k),\widetilde f(G),\alpha'_{|f(\k)}) \le \a'.$$
We say $(f,\widetilde f)$ is \emph{surjective} if
$$\image(f,\widetilde f)=\a',$$
that is, if $f\colon \k\to\k'$ and $\widetilde f\colon G\to G'$ are both surjective.
\end{definition}

\begin{lemma}
\label{Lma: conciseness carries by surjections}
	Let $(f,\widetilde f)\colon\a\twoheadrightarrow \a'$ be a surjective morphism in $\AKei$. Assume $\a$ is concise. Then $\a'$ is also concise.
\end{lemma}
\begin{proof}
Write $\a=(\k,G,\alpha)$ and $\a'=(\k',G',\alpha')$.
We have the following equality of subgroups of $G'$:
	$$\langle \{\alpha'_y\}_{y\in\k'} \rangle= \langle \{\alpha'_{f(x)}\}_{x\in\k} \rangle= \langle \{\widetilde f(\alpha_x)\}_{x\in\k} \rangle= \widetilde f\left(\langle \{\alpha_x\}_{x\in\k} \rangle\right)=\widetilde f(G) =G'.$$
\end{proof}

\begin{corollary}
\label{Cor: concise surj grp natural}
Let $\a=(\k,G,\alpha),\a'=(\k,G,\alpha)\in\AKei$ be concise and let $(f,\widetilde f)\colon \a\to\a'$ be a surjective morphism in $\AKei$. Then
	$$\rho_f\circ\rho_\a=\rho_{\a'}\circ \widetilde f\;\;\;\;:\;\;\;\;\;\;\;\;
	\begin{tikzcd}
	G\ar[r,"\rho_{\a}"]\ar[d,"\widetilde f"']\ar[dr,phantom,"//"]&\Inn(\k)\ar[d,"\rho_f"]\\
	G'\ar[r,"\rho_{\a'}"']&\Inn(\k')
	\end{tikzcd}$$
\end{corollary}
\begin{proof}
	The statements of \cref{Prop: concise augkei lift surj kei morphism} and \cref{Prp: kei Inn surj functorial to AugKei} are interpreted in terms of an adjunction
	$$\AKeiFunc\bullet\dashv\KeiFunc\bullet\;\;,\;\;\;\;
	\begin{tikzcd}
	\AKeiFunc\bullet\colon\Kei^\surj\ar[r,shift left]&\AKei_\concise^\surj\colon\KeiFunc\bullet\ar[l,shift left]
	\end{tikzcd},$$	
	where $\Kei^\surj\subseteq\Kei$ is the subcategory of all keis and surjective morphisms, $\AKei_\concise^\surj\subseteq\AKei$ is the subcategory of concise augmented keis and surjective morphisms.
	The desired square is obtained from the unit of the adjunction:
	$$
	\begin{tikzcd}[column sep=huge]
	\a\ar[r,"(\Id_\k{,}\rho_{\a})"]\ar[d,"(f{,}\widetilde f)"']\ar[dr,phantom,"//"]&(\k,\Inn(\k),\varphi_\k)\ar[d,"(f{,}\rho_f)"]\\
	\a'\ar[r,"(\Id_{\k'}{,}\rho_{\a'})"']&(\k',\Inn(\k'),\varphi_{\k'})
	\end{tikzcd}$$
\end{proof}

\begin{proposition}
	\label{Prop: augmented quotient by normal N}
	Let $\a=(\k,G,\alpha)\in\AKei$ and let $H\unlhd G$. Let
	$$\alpha_H\colon H\backslash \k\to H\backslash G\;\;,\;\;\;\;\alpha_H\colon Hx \mapsto H\alpha_x,$$
	where $H\backslash \k$ is the quotient in $\Set$. 
	Then $(H\backslash \k, H\backslash G,\alpha_H)$ is an augmented kei and the quotient map $(\k,G,\alpha)\to (H\backslash \k, H\backslash G,\alpha_H)$ is a morphism in $\AKei$. In particular $H\backslash \k$ inherits a well-defined kei structure from $\k$, and the quotient map $\k\to H\backslash \k$ is a morphism in $\Kei$.
\end{proposition}
\begin{proof}
	It suffices to show that $\alpha_H$ is a well-defined function and that the $H\backslash G$-action on $H\backslash \k$ is well-defined. Let $x\in\k$, let $g\in G$ and let $h\in H$. Then
	$$\alpha_{h(x)}=h\alpha_xh^{-1}=h(\alpha_xh\alpha_x^{-1})\alpha_x\in H\alpha_x,$$
	$$(hg)(x)=h(g(x))\in Hg(x)\;\;,\;\;\;\;g(h(x))=(ghg^{-1})(g(x))\in Hg(x).$$
	The aumgented kei axioms on $(H\backslash \k, H\backslash G,\alpha_H)$ are seen to hold by routine lifting to $\a$. 
Consequently $H\backslash \k$ is a kei with structure induced from $\k$:
$$\qq{\overline x}{\overline y}=\overline{\qq xy},$$
that is, the quotient map $\k\to H\backslash \k$ is a morphism in $\Kei$.
\end{proof}

\begin{definition}
	Let $\a=(\k,G,\alpha)\in\AKei$ and let $H\unlhd G$. We denote by $H\backslash\a$ we denote the augmented kei
	$$H\backslash\a:=(H\backslash \k,H\backslash G,\alpha_H)\in\AKei,$$
\end{definition}

\begin{proposition}
\label{Prop: factor augkei morphism through quotient}
	Let $\a=(\k,G,\alpha),\a'=(\k',G',\alpha')\in\AKei$ and let $(f,\widetilde f)\colon\a\to\a'$. Denote by $H_{\widetilde f}:=\ker(\widetilde f\colon G\to G')\unlhd G$. Then $(f,\widetilde f)$ factors through the quotient 
	$$H_{\widetilde f}\backslash \a\to\a'$$
\end{proposition}
\begin{proof}
	By \cref{Prop: augmented quotient by normal N}, $H_{\widetilde f}\backslash \a \in\AKei$ and $\a\to H_{\widetilde f}\backslash \a$ is a morphism in $\AKei$. 
	Let $x\in\k$, let $g\in G$ and let $h\in H_{\widetilde f}$.
	Then $\widetilde f(h)=1$, hence
	$$\widetilde f(hg)= \widetilde f(h) \widetilde f(g)= \widetilde f(g),$$
	$$f(h(x))=\widetilde f(h)(f(x))=f(x).$$
	This suffices to show that $(f,\widetilde f)$ factors through $H_{\widetilde f}\backslash \a $.
\end{proof}

\begin{proposition}
\label{Prop: factor from concise through quotient}
	Let $\a=(\k,G,\alpha)\in\AKei$ be concise, let $\k'\in\Kei$, and let $f\colon\k\to\k'$ be a morphism in $\Kei$. Let $\k''\le \k'$ denote the image $\k'':=f(\k)$ of $f$. Then $f$ factors through
	$$H_f\backslash \k\to \k',$$
	where
	$$H_f:=\ker(\rho_{\a,f}\colon G\twoheadrightarrow\Inn(\k''))\unlhd G.$$
\end{proposition}
\begin{proof}
By \cref{Prop: concise augkei lift surj kei morphism}, $f\colon \k\twoheadrightarrow \k''$ lifts to a surjective morphism
	$$(f,\rho_{\a,f})\colon \a\twoheadrightarrow \a_{\k''}$$
	in $\AKei$. By \cref{Prop: factor augkei morphism through quotient}, $(f,\rho_{\a,f})$ factors through
	$$H_{f}\backslash \a\to \a_{\k''}.$$
	Forgetting to $\Kei$, we find that $f\colon\k\to \k'$ factors through
	$$H_f\backslash \k\twoheadrightarrow \k''\hookrightarrow\k'.$$
\end{proof}

\subsection{Profinite Keis and Augmented Keis}
In this section we discuss profinite keis, which are formal cofiltered limits of finite keis. 
Closely related notions were extensively developed in \cite{WEBSITE:ByardCaiETAL2024}. Stone duality is the statement that profinite sets are equivalent to Stone spaces - compact, Hausdorff totally-disconnected topological spaces. More precisely, the fully-faithful embedding
$$\delta\colon \Set^\fin\hookrightarrow \Top$$
of finite sets as discrete topological spaces lifts to a fully faithful functor
$$\widehat\delta\colon \Pro\Set^\fin\hookrightarrow \Top.$$
Similarly, the embedding of finite groups into topological groups extends to a fully-faithful functor
$$\Pro\Grp^\fin\hookrightarrow \Grp^\Top.$$
Hence for $\G,\G'\in\Pro\Grp^\fin$ we denote by
$$\Hom_\Grp^\cont(\G,\G'):= \Hom_{\Pro\Grp^\fin}(\G,\G').$$
Like groups, the algberaic theory of keis consists of finitely-many operations of finite arity. An analogous statement for keis  holds as well:  the induced functor
$$\Pro\Kei^\fin\hookrightarrow\Kei^\Top$$
is fully-faithful (\cite[\textsection VI.2]{BOOK:Johnstone1982}). In considering topological augmented keis, it is not difficult to see that a similar argument works: the enduced functor
$$\Pro\AKei^\fin\hookrightarrow\AKei^\Top$$
is fully-faithful. We therefore make no distinction between profinite (augmented) keis and their corresponding topological (augmented) keis.
For $\kei,\kei'\in\Pro\Kei^\fin$ and $\A,\A'\in\Pro\AKei^\fin$ we denote by
$$\Hom_\Kei^\cont(\kei,\kei'):= \Hom_{\Pro\Kei^\fin}(\kei,\kei'),$$
$$\Hom_\AKei^\cont(\A,\A'):= \Hom_{\Pro\AKei^\fin}(\A,\A').$$

It is perhaps worth mentioning that the forgetful functors
$$\begin{tikzcd}[row sep=tiny]
	\Pro\Grp^\fin &\Pro\AKei^\fin\ar[l]\ar[r]& \Pro\Kei^\fin\\
	\G& (\kei,\G,\alpha)\ar[l,mapsto]\ar[r,mapsto]& \kei
\end{tikzcd}$$
are defined once as naturally extending the finite counterparts
$$\begin{tikzcd}[row sep=tiny]
	\Grp^\fin &\AKei^\fin\ar[l]\ar[r]& \Kei^\fin
\end{tikzcd}\;\;,$$
and again as restricted from the topological counterparts
$$\begin{tikzcd}[row sep=tiny]
	\Grp^\Top &\AKei^\Top\ar[l]\ar[r]& \Kei^\Top
\end{tikzcd}\;\;.$$
These two definitions in fact coincide. Also worth mentioning is that because $\Kei$ and $\AKei$ have a well-behaved notion of images (see \cref{Prop: AKei has images} for $\AKei$), profinite keis and augmented keis can be presented as limits of cofiltered diagrams with \emph{surjective} maps. This can be done canonically: For $\kei\in\Pro\Kei^\fin$, Let $\D_\kei$ denote the category whose objects are
$$\{(\k,f)\;\;|\;\;\;\; \k\in\Kei^\fin\;,\;\; f\colon\kei\twoheadrightarrow \k \},$$
and
$$\Hom((\k,f),(\k',f'))=\{h\colon \k\twoheadrightarrow \k'\;|\;h\circ f=f'\}.$$
Then $\D_\kei$ is cofiltered, and 
$$\kei\simeq \lim_{\substack{\longleftarrow\\ \D_\kei}}\k.$$
We set a convention: When we write
\begin{equation}
\label{Eqn: profinite kei convention}
\displaystyle{\kei=\lim_{\substack{\longleftarrow\\\D}}\k_d\in\Pro\Kei^\fin},
\end{equation}
it is to be understood that $\D$ is a cofiltered diagram in $\Kei^\fin$, and that the morphisms $\k_d\to\k_{d'}$ in the diagram are all surjections. Likewise for
$$\displaystyle{\A=\lim_{\substack{\longleftarrow\\\D}}\a_d\in\Pro\Kei^\fin} .$$

\begin{definition}
	Let $\A=(\kei,\G,\alpha)\in\Pro\AKei^\fin$ be a profinite augmented kei. We say $\A$ is \emph{concise} if $\{\alpha_x\}_{x\in\kei}$ generates $\G$ topologically:
	$$\overline{\langle\{\alpha_x\;|\;x\in\kei\}\rangle}=\G.$$
	For finite $\A$, this notion of concisesness coincides with \cref{Dfn: concise AKei}.
\end{definition}

\begin{lemma}
\label{Lma: conciseness carries by surjections profinite}
	Let $\A\twoheadrightarrow \A'$ be a surjective morphism in $\Pro\AKei^\fin$. Assume $\A$ is concise. Then so is $\A'$.
\end{lemma}
\begin{proof}
	Denote $\A= (\kei,\G,\alpha)$ and $\A'=(\kei',\G',\alpha')$.
	Because $\A$ is concise, $\overline{\langle\{\alpha_x\}_{x\in\kei}\rangle}=\G$. Therefore
	$$\overline{\langle\{\alpha'_y\}_{y\in\kei'}\rangle}= \overline{\langle\{\alpha'_{f(x)}\}_{x\in\kei}\rangle}= \overline{\langle\{\widetilde{f}(\alpha_x)\}_{x\in\kei}\rangle} =$$
	$$=\widetilde{f}\left(\overline{\langle\{(\alpha_x)\}_{x\in\kei}\rangle}\right)=\widetilde f(\G)=\G'.$$
	Hence $\A'$ is concise.
\end{proof}
Suppose $\displaystyle{\A=\lim_{\substack{\longleftarrow\\\D}}\a_d\in\Pro\Kei^\fin} $ is concise. Then the projections $\A\to \a_d$ are all surjective. By \cref{Lma: conciseness carries by surjections profinite}, each $\a_d\in\AKei^\fin$ is also concise.

\begin{proposition}
\label{Prop: pro-fin concise augkei lift surj kei morphism}
	Let $\A=(\kei,\G,\alpha)\in\Pro\Kei^\fin$ be concise. Let $\k\in\Kei$ be finite and let $f\colon \kei\twoheadrightarrow \k$ be a surjective morphism in $\Pro\Kei^\fin$.
	Then there is a unique morphism $\rho\colon\G\to\Inn(\k)$ s.t.
	$$(f,\rho)\colon \A\to \a_\k=(\k,\Inn(\k),\varphi_\k)$$
	is a morphism in $\Pro\AKei^\fin$.
\end{proposition}
\begin{proof}
	Write $\A$ as the limit
	$$\A =\lim\limits_{\substack{\longleftarrow\\d\in\D}}\a_d\;\;,\;\;\;\;\a_d=(\k_d,G_d,\alpha_d)\in\AKei^ \fin$$
of a cofiltered diagram in $\AKei^\fin$ with surjections, so that $$\kei= \lim\limits_{\substack{\longleftarrow\\d\in\D}}\k_d.$$
Therefore $f\colon\kei\twoheadrightarrow \k$ factors through $f_d\colon\k_d\twoheadrightarrow \k$ for some $d\in \D$.
The morphism $\A\to\a_d$ is surjective, therefore $\a_d$ is concise by \cref{Lma: conciseness carries by surjections profinite}. \Cref{Prop: concise augkei lift surj kei morphism} 
implies the existence of a canonical morphism
	$$(f_d,\rho_d)\colon \a_d\to \a_\k$$
	in $\AKei^\fin$. The composition
	$$(f,\rho)\colon\A\to\a_d\xrightarrow{(f_d,\rho_d)}\a_\k$$
	is a morphism in $\Pro\AKei^\fin$ lifting $f$.
	For all $d\in \D$ there is at most one such map $(f_d,\rho_d)$. Moreover the diagram $\D$ is cofiltered. Therefore $(f,\rho)$, independent of $d$, is uniquely defined by $f$.
\end{proof}

The following proposition lists versions of propositions from the previous section for profinite augmented keis. The proofs are fundamentally the same, and are therefore omitted.

\begin{proposition}
\label{Prop: profinite augkei statements}
${}$
\begin{enumerate}
	\item[1. (see \ref{Prop: augmented quotient by normal N}):] Let $\A=(\kei,\G,\alpha)\in\Pro\AKei^\fin$ and let $N\unlhd \G$ be closed. Then
	$$N\backslash\A:=(N\backslash\kei,N\backslash \G,\alpha_N)\in \Pro\AKei^ \fin,$$
	such that the quotient map $\A\to N\backslash \A$ is a morphism in $\Pro\AKei^\fin$.
	\item[2. (see \ref{Prop: factor augkei morphism through quotient}):] Let $(f,\widetilde f)\colon \A\to\A'$ be a morphism in $\Pro\AKei^\fin$. Denote by $N_{\widetilde f}:=\ker(\widetilde f)$. Then $(f,\widetilde f)$ factors through the quotient
	$$N_{\widetilde f}\backslash \A\to \A'.$$
	\item[3. (see \ref{Prop: factor from concise through quotient}):] Let $\A=(\kei,\G,\alpha)\in\Pro\AKei^\fin$ be concise, let $\k\in \Kei^\fin$ be finite, and let $f\colon\kei\to\k$ be a morphism in $\Pro\Kei^\fin$. Then $f$ factors through $N_{f}\backslash\kei$, where
	$$N_f:=\ker\left(\rho_{\A,f}\colon \G\to\Inn(f(\kei))\right)\unlhd \G.$$
\end{enumerate}
\end{proposition}

\begin{definition}
	A profinite group $\G\in\Pro\Grp^\fin$ is \emph{small} if for all $d\in\NN$, $\G$ has finitely-many open subgroups of index $d$.
\end{definition}

\begin{proposition}
	Let $\G\in\Pro\Grp^\fin$. Then $\G$ is small iff for all  $G\in\Grp^\fin$,
	$$\left|\Hom_{\Grp}^\cont(\G,G)\right|<\infty.$$
\end{proposition}
\begin{proof}
	Let $H\le \G$ be an open subgroup of index $d$. Then 
	$$N_H:=\bigcap_{gH\in \G/H}gHg^{-1}\le \G$$
	is open and normal, satisfying
	$$\G/N_H\hookrightarrow \prod_{\G/H}\G/gHg^{-1}.$$
	Therefore every open subgroup $H$ of index $d$ contains an open normal subgroup $N_H$ of index $[\G:N_H]\le d^d$. It follows that $\G$ is small if and only if for all $d\in\NN$ there are finitely-many \emph{normal} open subgroups $N\unlhd\G$ of index $[\G:N]\le d$. This in turn occurs
	if and only if for all $G\in\Grp^\fin$,
	$$\left|\Hom_{\Grp}^\cont(\G,G)\right|<\infty.$$
\end{proof}

We therefore give the following analogous definition:

\begin{definition}
	Let $\kei\in\Pro\Kei^\fin$. We say $\kei$ is \emph{small} if for every $\k\in\Kei^\fin$,
	$$\left|\Hom_{\Kei}^\cont(\kei,\k) \right|<\infty.$$
\end{definition}

\begin{proposition}
\label{Prop: concise kei small condition}
	Let $\A=(\kei,\G,\alpha)\in\Pro\AKei^\fin$ be concise. If $\G$ is a small profinite group and if $\left|\G\backslash\kei\right|<\infty$, then $\kei\in\Pro\Kei^\fin$ is small.
\end{proposition}
\begin{proof}
	A finite kei $\k$ has finitely-many sub-keis, therefore $\kei$ is small iff for all $\k\in\Kei^\fin$ there are finitely-many \emph{surjective} morphisms $\kei\twoheadrightarrow\k$. Since $\A$ is concise, by \cref{Prop: pro-fin concise augkei lift surj kei morphism} any surjective morphism $f\colon \kei \twoheadrightarrow\k$ extends uniquely to a morphism
	$$(f,\rho_{\A,f})\colon\A\twoheadrightarrow\a_\k$$
	in $\Pro\AKei^\fin$. Denote by
	$$H_f:=\ker(\rho_{\A,f}\colon\G\twoheadrightarrow\Inn(\k))\unlhd \G,$$
	$$H_\k:=\bigcap_{f\colon\kei\twoheadrightarrow\k}H_f\unlhd\G.$$
	Since $\Inn(\k)$ is finite, and $\rho_{\A,f}$ continuous, then $H_f\unlhd\G$ is open. We also have
	$$\left|\Hom_{\Grp}^\cont(\G,\Inn(\k))\right|<\infty$$
	because $\G\in\Pro\Grp^\fin$ is small. Therefore $H_\k\unlhd \G$ is open. By \cref{Prop: profinite augkei statements}, every  $f\colon \kei\twoheadrightarrow\k$ factors through $H_{f}\backslash \kei$, and therefore through
	$$\overline f\colon H_\k\backslash \kei\to\k.$$
	The quotient  $H_\k\backslash\kei$ is finite because $\G\backslash\kei$ is finite, and the fibers of the map
	$$H_\k\backslash\kei\twoheadrightarrow\G\backslash\kei$$
	are transitive $H_\k\backslash \G$-sets - therefore finite.
	We conclude that for all $\k\in\Kei^\fin$,
	$$\Hom_{\Kei}^\cont(\kei,\k)^\surj\xleftarrow{\sim}\Hom_{\Kei^\fin}(H_\k\backslash \kei,\k)^\surj\in\Fin.$$
	Hence $\kei$ is small.
\end{proof}

\subsection{Disjoint Unions in $\Kei$}

In this section we shall discuss disjoint unions of keis. Much like the tensor product of non-commutative rings, this is not the coproduct in $\Kei$. It is however a symmetric monoidal structure on $\Kei$, with the empty kei playing the role of the unit. This discussion can be extended to profinite keis as well. The statements and their proofs are similar enough that we omit them in the profinite case.

\begin{proposition}
	Let $\k_1,\k_2\in\Kei$. Then the disjoint union $\k_1\sqcup\k_2$ is a kei with structure
	$$\qq xy=\;\;\;\;\;\;\begin{tabular}{|c||c|c|}
\hline
&$x\in\k_1$&$x\in\k_2$\\
\hline
\hline
$y\in\k_1$&$\varphi_{\k_1,x}(y)$&$y$\\
\hline
$y\in\k_2$&$y$&$\varphi_{\k_2,x}(y)$\\
\hline
\end{tabular}\;\;.$$
	This puts a symmetric monoidal structure $(\Kei,\sqcup,\emptyset)$ on $\Kei$.
\end{proposition}
\begin{proof}
	The first two kei axioms are easily verified. As for the third, let $x,y\in \k_1\sqcup \k_2$. If $x,y\in\k_i$ for some $i$, w.l.o.g. $i=1$, then
	$$\varphi_{x}\circ\varphi_y= (\varphi_{\k_1,x}\circ \varphi_{\k_1,y})\sqcup \Id_{\k_2}=$$
	$$=(\varphi_{\k_1,\qq xy}\circ \varphi_{\k_1,x})\sqcup \Id_{\k_2} =\varphi_{\qq xy}\circ \varphi_x.$$
	otherwise, w.l.o.g. $x\in\k_1$ and $y\in\k_2$, then
	$$\varphi_x\circ\varphi_y= \varphi_{\k_1,x}\sqcup\varphi_{\k_2,y} = \varphi_y\circ\varphi_x = \varphi_{\qq xy}\circ \varphi_x.$$
\end{proof}

Let $T\in\Set^\fin$. The disjoint union of $T$ copies of the terminal object $*\in\Kei$ recovers the trivial kei $\triv\in\Kei$ with underlying set $T$.
Hence there is a functor
$$\iota_T\colon \Kei^T\to\Kei_{/\triv}\;\;\;\;,\;\;\;\;\;\;\;\;(\k_t)_{t\in T}\longmapsto (\;\; \bigsqcup_{t\in T}\k_t\to \triv\;\;).$$
The functor $\iota_T$ preserves small limits and filtered colimits, and therefore has a left adjoint
$$\lambda_T\colon \Kei_{/\triv}\to \Kei^T.$$
The functor $\iota_T$ is also fully faithful, therefore $\lambda_T$ can be computed in terms of the unit
$$\kl\to(\iota_T\circ\lambda_T)(\kl).$$
We explicitly construct $\lambda_T(\kl)$ for $\kl\in\Kei$, making use of an augmentation on $\kl$:

\begin{proposition}
\label{Lma: kei factoring technical}
\label{Prp: Map concise to disjoint union practical}
	Let $\triv\in\Kei^\fin$ be trivial, let $(\kl,G,\alpha)\in\AKei$ be concise, and let $\eta\colon\kl \to \triv $ be a morphism in $\Kei$. For $t\in \triv$, we denote by $\kl_t:=\eta^{-1}(t)\subseteq \kl$ and by $H_t\le G$ the subgroup
	$H_t:=\langle\{\alpha_x\;|\;x\in\kl\setminus\kl_t\}\rangle\le G$.  Then
	\begin{enumerate}
		\item The set $\bigsqcup_T H_t\backslash \kl_t$ is well-defined, and a quotient of $\kl$ in $\Kei$.
		\item For all $(\k_t)_T\in\Kei^T$, let $\k_\triv\in\Kei$ denote the disjoint union $$\k_\triv:=\bigsqcup_{T}\k_t.$$
		Then there is a natural bijection
		$$\Hom_{\Kei_{/\triv}}(\kl,\k_\triv)\simeq \prod_{t\in T} \Hom_{\Kei}(H_t\backslash\kl_t,\k_t).$$
	\end{enumerate}
\end{proposition}

\begin{proof}
	Let $t \in\triv$. Then for all $y\in\kl_t$ and all $x\in\kl$,
	\begin{equation}
	\label{Eqn: eta eqn}
	\eta(\alpha_x(y))=\eta(\qq{x}{y})=\qq{\eta(x)}{\eta(y)}=\eta(y)=t.
	\end{equation}
	Since $(\kl,G,\alpha)$ is concise, it follows that the $\kl_t$ are $G$-invariant. For all $g\in G$ and $x\in \kl$,
	$$g\alpha_xg^{-1}=\alpha_{g(x)},$$
	where $g(x)\in \kl_{\eta(x)}$. Hence the subgroups
	$$H_t=\langle\{\alpha_x\;|\;x\in\kl\setminus\kl_{t}\}\rangle\;\;,\;\;\;\; H'_t:=\langle\{\alpha_x\;|\;x\in\kl_{t}\}\rangle\le G$$
	are normal. Next, let $x\in \kl$ and let $h\in H_{\eta(x)}$. Then
	$$\alpha_{h(x)}\alpha_x^{-1}=h\alpha_xh^{-1}\alpha_x^{-1}\in[H_{\eta(x)},H'_{\eta(x)}]\le H_{\eta(x)}\cap H'_{\eta(x)}.$$
	Let $t\in \triv$. If $t\neq \eta(x)$, then $H'_{\eta(x)}\le H_t$. Hence for all $t\in\triv$,
	$$\alpha_{h(x)}\alpha_x^{-1}\in H_t.$$
	Let $x,y\in\kl$. Then for all $h\in H_{\eta(y)}$,
	$$\qq{x}{h(y)}=\alpha_xh(y)=\alpha_xh\alpha_x^{-1}(\alpha_x(y))\in H_{\eta(y)}(\qq xy)= H_{\eta(\qq xy)}(\qq xy),$$
	and for all $h\in H_{\eta(x)}$,
	$$\qq{h(x)}y=\alpha_{h(x)}(y)=\alpha_{h(x)}\alpha_x^{-1}(\alpha_x(y))\in H_{\eta(\qq xy)}(\qq xy).$$
	Therefore the set $\lambda_\eta:=
	\hspace{2pt}\cdot\hspace{-7pt}\bigcup_T H_t\backslash \kl_t$ is a kei with structure inherited from $\kl$. For all $\overline x,\overline y\in\lambda_\eta$ with $\eta(x)\neq\eta(y)$, we have $\alpha_x\in H_{\eta(y)}$, therefore
	$$\qq{\overline x}{\overline y}=\overline{\qq xy}=\overline {\alpha_xy}=\overline y.$$
	Hence $\lambda_\eta$ is the disjoint union
	$$\lambda_\eta= \bigsqcup_T H_t\backslash \kl_t \in\Kei.$$
	Finally, let $(\k_t)_{t\in T}\in \Kei^T$ and let $f\colon \kl\to \k_\triv$ be a morphism in $\Kei_{/\triv}$. Then for all $t\in \triv$, for all $y\in \kl_t$ and all $x\in\kl\setminus\kl_t$,
	$$f(\alpha_x(y))=f(\qq{x}{y})=\qq{f(x)}{f(y)}=f(y).$$
	Therefore
	$f_{|\kl_t}$ is well-defined on the quotient $H_t\backslash \kl_t$, hence $f$ factors through
	$$\lambda_\eta= \bigsqcup_T H_t\backslash \kl_t\xrightarrow{\overline f}\k_\triv.$$
We conclude that
	$$\Hom_{\Kei_{/\triv}}(\kl,\k_\triv)\simeq \Hom_{\Kei_{/\triv}}\big(\bigsqcup_T N_t\backslash\kl_t,\k_\triv\big)\simeq \prod_{T} \Hom_{\Kei}(N_t\backslash\kl_t,\k_t).$$
\end{proof}

\section{Fundamental Arithemtic Keis}
\label{Section: fundamental arithmetic kei}



In this section we define the fundamental kei $\kei_n$ of square-free integers $n\in\NN$ by emulating the construction of $\kei_L$, the fundamental kei of a link $L$ in $S^3$. In the case of traditional links, $\kei_L$ is intimately tied to the automorphism group of some infinite branched cover of the sphere. The construction is similar here, and since infinite Galois groups are profinite, $\kei_n$ will in fact be a \emph{profinite} kei. We then analogously define $\k$-colorings of $n$ in terms of said $\kei_n$. 

\subsection{The Fundamental Kei of a link}
\label{Subsection: Fund kei of link}
Our definition of the fundamental arithmetic kei $\kei_n$ is motivated by a presentation of the fundamental kei $\kei_L$ found in the introduction.
The purpose of this section is to justify this presentation. 
Neither object $\kei_L,\SQ_L$ discussed here - nor indeed the general notion of a \emph{quandle} - are needed anywhere else in the paper. If you are comfortable with $\kei_L$ as presented in the introduction, feel free to skip this part.

\begin{proposition}
\label{Prop: fund kei justification}
	Let $L\subseteq S^3$ be a link. Let $M_L$ denote the branched double cover of $S^3$ with ramification locus $L$ and let $\widetilde M_L$ denote the universal cover of $M_L$. Then there is an augmented kei
	$$\left(\pi_0(\widetilde M_L\times_{S^3}L)\;,\;\Aut_\Top(\widetilde M_L/S^3)\;,\;m\right)\in\AKei$$
	that recovers the fundamental kei $\kei_L$ that corepresents colorings of $L$:
	$$\kei_L\simeq \pi_0(\widetilde M_L\times_{S^3}L).$$
\end{proposition}

An \emph{augmented quandle} $(Q,G,\alpha)$, much like an augmented kei, consists of a set $Q$, a group $G$ acting on $Q$, and an augmentation map $\alpha\colon Q\to G$, satisfying two of the three augmented kei axioms:
$$\forall x\in Q\;,\;\;g\in G\;:\;\;\alpha_x(x)=x\;\;\textnormal{and}\;\;\alpha_{g(x)}=g\alpha_xg^{-1}.\; \footnote{See \cite{ARTICLE:Joyce1982} for the definition of quandles and augmented quandles in general, and quandles one associates to an oriented link $L$. We caution the reader that what we call the fundamental quandle is \emph{not} the fundamental quandle from \cite{ARTICLE:Joyce1982}. Rather this is the \emph{knot quandle} defined in \cite{ARTICLE:Joyce1982}.}$$
The set $Q$, together with the binary operator $(x,y)\mapsto \alpha_x(y)$, is an algebraic structure known as a \emph{quandle}.
To be more specific, quandles are a more general notion than keis, obtained by relaxing the second axiom, requiring that each element act merely as a permutation (rather than an involution). Quandles to oriented links are what keis are to unoriented links: much like keis, one can use quandles to color an oriented link $L$, and these colorings are similarly controled by a \emph{fundamental quandle} $\SQ_L$ attached to $L$. A topological definition of $\SQ_L$ is found in \cite[\textsection 14]{ARTICLE:Joyce1982}, where structure is given by an augmentation with $\pi_1(S_L)$. This boils down to the following presentation\footnote{See \cite[p.17]{BOOK:Nosaka2017}}:

We take $U_L\subseteq S^3$ to denote a tubular neighborhood of $L$, $S_L$ to denote the complement $S_L=S^3\setminus U$, and $\partial U_L\subseteq S_L$ the boundary of $U_L$.  
We fix a base point $*\in S_L$. Then $\SQ_L$ is defined as the set of paths $*\rightsquigarrow \partial U_L$ in $S_L$ - up to homotopy.
The structure on $\SQ_L$ is defined by an augmentation with $\SG_L=\pi_1(S_L)$. Here $\SG_L$ acts on $\SQ_L$ via concatenation, and the augmentation map $m\colon \SQ_L\to\SG_L$ maps a path $x$ to a \emph{meridian} $m_x\in \SG_L$, looping once around a component of $L$ in a manner consistent with the orientation of $L$.

\begin{proof}(\cref{Prop: fund kei justification})
	Let $\widetilde S_L$ denote the universal covering of $S_L$, and let
	$$\SG_L:=\pi_1(S_L)\simeq\Aut(\widetilde S_L/S_L).$$
	The pullback $\widetilde S_L\times_{S_L} \partial U_L$ in $\Top$ consists of 
paths
$*\rightsquigarrow \partial U$ in $S_L$, up to homotopy that fixes both endpoints. Passing to connected components we have an isomorphism$$
\pi_0(\widetilde S_L\times_{S_L} \partial U )\iso \SQ_L$$
of $\SG_L$-sets. We denote by $H\unlhd \SG_L$ the subgroup
$$H=\langle m_x^2\mid x\in \SQ_L\rangle\unlhd \SG_L,$$
and by
$$Y=H\backslash \widetilde S_L$$
the covering of $S_L$ corresponding to $H$.
By \cite[Thm. 10.2]{ARTICLE:Joyce1982}, 
$$\kei_L \simeq H\backslash \SQ_L\simeq \pi_0(Y\times_{S_L}\partial U),$$
with structure on $H\backslash \SQ_L$ given by augmentation with the group
$$H\backslash \SG_L\simeq \Aut(Y/{S_L}).$$
For every branched cover $Z\to S^3$ that is unramified away from $L$, we denote by $Z^\circ=Z\times_{S^3}S_L$. In \cite[\textsection 5]{WEBSITE:Winker1984}, the group $H\backslash \SG_L $ sits in an exact sequence
$$\begin{tikzcd}[row sep=tiny]
	0\ar[r]& \pi_1(M_L) \ar[r]&  H\backslash \SG_L \ar[r]& \Aut_{\Top}(M_L/S^3) \ar[r]& 0\\
	&\Aut((\widetilde M_L)^\circ/M_L^\circ) \ar[u,"\scriptstyle{||}",phantom]&&\Aut(M_L^\circ/S_L) \ar[u,"\scriptstyle{||}",phantom]
\end{tikzcd}\;\;.
$$
From the proof in \cite{WEBSITE:Winker1984}, it is not hard to see that
$$H\backslash \SG_L\simeq \Aut_\Top((\widetilde M_L)^\circ/S_L)\simeq \Aut_\Top(\widetilde M_L/S^3).$$

The pullback $M_L\times_{S^3}U_L\subseteq M_L$ is a disjoint union of solid tori that homotopy retracts onto $ M_L\times_{S^3}L$, and has boundary
$$\partial(M_L\times_{S^3}U_L) =
M_L^\circ\times_{S_L}\partial U_L\subseteq M_L.$$
The same holds for any covering space of $M_L$. Therefore
$$\kei_L=\pi_0((\widetilde M_L)^\circ\times_{S_L}\partial U_L)\simeq \pi_0(\widetilde M_L\times_{S^3} U_L)\simeq \pi_0(\widetilde M_L\times_{S^3} L).$$

\end{proof}

\subsection{The Fundamental Kei of an Arithmetic Link}

\begin{definition}
	We denote by $\NNN$ the set of square-free positive integers.
\end{definition}

\begin{definition}
	For $n=2^km\in\NN$ with $m$ odd, we define
	$$n^*=(-1)^{\frac{m-1}2}n.$$
	It is easily verified that $(n_1n_2)^* =n_1^*n_2^*$. Thus for $n\in\NNN$ we have
	$$n^*=\prod_{p|n}p^*.$$
\end{definition}

\begin{definition}
\label{Def: Ln Quad Field}
	Let $1\neq n\in\NNN$. By $\L_n$ we denote the quadratic number field
	$$\L_n=\QQ(\sqrt{n^*}).$$
	The field $\L_n$ is ramified at precisely $n$.
	We define the field $\F_n$ to be the maximal unramified extension of $\L_n$:
	$$\F_n=\L_n^{un}.$$
	The extension $\F_n/\QQ$ is Galois. By $\G_n$ we denote the Galois group
	$$\G_n=\Gal(\F_n/\QQ)\in\Pro\Grp^\fin.$$
\end{definition}

\begin{proposition}
\label{Prop: Fm subseteq Fn and ramifications}
Let $1\neq m,n\in\NNN$ s.t. $m|n$. Then
$$\F_m\subseteq\F_n,$$
and $\F_m$ is the maximal subfield of $\F_n$ that is unramified over $\QQ$ away from $m$.
\end{proposition}
\begin{proof}
	All primes $p|n$ are doubly-ramified in $\L_n$ and in  $\L_m\L_n=\L_m\L_{n/m}$. A short computation of ramification indices shows that the extension
	$$\L_m\L_n/\L_n$$
	is unramified. Since $\F_m/\L_m$ is unramified, so is the tower $\F_m\L_n/\L_m\L_n/\L_n$, whereby
	$$\F_m\subseteq \F_m\L_n\subseteq \L_n^{un}= \F_n.$$
	Finally, the ramification index of any prime $\p$ in $\F_n$ is at most $2$. Therefore any field $F\subseteq \F_n$ that contains $\L_m$ and is unramified over $\QQ$ away from $m$ is necessarily unramified over $\L_m$. Therefore $\F_m$ is the maximal subfield of $\F_n$ that is unramified over $\QQ$ away from $m$.
\end{proof}

\begin{definition}
	For $p$ prime, we denote by
	$$\tau_p\in \Gal(\L_p/\QQ)$$
	the only non-trivial element in $\Gal(\L_p/\QQ)\simeq\ZZ/2\ZZ$.
\end{definition}

\begin{definition}
\label{Def: kei_n as set}
	Let $n\in\NNN$. We define $\kei_n\in\Top$ to be
	$$\kei_n:=\left|\Spec\left(\O_{\F_n}\otimes_\ZZ\ZZ/n\ZZ\right)\right|\in\Top.$$
As a set, $\kei_{n}$ comprises all primes $\p$ in $\F_n$ dividing $n$. For every $F\subseteq \F_n$ finite, Galois over $\QQ$, $|\Spec\left( \O_F\otimes_\ZZ \ZZ/n\ZZ \right)|$ is finite discrete. The forgetful functor 
	$$\textrm{Schemes}^\op\to \Top$$
	preserves cofiltered limits, therefore
	$$\kei_{n}= |\Spec \left(\O_{\F_n}\otimes_\ZZ\ZZ/n\ZZ\right)| \simeq\lim_{\substack{\longleftarrow\\F}} |\Spec\left( \O_F\otimes_\ZZ \ZZ/n\ZZ \right)|\in \Pro\Set^\fin.$$
	The profinite group $\G_n$ acts continuously on $\kei_{n}$ via $\p\mapsto g(\p)$.
\end{definition}

The goal now is to describe a topological kei structure on $\kei_n$.

\begin{definition}
	Let $n\in\NNN$ and let $m|n$ be prime. By $\kei_{n,m}$ we denote
	$$\kei_{n,m}=|\Spec \left(\O_{\F_n}\otimes_\ZZ \ZZ/m\ZZ\right)|\subseteq \kei_n.$$
For $p|n$ prime, the $\G_n$-action on $\kei_n$ restricts to a transitive action on $\kei_{n,p}$. For $\p\in \kei_{n,p}$, the stabilizer of $\p$ in $\G_n$ is the decomposition group of $\p$, denoted by
	$$D_\p:=\Stab_{\G_n}(\p).$$
\end{definition}

\begin{definition}
\label{Def: arithmetic meridian}
	Let $n\in\NNN$ and let $\p\in \kei_{n}$. By $\I_{\p}\le \G_n$ we denote the inertia group
	$$\I_\p:=\{g\in\G_n\;|\;\forall \alpha\in\O_{\F_n}\;,\;g(\alpha)-\alpha\in\p\}\le \G_n$$
	over $p$, the rational prime under $\p$. 
	The extension $\F_n/\L_p$ is unramified at $p$ and $\L_p/\QQ$ is totally ramified at $p$, 
	therefore 
	$$\I_\p\simeq\Gal(\L_p/\QQ)\simeq\ZZ/2\ZZ.$$
We denote the unique nontrivial element in $\I_{\p}$ by $$\m_{n,\p}\in\I_{\p}\le \G_n.$$
\end{definition}

\begin{proposition}
\label{Prop: fun. kei is aug. kei per p}
	Let $n\in\NNN$, let $\p\in \kei_{n}$ and let $g\in\G_n$. Then
	\begin{enumerate}
		\item $\m_{n,\p}(\p)=\p$.
		\item $\m_{n,\p}^2=1_{\G_n}$.
		\item $g\m_{n,\p}g^{-1}=\m_{n,g(\p)}$.
	\end{enumerate}
\end{proposition}
\begin{proof}
	Claim $(1)$ holds because the inertia group $\I_{\p}$ is a subgroup of the decomposition group $D_{\p}$. In particular $\m_{n,\p}\in\I_{\p}$ stabilizes $\p$:
	$$\m_{n,\p}(\p)=\p.$$
	Claim $(2)$ holds because $\I_{\p}\simeq\ZZ/2\ZZ$ - therefore $\m_{n,\p}\in \I_{\p}$ is an involution.
	Claim $(3)$ holds because conjugation by $g$ maps $\I_{\p}$ isomorphically onto $\I_{g(\p)}$. As the sole non-trivial elements in their respective groups, $\m_{n,\p}$ is mapped to $\m_{n,g(\p)}$:
	$$g\m_{n,\p}g^{-1}=\m_{n,g(\p)}.$$
\end{proof}

\begin{proposition}
\label{Prop: fun. kei aug. map is continuous per p}
	Let $n\in\NNN$. Then the function
	$$\m_n\colon \kei_{n}\to \G_n\;\;,\;\;\;\;\p\mapsto \m_{n,\p}$$
	is continuous.
\end{proposition}
\begin{proof}
	Let $p|n$ be prime. The $\G_n$-action on $\kei_{n,p}$ is continuous and transitive. Fix some $\p_0\in \kei_{n,p}$. The stabilizer $D_{\p_0}\le \G_n$ is closed, the quotient map $\G_n\to\G_n/D_{\p_0}$ is open, and the map
	$$\G_n/D_{\p_0}\xrightarrow{gD_{\p_0}\longmapsto g(\p_0)}\kei_{n,p}$$ is a homeomorphism.
	 The composition
	$$\G_n\xrightarrow{g\mapsto g(\p_0)}\kei_{n,p}\xrightarrow{\m_{n}}\G_n\;\;,\;\;\;\;g\mapsto g \m_{ n,\p_0}g^{-1}$$
	is also continuous. It follows that $(\m_{n})_{|\kei_{n,p}}\colon \kei_{n,p}\to \G_n$ is continuous.
	Since $\kei_n\simeq\coprod_{p|n}\kei_{n,p}\in\Top$, we conclude that $\m_n\colon\kei_n\to\G_n$ is continuous.
\end{proof}

\begin{proposition}
\label{Prop: An is Top AugQdl}
	Let $n\in\NNN$. Then $(\kei_n,\G_n,\m_n)$ is a topological augmented kei.
\end{proposition}
\begin{proof}
Let $\p\in\kei_n$ and let $g\in\G_n$. From \cref{Prop: fun. kei is aug. kei per p},
$$\m_{n,\p}(\p)=\p\;\;\;\;\textrm{,}\;\;\;\;\;\;\;\; \m_{n,\p}^2=1_{\G_n} \;\;\;\;\;\;\;\;\textrm{and}\;\;\;\;\;\;\;\; g\m_{n,\p}g^{-1}=\m_{n,g(\p)}.$$
It follows that $(\kei_n,\G_n,\m_n)$ is an augmented kei. 
The $\G_n$-action on $\kei_n$ is continuous, and by \cref{Prop: fun. kei aug. map is continuous per p} the augmentation map $\m_n\colon \kei_n\to\G_n$ is continuous. We conclude that $(\kei_n,\G_n,\m_n)$ is a topological augmented kei.
\end{proof}


\begin{definition}
\label{Def: AKei_n and Kei_n}
	Let $n\in\NNN$. The \emph{augmented arithmetic kei} of $n$ is defined as
	$$\A_n=(\kei_n,\G_n,\m_n)\in\Pro\AKei^\fin.$$
	The \emph{fundamental arithmetic kei} of $n$ is defined to be the kei $\kei_n\in\Pro\Kei^\fin$, with structure induced from the augmentation with $\G_n$:
	$$\qq{\p}{\q}:=\m_{n,\p}(\q).$$
\end{definition}

\begin{proposition}
\label{Prop: arithmetic kei is concise}
	Let $n\in\NNN$. Then $\A_n\in\Pro\AKei^\fin$ is concise.
\end{proposition}
\begin{proof}
	Let $N\le \G_n$ denote the closed subgroup
	$$N=\overline{\langle\{\m_{n,\p}\;|\;\p\in\kei_n\}\rangle}\le \G_n.$$
	From \cref{Prop: fun. kei is aug. kei per p}, $g\m_{n,\p}g^{-1}=\m_{n,g(\p)}$ for all $\p\in\kei_n$ and $g\in\G_n$, hence $N\unlhd\G_n$ is normal. Since $\F_n/\QQ$ is unramified away from $n$ and $\I_{\p}\le N$ for all $\p\in\kei_n$, the quotient $\G_n/N$ corresponds to a nowhere-ramified Galois extension of $\QQ$. By the Hermite-Minkowski theorem this is $\QQ$ itself, hence
	$$\overline{\langle\{\m_{n,\p}\;|\;\p\in\kei_n\}\rangle} =N=\G_n.$$
	Hence $\A_n\in\Pro\AKei^\fin$ is concise.
\end{proof}

\begin{proposition}
\label{Prop: arith fun qdl small}
	Let $n\in\NNN$. Then $\kei_n\in\Pro\Kei^\fin$ is small.
\end{proposition}
\begin{proof}
	Let $d\in\NN$. The open subgroups of $\G_n$ of index $d$ correspond to number fields $L\subseteq \F_n$ of degree $[L:\QQ]=d$. Such $L$ are unramified away from $n$, therefore satisfy
$$\Disc(L)\le n^d.$$
	By the Hermite-Minkowsk theorem, there are finitely many such fields $L$. Hence $\G_n\in\Pro\Grp^\fin$ is small.
	By \cref{Prop: arithmetic kei is concise}, $\A_n\in\Pro\AKei^\fin$ is concise. For each $p|n$, the group $\G_n$ acts transitively on $\kei_{n,p}$. Therefore
	$$\G_n\backslash \kei_n\simeq\{p|n\;\textrm{prime}\}\in\Set^\fin.$$
	By \cref{Prop: concise kei small condition}, $\kei_n$ is small.
\end{proof}

\begin{example}
	Following \cite{ARTICLE:Yamamura1997}, for $n=3,7,11,19,43,67,143$, we have $\F_n=\L_n$. In terms of $\kei_n$, these $n$ are indistinguishable from the unknot, since
	$$\kei_n\simeq *.$$
\end{example}

\subsection{$\k$-Coloring Arithmetic Links}

\begin{definition}
\label{Def: k-col arith}
	Let $n\in\NNN$ and let $\k\in\Kei^\fin$. We define sets
	$$\Col_\k(n):=\Hom_{\Kei}^\cont(\kei_n, \k)$$
and
	$$\Col_\k^\dom(n):= \{f\colon\kei_n\twoheadrightarrow \k\}
	\subseteq\Col_\k(n).$$
	By \cref{Prop: arith fun qdl small}, $\Col_\k(n)$ is finite. 	We define $\col_\k(n) , \col_\k^\dom(n)\in\NN$ via
	$$\col_\k(n):=\left|\Col_\k(n)\right|\;\;,\;\;\;\; 
	\col^\dom_\k(n):=\left|\Col^\dom_\k(n)\right|.$$
\end{definition}

\begin{definition}
	Let $n\in\NNN$, let $\k\in\Kei^\fin$, and let $F\subseteq \F_n$ be finite and Galois over $\QQ$. A $\k$-coloring $f\in\Col_\k(n)$ of $n$ is said to be \emph{defined on $F$} if $f$ is well-defined on the quotient $\Gal(\F_n/F)\backslash\kei_n$.
	We say that $\Col_\k(n)$ is \emph{defined on $F$} if
	every $f\in\Col_\k(n)$ is defined on $F$.
\end{definition}

\begin{lemma}
\label{lma: arithmetic factor through N quotient}
	Let $n\in\NNN$, let $\k\in\Kei^\fin$ and let $f\colon\kei_n\to \k$ be a kei morphism. Denote by $\k'=f(\kei_n)\le \k$, and by
	$$H_f:=\ker\left(\rho_{\A_n,f}\colon\G_n\twoheadrightarrow\Inn(\k')\right)\unlhd\G_n.$$
	Let $F_f=(\F_n)^{H_f}\subseteq\F_n$. Then $f$ factors through $$H_f\backslash \kei_n\simeq\Spec \left(\O_{F_f}\otimes_\ZZ\ZZ/n\ZZ\right) \to \k,$$
	which is to say that $f$ is defined on $F_f$.
\end{lemma}
\begin{proof}
	By \cref{Prop: arithmetic kei is concise}, $\A_n\in\Pro\AKei^\fin$ is concise. The claim therefore follows from \cref{Prop: profinite augkei statements} (\cref{Prop: factor from concise through quotient})
\end{proof}

\begin{lemma}
\label{Lem: coloring preserves limits}
	Let $n\in\NNN$. Then the construction $\Col_\k(n)$ assembles to a functor
	$$\Col_\bullet(n)\colon\Kei^\fin\to\Set^\fin$$
	that preserves finite limits.
\end{lemma}
\begin{proof}
As is the case for all corepresentable functors, the functor 
$$\Hom_{\Pro\Kei^\fin}(\kei_n,-) \colon\Pro\Kei^\fin\to\Set $$
preserves all limits. The canonical embedding
$$\Kei^\fin\hookrightarrow\Pro\Kei^\fin$$
preserves finite limits. The composition
$$\Col_\bullet(n)\colon\Kei^\fin\hookrightarrow\Pro\Kei^\fin\to\Set$$
therefore preserves finite limits. By \cref{Prop: arith fun qdl small}, $\kei_n$ is small, so these finite limits in $\Set$ are in fact computed in $\Set^\fin$.
\end{proof}

\begin{proposition}
\label{Prp: Arithmetic kei quotients}
	Let $m,n\in\NNN$ be s.t. $m|n$. Let $H\le\G_n$ denote the closed subgroup
	$$H:=\overline{\langle\m_{n,\p}\;|\;\p\in\kei_n\setminus\kei_{n,m}\rangle}\le \G_n.$$
	Then
	$$\F_n^{H}=\F_m$$
	and
	$$H\backslash \kei_{n,m}\simeq\kei_m\in\Pro\Kei^\fin.$$
\end{proposition}
\begin{proof}
	The elements $\{\m_\p\}_{\p\in\kei_n\setminus\kei_{n,m}}$ are precisely the nontrivial inertia elements for primes in $\F_n$ over all $p|\frac nm$. 
	Therefore $\F_n^{H}\subseteq\F_n$ is the maximal subfield ramified over $\QQ$ at most at $m$.
	By \cref{Prop: Fm subseteq Fn and ramifications}, 
	$$\F_n^{H}= \F_m$$
	- and $H=\ker(\G_n\twoheadrightarrow \G_m)\unlhd\G_n$ accordingly. Let
	$$\A_{n,m}=(\kei_{n,m},\G_n,(\m_n)_{|\kei_{n,m}})\in\Pro\AKei^\fin.$$
	By \cref{Prop: profinite augkei statements}
 (\cref{Prop: factor augkei morphism through quotient}),
	the canonical morphism
	$$(f,\widetilde f)\colon\A_{n,m}\to\A_m\;\;:\;\;\;\; f( \p)= \p\cap\O_{\F_m} \;\;,\;\;\;\;\widetilde f(g)=g_{|\F_m} $$
	factors through the quotient  $H\backslash \A_{n,m}$. The map $H\backslash \A_{n,m}\to \A_m$ is an isomorphism because it induces isomorphisms $H\backslash \G_n\iso \G_m$ and 
	$$H\backslash\kei_{n,m}= H\backslash\Spec\left(\O_{\F_n}\otimes_\ZZ\ZZ/m\ZZ\right)
	\iso \Spec\left(\O_{\F_n^{H}}\otimes_\ZZ\ZZ/m\ZZ\right)=\kei_m.$$
	By \cref{Prop: augmented quotient by normal N}, the kei structure on $H\backslash\kei_{n,m}$ is induced by that of $\kei_{n,m}\in\Kei$, therefore we obtain an isomorphism
	$$H\backslash\kei_{n,m}\iso\kei_m\in\Pro\Kei^\fin.$$
\end{proof}

\begin{lemma}
\label{Lma: quot of arithmetic kei is trivial II}
	Let $n\in\NNN$. Then $\G_n\backslash\kei_n$ is the trivial kei $\{p|n\;\textrm{prime}\}$:
	$$\G_n\backslash\kei_n\simeq \triv_{\omega(n)}\in\Kei.$$
\end{lemma}
\begin{proof}
	As a set,
	$$\G_n\backslash\kei_n=\Gal(\F_n/\QQ)\backslash\left|\Spec(\O_{\F_n}\otimes_\ZZ\ZZ/n\ZZ)\right|\simeq\left|\Spec(\ZZ/ n\ZZ)\right|=\{p|n\;\textrm{prime}\}.$$
	The kei structure on $\G_n\backslash\kei_n$ is induced by Galois action on $\kei_n$. These act trivially on sets of rational primes. Therefore $\G_n\backslash\kei_n\in\Kei$ is trivial:
	$$\G_n\backslash\kei_n\simeq \triv_{\omega(n)}.$$
\end{proof}

\begin{proposition}
\label{Prop: Trivial colorings factor through pi_0 II}
Let $\triv\in\Kei$ be trivial. There is a natural isomorphism
$$\Hom_{\Kei}^\cont(\kei_n, \triv)\simeq\triv ^{\{p|n\;\textnormal{prime}\}}.$$
\end{proposition}
\begin{proof}
Let $f\colon \kei_n\to \triv$ be a kei morphism. Denote by $\triv':=f(\kei_n)\le \triv $, also trivial. Therefore $\Inn(\triv')=\{\Id\}$, and
$$H_f:=\ker \Big(\rho_{\A_n,f}\colon \G_n\to\Inn(\triv')\Big)=\G_n.$$
By \cref{lma: arithmetic factor through N quotient}, $f$ factors through $\G_n\backslash\kei_n$. By \cref{Lma: quot of arithmetic kei is trivial II}, $\G_n\backslash \kei_n$ is a trivial kei with underlying set $\{p|n\;\textnormal{prime}\}$. Therefore
$$\Hom_{\Kei}^\cont(\kei_n, \triv)\simeq
\Hom_\Kei(\G_n\backslash\kei_n, \triv)\simeq \triv ^{\{p|n\;\textnormal{prime}\}}.$$
\end{proof}

\begin{proposition}
\label{Prop: arithmetic coloring by disjoint union}
	Let $n\in\NNN$ be square-free and let $\k_1,\k_2\in\Kei^\fin$. Then
	$$\col_{\k_1\sqcup \k_2}(n)=\sum_{n=n_1n_2}\col_{\k_1}(n_1)\cdot\col_{\k_2}(n_2).$$
\end{proposition}
\begin{proof}
Let $\triv_2=\{*_1,*_2\}\in\Kei$. By \cref{Prop: Trivial colorings factor through pi_0 II}, there is a natural bijection
	$$\Hom_\Kei(\kei_n,\triv_2)\simeq \triv_2 ^{\{p|n\;\textnormal{prime}\}}\simeq \{(n_1,n_2)\;|\;n=n_1n_2\}.$$
	for $\eta\colon \kei_n\to \triv_2 $ corresponding to $(n_1,n_2)$, we denote by
	$$\Col_{k_1\sqcup \k_2}(n)_{(n_1,n_2)}:= \Hom_{/\eta}(\kei_n, \k_1\sqcup\k_2).$$
	By \cref{Prp: Map concise to disjoint union practical} and \cref{Prp: Arithmetic kei quotients}, there is a bijection 
	$$\Col_{\k_1\sqcup \k_2}(n)_{(n_1,n_2)}\simeq$$
	$$\simeq\Hom_{\Kei}^\cont(H_{n,n_1}\backslash \kei_{n,n_1},\k_1)\times
\Hom_{\Kei}^\cont(H_{n,n_2}\backslash \kei_{n,n_2},\k_2)\simeq$$
$$\simeq\Hom_{\Kei}^\cont(\kei_{n_1},\k_1)\times
\Hom_{\Kei}^\cont(\kei_{n_2},\k_2).$$
where $H_{n,m}:=\Gal(\F_n/\F_m)$ for all $m|n$ as in \cref{Prp: Arithmetic kei quotients}.
	Therefore 
	$$\Col_{\k_1\sqcup \k_2}(n)_{(n_1,n_2)}
	\simeq\Col_{\k_1}(n_1)\times \Col_{\k_2}(n_2),$$
	and
	$$\col_{\k_1\sqcup \k_2}(n)=\sum_{n=n_1n_2}\left| \Col_{k_1\sqcup \k_2}(n)_{(n_1,n_2)}\right|=\sum_{n=n_1n_2} \col_{\k_1}(n_1)\cdot\col_{\k_2}(n_2).$$
\end{proof}

\section{Main Conjecture}
\label{Section: main conjecture}
While the analogy of Barry Mazur likens positive square-free integers to links, the fibration of the affine line $\mathbb{A}^1_{\FF_p}$ over $\FF_p$, paints a sharper picture, likening a degree-$k$ separable polynomial $f(t)\in\FF_p[t]$ to the closure of a braid on $k=\log_p\left|\FF_p[t]/f(t)\right|$ strands. By coarse analogy, $n\in\NNN$ should be a braid on $\log n$ strands. The scope of this claim is limited by the lack of a Frobenius map, more so by the fact that $\log n\notin \NN$. Nevertheless, the claim is bolstered by statistical phenomena about random numbers: A random number of magnitude $X$ is prime with probability $\approx \left(\log X\right)^{-1}$, mirroring the fact that the closure of a random braids $\sigma\in B_k$ is connected with probability $1/k$. The aim of this paper is to use asymptotic statistical phenomena of kei-coloring invariants to lend further credence to this notion.

For a finite kei $\k\in\Kei^\fin$, basic statistical invariants about $\k$-coloring of random links admit particularly straightforward description when viewed as closures of braids: Let $k\in\NN$ and let $B_k$ be the Artin braid group on $k$ strands. The braid closure $\overline \sigma$ of a braid $\sigma\in B_k$ is a link in $S^3$. The function mapping $\sigma\in B_k$ to $\col_\k(\overline{\sigma})\in\NN$ extends to a locally constant function on the profinite completion $\widehat{B}_k$. By \cite{WEBSITE:DavisSchlankHilbert2023} there exists for every $\k\in\Kei^\fin$ an integer-valued polynomial $P_\k(t)\in\QQ[t]$ s.t. for all $k\gg 0$,
$$P_\k(k)=\int\limits_{\widehat{B}_k}\col_\k(\overline \sigma)d\mu,$$
where $\mu$ is the Haar measure on $\widehat{B}_k$.
We use this to further bolster the notion that a \emph{random} $n\in\NNN$ should be thought of as the closure of a \emph{random} braid on $\approx\log X$ strands.

\subsection{The Hilbert Polynomial of a Finite Kei}
Let $\k\in\Kei^\fin$. Recall in \cite{WEBSITE:DavisSchlankHilbert2023} that there exists an integer-valued polynomial $P_\k(x)\in\QQ[x]$ s.t. for all $k\gg 0$,
$$P_\k(k)=\int\limits_{\widehat{B}_k}\ccol_\k d\mu,$$
where $\mu$ is the Haar measure on the profinite braid group $\widehat{B}_k$.
The degree of $P_\k$ is explicitly calculated in terms of $\k$:
$$\deg P_\k=-1+\max_{\k'\le\k}\left|\Inn(\k')\backslash\k'\right|.$$

\subsection{Statistics of $\k$-Coloring Random Arithmetic Links}

\begin{definition}
	Let $s\in\NNN$ be square-free. By $\NNs$ we denote the set of square-free elements coprime to $s$:
	$$\NNs:=\{n\in\NNN\;|\;(n,s)=1\}.$$
\end{definition}

\begin{definition}
Let $s\in\NNN$, let $f\colon\NNs\to\CC$, and let $L\in\CC$. We say $f$ converges to $L$ if for all $\varepsilon>0$ there exists $u_\varepsilon\in\NNs$ s.t. for all $u\in\NNs$,
	$$u_\varepsilon|u\;\;\Longrightarrow\;\;\left|f(u)-L\right|<\varepsilon.$$
		We shall use either of the following notations:
		$$\lim_{u\in\NNs}f(u)=L\;\;\;\;\;\;\;\;,\textnormal{}\;\;\;\;\;\;\;\; f(u)\xrightarrow{u\in\NNN}L.$$
\end{definition}

\begin{definition}
	Let $s\in\NNN$ be squarefree. By $\gamma_s\in\RR$ we denote
	$$\gamma_s:=\prod_{p\nmid s}(1-p^{-2})\cdot\prod_{p|s}(1-p^{-1})=\zeta(2)^{-1}\cdot \prod_{p|s}(1+p^{-1})^{-1}.$$
\end{definition}

\begin{remark}
\label{Lma: basic s asymptotics}
For square-free $s$, the number of $s$-coprime square-free integers up to $X$ is
$$\left|\NNs\cap[1,X]\right| =\gamma_s X(1+o(1)).$$
Therefore the probability of $n\in\NNs\cap[1,X]$ being prime is
$$\frac{\pi(X)+O_s(1)}{\left|\NNs\cap[1,X]\right|}=\frac{\pi(X)+O_s(1)}{X}\frac{X}{\left|\NNs\cap[1,X]\right|}
=\frac{1+o(1)}{\gamma_s\log X}.$$
	In likening a squarefree integer to the closure of a braid, the correct number of strands is purported to be inverse to the probability of primality. In light of this computation, for fixed $s\in\NNN$ we should think of $n\in\NNs\cap[1,X]$ as the closure of a braid with 
	$\approx\gamma_s \log X$
	strands.
\end{remark}

\begin{definition}
\label{Def: curly N and E}
	Let $f\colon\NNN\to \CC$ and let $s\in\NNN$. For $1\le X\in\RR$ we denote by
	$$\counter_{f,s}(X):=\sum_{n\in\NNs\cap[1,X]}f(n)$$
	and
	$$\avg_{f,s}(X)=\frac{\counter_{f,s}(X)}{|\NNs\cap[1,X]|}.$$
	For $\k\in\Kei^\fin$, we denote by
	$$\counter_{\k,s}(X):=\counter_{\col_\k,s}(X)\;\;\;\;\textrm{and}\;\;\;\;\;\;\;\; \avg_{\k,s}(X):=\avg_{\col_\k,s}(X).$$
\end{definition}

\begin{definition}
\label{Def: gen sum type}
Let $f\colon \NNN\to\CC$, let $0\le \beta\in\NN$ and let $0\neq c\in\CC$. We say $f$ has \emph{generic summatory type} $(\beta,c)$ if for every $s\in\NNN$ there exists a limit
		$$c_s(f):=\lim_{X\to\infty} \frac{\avg_{f,s}(X)}{(\gamma_s\log X)^{\beta}}\in\CC\setminus\{0\},$$
	 and if
	$$c(f):=\lim_{s\in\NNN} c_s(f)=c.$$
	We denote by $\GenType\beta{c}$ the set of all such functions:
	$$\GenType\beta c:=\{f\colon\NNN\to\CC\;|\;f\;\textrm{has generic summatory type}\;(\beta,c)\}.$$
	In the case of $f=\col_\k$ for some $\k\in\Kei^\fin$, we denote by
	$$c_s(\k):=c_s(\col_\k)\;\;\;\;\textnormal{and}\;\;\;\;\;\;\;\;c(\k):=c(\col_\k).$$
\end{definition}
\begin{remark}
	In light of \cref{Lma: basic s asymptotics}, for $f\in\GenType \beta c$ and $s\in\NNN$,
	$$\gamma_s^{1+\beta}c_s(f)=\lim_{X\to\infty}\frac{\counter_{f,s}(X)}{X\log^\beta (X)}.$$
	In practice, it is the RHS here that we will be computing for every $s\in\NNN$.
\end{remark}

We are now ready to state the main conjecture of the paper:
\begin{conjecture}
\label{Conj: Main conjecture}
	Let $\k\in\Kei^\fin$. Let $P_\k(t)\in\QQ[t]$ be the Hilbert polynomial of $\k$, as defined in \cite{WEBSITE:DavisSchlankHilbert2023}.
	Then there exists $0<c\in\QQ$ s.t.
	$$\col_{\k}\in\GenType{\deg P_\k} c.$$
\end{conjecture}

\section{Preliminary Numerical Computations}
\label{Section: preliminary computations}
A few preliminary definitions will be useful. In the following discussion of arithmetic functions, we follow standard notation:
For arithmetic functions $f,g\colon\NN\to\CC$, the convolution
$f*g\colon\NN\to\CC$ is defined via
$$(f*g)(n):=\sum_{n_1n_2=n}f(n_1)g(n_2).$$
By $\delta_1\colon\NN\to\CC$ we denote the unit with respect to convolution:
$$\delta_{1}(n)=\delta_{1,n}=\begin{cases}
	1&,\;n=1\\
	0&,\;n\neq 1
\end{cases}\;\;.$$
By $\uno,\mu\colon\NN\to\CC$ we respectively denote the constant function $\uno(n)=1$
and the M\"obius function $\mu$, its inverse with respect to convolution. This is the multiplicative arithmetic function satisfying for all prime powers $p^\nu\neq 1$,
$$\mu(p^\nu)=\begin{cases}
		-1&,\;\nu=1\\
		0&,\;\nu>1
	\end{cases}.$$
	
The function $\omega\colon \NN\to \CC$ counts distinct prime divisors:
$$\omega(n):=|\{p|n\;\textnormal{prime}\}|.$$

\begin{definition}
	Let $s\in\NNN$. We denote by $\ley,\gy\;\colon\NN\to\CC$ the functions
	$$\gy(n)=\begin{cases}
		1&,\;(n,s)=1\\
		0&,\;(n,s)>1
	\end{cases}\;\;\;\;,\;\;\;\;\;\;\;\;
	\ley(n)=\begin{cases}
		1&,\;n|s^k\;\textnormal{ for some }k\gg 0\\
		0&,\;\textnormal{otherwise}
	\end{cases}\;\;.$$
	We denote by $\mu_s,\muy\colon\NN\to\CC$ the functions
	$$\mu_s(n)=\gy(n) \mu(n)\;\;,\;\;\;\;\muy(n)=\ley(n)\mu(n).$$
	These are all multiplicative arithmetic functions, and they satisfy
	$$\gy*\mu_s=\ley*\muy=\delta_1\;\;,\;\;\;\;\gy*\ley=\uno.$$
\end{definition}

\subsection{Divisor Counting}
Let $\powa\in\NN$. Then for all $0<n\in\NN$
$$\uno^{* \powa}(n)=\underbrace{\uno*\cdots*\uno}_{\powa}(n)=d_\powa(n),$$
where $d_\powa$ is the $\powa $-fold divisor counting function. This is the multiplicative function satisfying
$$\uno^{*\powa}(p^\nu)={\powa +\nu-1\choose \powa-1}$$
For all prime powers $p^\nu$.
We begin by stating some standard results regarding the counting of divisors:

\begin{lemma}
	\label{Lma: divisor func sub mul}
	Let $1\le \powa\in\NN$. Then for all $m,n\in\NN$,
	$$\uno^{* \powa}(mn)\le \uno^{* \powa}(m)\uno^{* \powa}(n).$$
\end{lemma}
\begin{proof}
There is an equality
$${\powa +\nu-1\choose \powa-1}=\frac{(\powa +\nu-1)\cdots (\nu+1)}{(\powa-1)!}=\prod_{k=1}^{\powa-1}(1+\frac \nu k).$$
The desired inequality immediately follows from the fact that
$$(1+\frac \nu k) (1+\frac {\nu'}k)\ge 1+\frac {\nu+\nu'} k.$$
\end{proof}

\begin{lemma}
\label{Lma: divisor arithmetic that doesnt really need wrapping}
	Let $\powb\in\NN$ and let $n\in\NNN$. Then
	$$\sum_{m|n} \powb^{\omega(m)}=(\powb+1)^{\omega(n)}.$$
\end{lemma}
\begin{proof}
Since $n$ is square-free, we have
$$\uno^{* \powb}(n)=\prod_{p|n}\uno^{* \powb}(p)= \powb ^{\omega(n)}.$$
Hence
$$(\powb +1)^{\omega(n)}=\uno^{*(\powb +1)}(n)=\sum_{mm'=n}\uno^{* \powb}(m)\uno(m')=\sum_{m|n} \powb ^{\omega(m)}.$$	
\end{proof}

\begin{lemma}
\label{Lma: Piltz comprehensive with error term}
	Let $1\le \powa\in\NN$. Then there exists $C=C_\powa>0$ s.t. the  term
	$$\theta_\powa(X):=\sum_{n\le X}\uno^{* \powa}(n)-\frac{X\log^{\powa-1} (X)}{(\powa-1)!}$$
	is bounded for all $X\ge 1$ by
	$$|\theta_ \powa(X)|\le \begin{cases}
	CX(\log^{\powa-2} (X)+1) &, \; \powa\ge 2\\
	C&, \; \powa =1
	\end{cases}\;\;.$$
	In particular, as $X\to\infty$ we have
	$$\sum_{n\le X}\uno^{* \powa}(n)= \frac{X\log^{\powa-1}(X)}{(\powa-1)!}(1+o(1)).$$
\end{lemma}
\begin{proof}
	If $\powa =1$, then for all $X\ge 1$,
	$$|\theta_1(X)|=|\floor X-X|\le 1.$$
	If $\powa\ge 2$, then there exists $D_\powa(t)\in\RR[t]$ with leading term $\frac{t^\powa}{(\powa-1)!}$ s.t. 
	$$\sum_{n\le X}\uno^{* \powa}(n)=XD_\powa(\log X)+O(X^{1-{\frac1{\powa-1}}}(\log^{\powa-2}(X)).\; \footnote{See for example \cite[\textsection XII]{BOOK:Titch1987} or \cite[\textsection 2.1.1]{BOOK:MontVaughan2006}.} $$
	Hence
	$$\frac{\theta_\powa(X)}{X}=\frac1X\sum_{n\le X}\uno^{* \powa}(n)-\frac{\log^{\powa-1}(X)}{(\powa-1)!}=O(\log^{\powa-2}(X)).$$
The function $$\frac{\theta_\powa(X)}{X(\log^{\powa-2}(X)+1)}\colon[1,\infty)\to\RR$$
is then bounded not only for $X\gg0$, but for all $X\in[1,\infty)$ because
$$X\ge 1\;\Longrightarrow\;\log^{\powa-2}(X)+1\ge 1.$$
Hence there is some $C_\powa>0$ s.t. for all $X\ge 1$,
$$|\theta_\powa(X)|\le C_\powa\cdot X(\log^{\powa-2}(X)+1).$$
\end{proof}

We have the following immediate corollary in the same spirit of stating a bound for all $X\ge 1$:
\begin{lemma}
\label{Lma: Piltz lazy}
	Let $1\le \powa\in\NN$. Then there is a constant $C=C_\powa>0$ s.t. for all $X\ge 1$,
	$$\sum_{n\le X}\uno^{* \powa}(n)\le CX(\log^{\powa-1}(X)+1).$$
\end{lemma}

\begin{lemma}
	Let $1\le\powa\in\NN$ and let $\epsilon>0$. Then for every prime $p$,
	$$\sum_{\nu=0}^\infty\frac{\uno^{*\powa}(p^\nu)}{p^{\epsilon\nu}}=(1-p^{-\epsilon})^{-\powa}.$$
	More generally, for every $s\in\NNN$,
	$$\sum_{n=1}^\infty\frac{\ley^{* \powa}(n)}{n^ \epsilon}=\prod_{p|s}(1-p^{-\epsilon})^{-\powa}.$$
\end{lemma}
\begin{proof}
From the formal equality of power series
$$(1-t)^{-\powa}=\sum_{\nu=0}^\infty {\powa +\nu-1\choose \powa-1}t^{\nu},$$
we have
	$$\sum_{\nu=0}^\infty\frac{\uno^{* \powa}(p^\nu)}{p^{\epsilon\nu}}=\sum_{\nu}{\powa +\nu-1\choose \powa-1}p^{-\epsilon\nu}=(1-p^{-\epsilon})^{-\powa}.$$
	The function $\ley^{*a}$ is multiplicative, satisfying
	$$\ley^{* \powa}(p^\nu)=\begin{cases}
	\uno^{* \powa}(p^\nu)&,\;p|s\\
	0&,\; p\nmid s,\;\nu\ge 1
	\end{cases}\;\;.$$
	Hence
	$$\sum_{n=1}^\infty\frac{\ley^{* \powa}(n)}{n^\epsilon}=\prod_{p|s} \sum_{\nu}\frac{\uno^{* \powa}(p^\nu)}{p^{\epsilon\nu}}=\prod_{p|s}(1-p^{-\epsilon})^{-\powa}.$$
\end{proof}

\begin{definition}
	Let $1\le \powa\in\NN$ and let $f\colon \NN\to \RR$ be  multiplicative, s.t. $$0\le f(n)\le \uno^{*\powa}(n)$$
	for all $n\in\NN$. For every prime $p$ we denote by
	$$M_p^{(\powa)}(f):=(1-p^{-1})^\powa\sum_{\nu=0}^\infty\frac{f(p^\nu)}{p^\nu}\in (0,1],$$
and we denote by $M^{(\powa)}(f)$ the infinite product
$$M^{(\powa)}(f):=\prod_pM_p^{(\powa)}(f)\in[0,1].$$
\end{definition}

\begin{lemma}
\label{Lma: technical prime sum prod convergence}
Let $\powa\in\NN$ and let $f\colon \NN\to \RR$ be a function s.t. $0\le f\le \uno^{* \powa}$. 
If the product $M^{(\powa)}(f)$ converges (to nonzero limit), then so does the infinite sum
$$\sum_p\frac{\powa-f(p)}p.$$
\end{lemma}

\begin{proof}
	For all prime $p$,
	$$1-M_p^{(\powa)}(f)=M_p^{(\powa)}(\uno^{* \powa})-M_p^{(\powa)}(f) \ge (1-p^{-1})^ \powa\cdot \frac{\uno^{* \powa}(p)-f(p)}{p}\ge 2^{-\powa}\frac{\powa-f(p)}{p}.$$
	Hence
$$0\le M_p^{(\powa)}(f)\le 1-2^{-\powa}\frac{\powa-f(p)}{p}\le 1.$$
The convergence of $M^{(\powa)}(f)$ implies the convergence of
$$\prod_p\left(1-2^{-\powa}\frac{\powa-f(p)}{p}\right).$$
The convergence of the desired sum then follows, as
$$\sum_p\frac{\powa-f(p)}{p}=2^ \powa\cdot \sum_p 2^{-\powa}\frac{\powa-f(p)}{p}.$$
\end{proof}

\begin{lemma}
\label{Lma: technical loggish convergence}

Let $f\colon \NN\to \RR$ be a function s.t.
$$0\le f\le \ley^{*\powb}$$ 
for some $b\in\NN$ and $s\in\NNN$. Then for all $i\in\NN$, the series
	$$\sum_{n=1}^\infty\frac{f(n)}{n}\log^i (n)$$
	converges.
\end{lemma}
\begin{proof}
	Let $i\in\NN$. Then $\log^i(n)<\sqrt n$ for $n\gg 0$. Convergence is therefore dominated by
	$$\sum_{n=1}^\infty\frac{f(n)}
	{\sqrt n}\le \sum_{n=1}^\infty\frac{\ley^{* \powb}(n)}
	{\sqrt n}=\prod_{p|s}(1-p^{-\frac 12})^{-\powb}.$$
\end{proof}

\begin{lemma}
\label{Lma: technical s-bounded lemma}
	Let $1\le \powa\in\NN$ and let $f\colon\NN\to\RR$ be multiplicative s.t.
	$$0\le f\le \ley^{* \powb}$$
	for some $\powb\in\NN$ and $s\in\NNN$. For $g\colon \NN\to\CC$ and $X\ge 1$ we denote by
	$$\ccounter_g(X):=\sum_{1\le n\le X}g(X).$$
	Then as $X\to\infty$,
	$$\lim_{X\to\infty}\frac {\ccounter_{f*\mathds 1^{* \powa}}(X)}{X (\log X)^{\powa-1}}=\frac{1}{(\powa-1)!}\cdot \sum_{m=1}^\infty\frac{f(m)}m.$$
\end{lemma}

\begin{proof}
Let $X\ge 0$. Then
$$\ccounter_{f*\mathds 1^{* \powa}}(X) = \sum_{m\le X}f(m)\sum_{a\le X/m}\mathds 1^{* \powa}(a)=\sum_{m\le X}f(m)\cdot\left(\frac Xm\frac{\log^{\powa-1}(X/m)}{(\powa-1)!}+\theta_ \powa(X/m)\right).$$
Therefore
$$\frac {\ccounter_{f*\mathds 1^{* \powa}}(X)}{X \log^{\powa-1}(X)} = \frac1{(\powa-1)!}\sum_{m\le X}\frac{f(m)}{m}\left(1-\frac{\log m}{\log X}\right)^{\powa-1} + \frac1{\log^{\powa-1}(X)}\sum_{m\le X}\frac{f(m)}m\cdot \frac mX\theta_ \powa(X/m).$$
Considering the first summand, it follows from \cref{Lma: technical loggish convergence} that for all $i>0$,
$$0\le \sum_{m\le X}\frac{f(m)}{m}\left(\frac{\log m}{\log X}\right)^i\le \frac1{\log^i(X)}\sum_{m=1}^\infty\frac{f(m)}m\log ^i(m)\xrightarrow{X\to\infty}0.$$
Therefore
$$\sum_{m\le X}\frac{f(m)}{m}\left(1-\frac{\log m}{\log X}\right)^{\powa-1} = \sum_{m\le X}\frac{f(m)}{m} + o(1)= \sum_{m=1}^\infty\frac{f(m)}{m} + o(1). $$
Turning to the remaining summand, we first address the case $\powa\ge 2$. By \cref{Lma: Piltz comprehensive with error term} there is a constant $C>0$ s.t.
$$\left|\frac mX\theta_ \powa(X/m)\right|\le C\left(\log^{\powa-2}(X/m)+1\right)\le C\left(\log^{\powa-2}(X)+1\right)$$
for all $1\le m\le X$.
Therefore
$$\left| \sum_{m\le X}\frac{f(m)}m\cdot \frac mX\theta_ \powa(X/m)\right|\le \sum_{m\le X}\frac{f(m)}m\cdot \left|\frac mX\theta_\powa(X/m)\right|\le $$
$$\le C\left(\log^{\powa-2}(X)+1\right)\sum_{m=1}^{\infty}\frac{f(m)}m=o(\log^{\powa-1}(X)).$$
In the case $\powa =1$, this summand is bounded, for some constant $C>0$, by
$$\left|\sum_{m\le X}\frac{f(m)}X\theta_1(X/m)\right|\le \frac CX\sum_{m\le X}f(m) \le \frac CX\sum_{m\le X} \ley^{* \powa}(m) =O\left(X^{-1}\log^{\omega(s)}(X)\right)=o(1).$$
Thus for any $\powa\ge 1$,
$$\frac {\ccounter_{f*\mathds 1^{* \powa}}(X)}{X\log^{\powa-1}(X)} = \frac1{(\powa-1)!}\left(\sum_{m=1}^\infty\frac{f(m)}{m}+o(1)\right) + \frac{o(\log^{\powa-1}(X))}{\log^{\powa-1}(X)}=$$
$$=\frac1{(\powa-1)!}\sum_{m=1}^\infty\frac{f(m)}m+o(1).$$
\end{proof}

\begin{lemma}
\label{Lma: ugly numerical inequality}
	Let $k\in\NN$ and let $0\le b_i\le a_i$ for $1\le i\le k$. Then
	$$\prod_{i}a_i-\prod_ib_i\le \sum_{i}(a_i-b_i)\prod_{j\neq i}a_j.$$
\end{lemma}
\begin{proof}
If $a_i=0$ for some $i$ then both sides of the inequality equal $0$. We assume therefore that $a_i>0$ for all $i$. Then $0\le \frac{b_i}{a_i}\le 1$ for all $i$, therefore
$$1-\prod_i \frac{b_i}{a_i}\le \sum_i\left(1-\frac{b_i}{a_i}\right).$$
Thus
$$\prod_{i}a_i-\prod_ib_i =\prod_{j}a_j\left(1-\prod_i \frac{b_i}{a_i}\right)\le \prod_{j}a_j \sum_i\left(1-\frac{b_i}{a_i}\right) = \sum_{i}(a_i-b_i)\prod_{j\neq i}a_j.$$
\end{proof}

\begin{proposition}
\label{prop: Tenenbaum 3.12 II}
    Let $1\le \powa\in\NN$, and let $f\colon\NN\to\RR$ be multiplicative  s.t. $$0\le f(n)\le \uno^{* \powa}(n)$$
    for all $n\in\NN$. For $X\ge 1$ we denote by
	$$\ccounter_f(X):=\sum_{1\le n\le X}f(X).$$
    Then
	$$\lim_{X \to \infty}\frac{\ccounter_f(X)}{X\log^{\powa-1}(X)}=
	\frac{1}{(\powa-1)!}M^{(\powa)}(f).$$
\end{proposition}
\begin{proof}
For all $s\in\NNN$ we define the function $\fya\colon\NN\to\RR$ via
$$\fya=(f\cdot\ley)*\gy^{* \powa}.$$
Since $f,\ley,\gy$ are multiplicative, then so is $\fya$, satisfying for all prime powers $p^\nu\neq1$,
$$\fya(p^\nu)=\sum_{i=0}^\nu f(p^i)\ley(p^i)\gy^{* \powa}(p^{\nu-i})=\begin{cases}
	f(p^\nu)&,\;p|s\\
	\uno^{* \powa}(p^\nu)&,\;p\nmid s
\end{cases}\;\;.$$
Therefore for all $s\in\NNN$,
$$0\le f\le \fya\le\uno^{*\powa}.$$
The function $\fya$ can be written
$$\fya=(f\cdot\ley)*\gy^{* \powa} =\left((f\cdot \ley )*\uno^{* \powa}\right)*\muy^{* \powa}.$$
Note that $\muy^{* \powa}(a)=0$ for $a>s^{\powa}$.
Since
$0\le f\ley\le \uno^{* \powa}\ley = \ley^{*\powa}$,
by \cref{Lma: technical s-bounded lemma} for all $X\ge 1$, we have
$$\frac{\ccounter_{\fya}(X)}X=
\frac 1X\sum_{ab\le X}
\muy^{*\powa}(a)\left((f\cdot\ley)*\uno^{*\powa}\right)(b)=$$
$$=\sum_{a\le s^ \powa}\frac {\muy^{* \powa}(a)}{a}\cdot
\frac {a}X\sum_{b\le \frac Xa}\left((f\cdot \ley)* \uno^{* \powa}\right)(b)=$$
$$=\sum_{a\le s^ \powa}\frac {\muy^{* \powa}(a)}a\cdot
\frac{\log^{\powa-1}(X/a)}{(\powa-1)!}\cdot \sum_{m=1}^\infty\frac{(f\cdot \ley) (m)}m\cdot (1+o(1))=$$
$$= \prod_{p|s}(1-p^{-1})^\powa\cdot\frac{\log^{\powa-1}(X)}{(\powa-1)!}\cdot
\prod_{p|s}\sum_{\nu}\frac{f(p^\nu)}{p^\nu}\cdot (1+o(1))= $$
$$= \frac{\log^{\powa-1}(X)}{(\powa-1)!}\cdot\prod_{p|s}M_p^{(\powa)}(f)\cdot (1+o(1)).$$
For all $s\in\NNN$ and $X\ge 1$, $\ccounter_{f}(X)\le \ccounter_{\fya}(X)$. Therefore
$$0\le \limsup_{X\to\infty}\frac{\ccounter_f(X)}{X\log^{\powa-1}(X)}\le \inf_{s\in\NNN}\limsup_{X\to\infty} \frac{\ccounter_{\fya}(X)}{X\log^{\powa-1}(X)}=$$
$$=\frac1{(\powa-1)!}\prod_{p}M_p^{(\powa)}(f)= \frac1{(\powa-1)!}M^{(\powa)}(f).$$
Should the product $M^{(\powa)}(f)$ diverge, it necessarily diverges to $0$, whereby
$$\lim_{X\to\infty}\frac{\ccounter_f(X)}{X(\log X)^{\powa-1}}=0=\frac1{(\powa-1)!}M^{(\powa)}(f),$$
as desired.
We assume therefore that $M^{(\powa)}(f)$ converges. By \cref{Lma: ugly numerical inequality}, for all $s\in\NNN$ and $n\in\NN$,
$$0\le \fya(n)-f(n)\le \sum_{p^\nu||n} \fya(n/p^\nu)(\fya(p^\nu)-f(p^\nu))=$$
$$= \sum_{\substack{p^\nu||n\\p\nmid s}} \fya(n/p^\nu)(\uno^{* \powa}(p^\nu)-f(p^\nu))\le \sum_{\substack{p^\nu||n\\p\nmid s}} \uno^{* \powa}(n/p^\nu)(\uno^{* \powa}(p^\nu)-f(p^\nu)).$$
Summing over $1\le n\le X$ we have
$$0\le \ccounter_{\fya}(X)-\ccounter_f(X)\le \sum_{n\le X}\sum_{\substack{p^\nu\|n\\p\nmid s}} \uno^{* \powa}(n/p
^\nu)\left(\uno^{* \powa}(p^\nu)-f(p^\nu)\right)\le$$
$$\le \sum_{\substack{p\nmid s}}\left[\sum_{p^2|n}^{n\le X} \uno^{* \powa}(n)+\sum_{p||n}^{n\le X} \uno^{* \powa}(n/p
)\left(\uno^{* \powa}(p)-f(p)\right)\right].
$$
By \cref{Lma: divisor func sub mul}, the first summand is bounded as follows:
$$0\le \sum_{p\nmid s}\sum_{p^2|n}^{n\le X}\uno^{* \powa}(n)=\sum_{\substack{p\nmid s}}\sum_{m\le \frac X{p^2}}\uno^{* \powa}(mp^2)\le  \sum_{\substack{\\p\nmid s}}\sum_{m\le \frac X{p^2}}\uno^{* \powa}(m) \uno^{* \powa}(p^2)= $$
$$={\powa +1\choose 2} \sum_{p\nmid s}^{p\le \sqrt{X}}\sum_{m\le X/p^2}\uno^{* \powa}(m).$$
Therefore there exists $C=C_\powa>0$ s.t. for all $X\ge 1$,
$$\sum_{p\nmid s}\sum_{\substack{p^2|n}}^{n\le X}\uno^{* \powa}(n)\le \sum_{\substack{p\nmid s}}^{p\le \sqrt{X}} C\frac{X}{p^2}\left(\log^{\powa-1}(X/p^2)+1\right)\le CX\left(\log^{\powa-1}(X)+1\right)\sum_{p\nmid s}\frac1{p^2}.$$
The second summand is likewise bounded for all $X\ge 1$ by
$$\sum_{p\nmid s}\sum_{\substack{p||n}}^{n\le X}\uno^{* \powa}(n/p)(\uno^{* \powa}(p)-f(p))\le \sum_{\substack{p\nmid s}} (\powa-f(p))\sum_{m\le \frac Xp}\uno^{* \powa}(m)\le $$
$$\le \sum_{\substack{p\nmid s}}^{p\le X} (\powa-f(p))\cdot C\frac Xp\left(\log^{\powa-1}(X/p)+1\right)\le $$
$$\le CX\left(\log^{\powa-1}(X)+1\right)\sum_{p\nmid s}\frac{\powa-f(p)}{p}.$$
Hence there is a constant $C>0$ such that for all $s\in\NNN$ and
$X\ge 1$,
$$0\le \frac{\ccounter_{f_s^{(\powa)}}(X)-\ccounter_f(X)}{X(\log ^{\powa-1}(X)+1)}
\le C \sum_{p\nmid s}\frac1{p^2}+C\sum_{p\nmid s}\frac{\powa-f(p)}p.$$
The series $\sum_p\frac{\powa-f(p)}p$ converges following \cref{Lma: technical prime sum prod convergence} and the assumed convergence of $M^{(\powa)}(f)$. The series $\sum_pp^{-2}$ converges as well, therefore
$$\sum_{p\nmid s}\frac1{p^2}+\sum_{p\nmid s}\frac{\powa-f(p)}p \xrightarrow{s\in\NNN}0.$$
It follows that there is uniform convergence for $X\in[2,\infty[$:
$$\frac{\ccounter_{f_s^{(\powa)}}(X)}{X(\log X)^{\powa-1}}\xrightarrow{s\in\NNN} \frac{\ccounter_f(X)}{X(\log X)^{\powa-1}}.$$
Thus we obtain the desired limit
$$\lim_{X\to\infty} \frac{\ccounter_f(X)}{X\log^{\powa-1}(X)} =\lim_{s\in\NNN}\lim_{X\to\infty} \frac{\ccounter_{\fya}(X)}{X\log^{\powa-1}(X)} =$$
$$= \lim_{s\in\NNN} \frac1{(\powa-1)!}\prod_{p|s}M_p^{(\powa)}(f)=\frac1{(\powa-1)!}M^{(\powa)}(f).$$
\end{proof}

\subsection{Sums of Kronecker Symbols}
Recall that the Kronecker symbol $\Jac ab\in\{0,\pm1\}$, defined for $a,b\in\ZZ$ with $b>0$, satisfies for primes $p\neq 2$:
$$\Jac p2=\Jac{2^*}p=\Jac 2p=(-1)^{\frac{p^2-1}8}.$$

\begin{proposition}
\label{Prop: Schlank-Polya-Vinogradov}
    Let $\chi\colon \ZZ\to\CC$ be a non-principal character of modulus $d$ and let $s\in\NNN$. Then there exists $C>0$, independent of $\chi$ and $s$, s.t.
    $$\left|\sum_{A\le n\le B}(\chi\cdot\mathds1_{s})(n)\right|\le C\cdot 2^{\omega(s)}\sqrt d\log d$$
    for all $A,B\in\RR$.
\end{proposition}
\begin{proof}
Let $A,B\in\RR$. A run-through of $s$-coprime integers is obtained by M\"obius inversion:
$$\sum_{A\le n\le B}\chi(n)\gy(n)=\sum_{u|s}\mu(u)\sum_{\substack{A\le n\le B}}^{u|n}\chi(n)=\sum_{u|s}\mu(u)\chi(u)\sum_{\frac Au\le n\le \frac Bu}\chi(n).$$
	By the theorem of P\'olya-Vinogradov\footnote{see \cite{ARTICLE:Polya1918} or \cite{ARTICLE:Hildebrand1988}}, there is a constanct $C>0$ s.t.
	$$\left|\sum_{A\le n\le B}\chi(n)\gy(n)\right|\le \sum_{u|s}\left|\sum_{\frac Au\le n\le \frac Bu}\chi(n)\right|\le C\cdot 2^{\omega(s)} \sqrt d\log d.$$
\end{proof}

\begin{proposition}
\label{Prop: Bound S chi}
	For $\chi\colon\ZZ\to\CC$ a non-principal Dirichlet character with modulus $d$, for $s\in\NNN$ and $1\le Y\le  X$, 
	denote by
	$$S_{\chi,s}(X)=\sum\limits_{1\le n\le X} \mu_s^2(n)\chi(n)\;\;,\;\;\;\; S_{\chi,s}(Y,X)=\sum\limits_{Y <  n\le X} \mu_s^2(n)\chi(n).$$
	Then there is a constant $0<C\in\RR$ s.t. for all such $\chi$, $s\in\NNN$ and $X\ge Y\ge 1$,
	$$|S_{\chi,s}(X)|\;, \;\;|S_{\chi,s}(Y,X)| \le
C\cdot X^{\frac12}\cdot 2^{\frac12\omega(s)}d^{\frac14}\log^{\frac12} (d).$$
\end{proposition}
\begin{proof}
    $$S_{\chi,s}(X)=\sum_{1\le n\le X} \mu^2(n) \gy(n)\chi(n)=\sum_{a=1}^\infty\mu(a)\sum_{\substack{a^2|n}}^{1\le n\le X} \gy(n)\chi(n)=$$
    $$=\sum_{a}\mu(a)\sum_{1\le m\le \frac X{a^2}} \gy(a^2m)\chi(a^2m)=\sum_{1\le a\le \sqrt{X}}\mu_s(a)\chi(a)^2\sum_{1\le m\le \frac X{a^2}} \gy(m)\chi(m).$$
    Let $1\le Z\le \sqrt X$ be some parameter to be determined later. Then
    \begin{equation}
    \label{Eqn: S bound split}
    S_{\chi,s}(X)=\left(\sum_{1\le a\le Z}+\sum_{Z<a\le\sqrt X}\right)\mu_s(a)\chi(a)^2\sum_{1\le m\le \frac X{a^2}}\chi(m)\gy(m).
    \end{equation}
    By \cref{Prop: Schlank-Polya-Vinogradov} there is $C'>0$ s.t. the first summand in (\ref{Eqn: S bound split}) is bounded by
    $$\left| \sum_{1\le a\le Z}\mu_s(a)\chi(a)^2\sum_{1\le m\le \frac X{a^2}}\chi(m)\gy(m)\right|\le \sum_{a\le Z}\left|\sum_{m\le \frac X{a^2}}\chi(m)\cdot \gy(m)\right|\le$$
    $$ \le Z\cdot C'2^{\omega(s)}\sqrt d\log d.$$
	The second summand in (\ref{Eqn: S bound split}) is bounded, for some constant $C''>0$, by
	$$\left| \sum_{Z< a\le \sqrt X}\mu_s(a)\chi(a)^2\sum_{1\le m\le \frac X{a^2}}\chi(m)\gy(m)\right| \le\sum_{a>Z}\sum_{m\le \frac X{a^2}}1\le \sum_{a>Z}\frac{X}{a^2}\le C''\cdot\frac{X}{Z}.$$
Therefore taking $C=\max\{C',C'',1\}$,
    \begin{equation}
    \label{Eqn: technical Joyce lemma - composite bound}
    \left|S_{\chi,s}(X)\right|\le
    Z\cdot C\cdot 2^{\omega(s)}\sqrt d\log d+ X\cdot \frac {C}Z.
    \end{equation}
    If $X<2^{\omega(s)}d^{\frac12}\log d$, then
    $$\left|S_{\chi,s}(X)\right|\le \sum_{1\le n\le X}| \mu_s^2(n)\chi(n)|\le X<2C\cdot X^{\frac12}\left(2^{\omega(s)}d^{\frac12}\log d\right)^{\frac12}=$$
    $$=2C\cdot X^{\frac12}\cdot 2^{\frac12\omega(s)}d^{\frac14}\log^{\frac12} (d).$$
    Otherwise, take $Z= X^{\frac12}\big(2^{\omega(s)}d^{\frac12}\log d\big)^{-\frac12}$. The modulus $d$ of a non-principal Dirichlet character $\chi$ is at least $3$, whereby $\sqrt d\log (d)>1$. Hence
    $$1\le Z\le X^{\frac12}2^{-\frac12\omega(s)}\le X^{\frac12},$$
    and (\ref{Eqn: technical Joyce lemma - composite bound}) then reads:
    $$\left|S_{\chi,s}(X)\right|\le 2C\cdot X^{\frac12}\cdot 2^{\frac12\omega(s)}d^{\frac14}\log^{\frac12} (d).$$
    For all $1\le Y\le X$, we then have
    $$\left|S_{\chi,s}(Y,X)\right|\le \left|S_{\chi,s}(X)\right| + \left|S_{\chi,s}(Y)\right|\le 4C\cdot X^{\frac12}\cdot 2^{\frac12\omega(s)}d^{\frac14}\log^{\frac12}(d).$$
\end{proof}

\begin{definition}
\label{Dfn: quasi bi-char}
	A function $B\colon\NN^2\to\CC$ is a \emph{bi-character} with \emph{slopes} $1\le c_1,c_2\in\NN$ if
	\begin{itemize}
		\item for all $a\in\NN$, there exist Dirichlet character $\chi_a$ of modulus $c_1a$ and $M_a\in\CC$ with $|M_a|\le 1$ s.t.
		$$B(a,-)=M_a\cdot(\chi_a)_{|\NN}\colon \NN\to\CC,$$
		\item for all $b\in\NN$, there exist Dirichlet character $\psi_b$ of modulus $c_2b$ and $N_b\in\CC$ with $|N_b|\le 1$ s.t.
		$$B(-,b)=N_b\cdot(\psi_b)_{|\NN}\colon \NN\to\CC.$$
We say $B$ is \emph{non-principal} at $a=a_0$ (resp. at $b=b_0$) if either $M_{a_0}=0$ (resp. $N_{b_0}=0$) or $\chi_{a_0}$ (resp. $\psi_{b_0}$) is non-principal.
	\end{itemize}
\end{definition}

\begin{proposition}
\label{Prop: bi-char summation}
	Let $s\in\NNN$ and let $B\colon\NN^2\to\CC$ be a bi-character with slopes $c_1,c_2$. Assume $a_0,b_0\in\NN$ are such that $B$ is non-principal at all $a\ge a_0$ and at all $b\ge b_0$.
	Then there exists a constant $C>0$, independent of $s,a_0,b_0,B$ s.t. for all $X\ge 1$,
	$$\left|\sum_{\substack{a\ge a_0\\b\ge b_0}}^{ab\le X}\mu_s^2(ab)B(a,b)\right|\le C\cdot 2^{\frac12\omega(s)} c_1^{\frac18 } c_2^{\frac18} \left(1+ \log \max\{c_1,c_2\}\right) X^{\frac78} \left(1+\log X\right)^{\frac12}.$$
\end{proposition}
\begin{proof}
	The function $B$ satisfies $B(a,b)=0$ whenever $(a,b)\neq 1$, therefore
	$$\mu_s^2(ab)B(a,b)=\mu_s^2(a)\mu_s^2(b)B(a,b).$$
	Let $X_1,X_2>0$, to be determined later, s.t. $X_1X_2=X$. Then
	$$\sum_{\substack{a\ge a_0\\b\ge b_0}}^{ab\le X} \mu_s^2(a)\mu_s^2(b)B(a,b) = \Bigg(\underbrace{\sum_{a=a_0}^{\floor{X_1}}\sum_{b=b_0}^{\floor{\frac{X}{a}}}}_{(I)}+
	\underbrace{\sum_{b=b_0}^{\floor{X_2}}\sum_{a=a_0}^{\floor{\frac{X}{b}}}}_{(II)} -
	\underbrace{\sum_{a=a_0}^{\floor{X_1}} \sum_{b=b_0}^{\floor{X_2}}}_{(III)}\Bigg) \mu_s^2(a)\mu_s^2(b)B(a,b).$$
The first summand $(I)$ is bounded by
	$$|(I)|=\left| \sum_{a=a_0}^{\floor{X_1}} \mu_s^2(a)\sum_{b=b_0}^{\floor{\frac{X}{a}}} \mu_s^2(b)B(a,b)\right|\le \sum_{a=a_0}^{\floor{X_1}}\left|  \sum_{b=b_0}^{\floor{\frac{X}{a}}} \mu_s^2(b)B(a,b)\right|=$$
	$$= \sum_{a=a_0}^{\floor{X_1}}|M_a|\left|  \sum_{b=b_0}^{\floor{\frac{X}{a}}} \mu_s^2(b)\chi_a(b)\right|\le \sum_{a=a_0}^{\floor{X_1}}
\left|  \sum_{b=b_0}^{\floor{\frac{X}{a}}} \mu_s^2(b)\chi_a(b)\right|.$$
	For all $a\ge a_0$, $\chi_a$ is non-principal of modulus $c_1a$. By \cref{Prop: Bound S chi} there exists $C'>0$, independent of $s,c_1,a$, s.t.
	$$\left|\sum_{b=b_0}^{\floor{\frac{X}{a}}} \mu_s^2(b)\chi_a(b)\right| =\left|S_{\chi_a,s}(b_0-1,X/a)\right|\le C' \cdot (X/a)^{\frac12} 2^{\frac12\omega(s)}(c_1a)^{\frac14}\log^{\frac12}(c_1a)$$
	for all $X\ge 1$. Since $a,c_1\ge 1$, we have $\log c_1,\log a\ge 0$, therefore
	$$\log(c_1a)=\log c_1+\log a\le (1+\log c_1)(1+\log a).$$
	Therefore
	$$|(I)|\le \sum_{a=a_0}^{\floor{X_1}}\left|  \sum_{b=b_0}^{\floor{\frac{X}{a}}} \mu_s^2(b)\chi_a(b)\right| \le
	C' \cdot 2^{\frac12\omega(s)}X^{\frac12}c_1^{\frac14} \left(1+ \log c_1\right)^{\frac12}
	\sum_{a=a_0}^{\floor{X_1}}a^{-\frac14}(1+\log a)^{\frac12}.$$
	There is a constant $C''>0$ s.t. for all $Y\ge 1$
	$$\sum_{1\le a\le Y}a^{-\frac14}(1+\log a)^{\frac12}\le C''
	Y^{\frac34}(1+\log Y)^{\frac12}.$$
	Hence
	$$|(I)|\le 
	C'C''\cdot 2^{\frac12\omega(s)}X^{\frac12}c_1^{\frac14} \left(1+ \log_2 c_1\right)^{\frac12}
	X_1^{\frac34}(1+\log X_1)^{\frac12}.$$
	Likewise,
	$$|(II)|\le 
	C'C''\cdot 2^{\frac12\omega(s)}X^{\frac12}c_2^{\frac14} \left(1+ \log_2 c_2\right)^{\frac12}
	X_2^{\frac34}(1+\log X_2)^{\frac12}.$$
	For the summand $(III)$, again by \cref{Prop: Bound S chi},
$$\left|(III)\right|=\left| \sum_{a=a_0}^{\floor{X_1}} \sum_{b=b_0}^{\floor{X_2}} \mu_s^2(a)\mu_s^2(b)B(a,b)\right|\le \sum_{a=a_0}^{\floor{X_1}}\left| \sum_{b=b_0}^{\floor{X_2}} \mu_s^2(b)\chi_a(b)\right| \le$$
$$\le \sum_{a=a_0}^{\floor{X_1}} C' 2^{\frac12\omega(s)}X_2^{\frac12}(c_1a)^{\frac14}\log(c_1a)^{\frac12}\le$$
$$\le 
C' 2^{\frac12\omega(s)}X_2^{\frac12} c_1^{\frac14} \left(1+ \log_2 c_2\right)^{\frac12}\sum_{a=a_0}^{\floor{X_1}} a^{\frac14}(1+\log a)^{\frac12} \le$$
$$\le C'C''' 2^{\frac12\omega(s)}X_2^{\frac12} c_1^{\frac14} \left(1+ \log_2 c_2\right)^{\frac12} X_1^{\frac54}(1+\log X_1)^{\frac12} .$$
We now take $X_1=c_1^{-\frac16}c_2^{\frac16 }X^{\frac12}$ and $X_2=c_1^{\frac16}c_2^{-\frac16 }X^{\frac12}$, so that
$$X^{\frac12} c_1^{\frac14} X_1^{\frac34}=
X^{\frac12} c_2^{\frac14} X_2^{\frac34} =
X_2^{\frac12} c_1^{\frac14} X_1^{\frac54}=
 c_1^{\frac18 } c_2^{\frac18} X^{\frac78}.$$
For $X\ge c_1^{-\frac13} c_2^{\frac13}$ we have $X_1\le X$, and for $X\ge c_1^{\frac13} c_2^{-\frac13}$ we have $X_2\le X$. Then for such $X$ we find, taking the above into account, that 
$$\left|\sum_{\substack{a\ge a_0\\b\ge b_0}} ^{ab\le X}\mu_s^2(ab)B(a,b)\right|\le |(I)|+|(II)|+|(III)|\le $$
	$$\le C'(2C''+C''')\cdot 2^{\frac12\omega(s)} c_1^{\frac18 } c_2^{\frac18} X^{\frac78}\left(1+ \log \max\{c_1,c_2\}\right)^{\frac12} \left(1+\log X\right)^{\frac12}.$$
	However if w.l.o.g. $c_2\ge c_1$ and $1\le X\le c_1^{-\frac13} c_2^{\frac13}$, then the bounded is obtained in simpler fashion:
	$$X=X^{\frac18}X^{\frac78}\le c_1^{-\frac1{24}} c_2^{\frac1{24}}X^{\frac78}\le c_1^{\frac1{8}} c_2^{\frac1{8}}X^{\frac78} $$
	and
	$$\log X\le \frac{\log c_2}3\le \log\max\{c_1,c_2\}.$$
	By \cref{Lma: Piltz lazy} there exists $\widetilde C>0$ s.t.
	$$\left| \sum_{\substack{a\ge a_0\\b\ge b_0}} ^{ab\le X}\mu_s^2(ab)B(a,b)\right| \le \sum_{\substack{a\ge a_0\\b\ge b_0}} ^{ab\le X}1 \le \sum_{\substack{n\le X}}(\uno*\uno)(n) \le \widetilde C X(1+\log X)\le $$
$$\le \widetilde C\cdot 2^{\frac12\omega(s)} c_1^{\frac18 } c_2^{\frac18} X^{\frac78}\left(1+ \log \max\{c_1,c_2\}\right)^{\frac12} \left(1+\log X\right)^{\frac12}.$$
The claim therefore holds for all $X\ge 1$ with $C:=\max\{C'(2C''+C'''),\widetilde C\}$.
\end{proof}

\begin{proposition}
\label{Prop: odd Kronecker is bi-char}
	The functions $B,B',B''\colon\NN^2\to\CC$,
	$$B(a,b)= \uno_2(ab)\Jac ba \;,\;\; 
	B'(a,b)= \uno_2(ab)\Jac {2b}a=\Jac 2a B(a,b)$$
	$$B''(a,b)= \uno_2(ab)\Jac b{2a} =\Jac b2 B(a,b) $$
	are bi-characters with slopes $c_1,c_2\le 8$. Moreover, $B$ (resp. $B'$; $B''$) is non-principal at $a\ge 2$ and at $b\ge 2$ ($a\ge 2$ and $b\ge 1$ ; $a\ge 1$ and $b\ge 2$).
\end{proposition}
\begin{proof}
Let $h\colon\ZZ\to\CC$ denote the function
$$h(n)=\Jac 2n=\Jac n2= \uno_2(n)(-1)^{\frac{n^2-1}8},$$
a Dirichlet character of modulus $8$. For odd $a,b\ge 1$, let $\chi_a,\psi_b\colon\ZZ\to\CC$ be
$$\chi_a(n)=\uno_2(n)\Jac na\;,\;\;\psi_b(n)=\uno_2(n) (-1)^{\frac{b-1}2\frac{n-1}2}\Jac nb.$$
Then $\chi_a$ (resp. $\psi_b$) is a
Dirichlet character of modulus $2a$ (resp. $4b$).
For $n\ge 1$,
$$\psi_b(n)=\uno_2(n)\Jac bn$$
- trivially for $n$ even, by quadratic reciprocity for $n$ odd.
For all $a\ge 1$ we have
$$B(a,-)=\uno_2(a)\cdot \chi_a\;,\;\;B'(a,-)=h(a)\cdot \chi_a\;,\;\;B''(a,-)=\uno_2(a)\cdot h\chi_a,$$
to be interpreted as zero for $a$ even. Likewise, for all $b\ge 1$ we have
$$B(-,b)=\uno_2(b)\cdot \psi_b\;,\;\;B'(-,b)=\uno_2(b)\cdot h\psi_b\;,\;\;B''(-,b)=h(b)\cdot \psi_b.$$
The function $h\chi_a$ (resp. $h\psi_b$) is a
Dirichlet character of modulus $8a$ (resp. $8b$). Moreover, $|\uno_2(n)|,|h(n)|\le 1$ for all $n$. Therefore $B$ (resp. $B'$ ; $B''$) is a bi-character with slopes $(c_1,c_2)=(2,4)$ (resp. $(2,8)$ ; $(8,4)$). For odd $m\ge 2$, the function $n\mapsto \Jac nm$ is clearly non-principal of modulus $m$. Since
$$\uno_2(n)\;\;\;\;(\;\textnormal{resp. } \uno_2(n)(-1)^{\frac{m-1}2\frac{n-1}2} \;\;;\;\;h(n)\;\;;\;\; \uno_2(n)(-1)^{\frac{m-1}2}h(n)\;) $$
has modulus $2$ (resp. $4 \;;\; 8 \;;\; 8$), coprime to $m$, it follows for all odd $m\ge 2$ that $\chi_m$ (resp. $\psi_m$ ; $h\chi_m$ ; $h\psi_m$) is non-principal. Moreover, for $m=1$ we have $h\chi_1=h\psi_1=h$, non-principal.
\end{proof}

\begin{lemma}
\label{Lma: basic counting lemma 2}
	Let $\powa>-1$ and let $1\le \powb\in\NN$. Then for all $X\ge 1$,
	$$\sum_{1\le n\le X}\mu^2(n)\cdot n^{\powa}\cdot \powb ^{\omega(n)}\le \frac1{1+\powa} X^{1+\powa}(1+\log X)^{\powb-1}.$$
\end{lemma}
\begin{proof}
The proof is by induction on $\powb $. The claim holds for $\powb =1$: since $\powa >-1$, for all $X\ge 1$ we have
$$\sum_{n\le X}\mu(n)^{2}n^{\powa}\le \int\limits_0^Xt^{\powa}dt=\frac1{1+ \powa}X^{1+ \powa}.$$
Assume next that the claim holds for some $\powb\in\NN$. By \cref{Lma: divisor arithmetic that doesnt really need wrapping},
	$$\sum_{n\le X}\mu^2(n) n^{\powa} (\powb +1)^{\omega(n)} =\sum_{\substack{m,d\ge 1\\md\le X}}\mu^2(md) (md)^{\powa} \powb ^{\omega(m)} \le$$
	$$\le\sum_{1\le d\le X}d^{\powa}\sum_{1\le m\le \frac Xd}\mu^2(m) m^{\powa} \powb ^{\omega(m)}.$$
	By the induction assumption, this is bounded by
	$$\frac1{1+ \powa} \sum_{d\le X}d^{\powa} (X/d)^{1+ \powa}(1+\log (X/d))^{\powb-1} \le \frac1{1+ \powa} X^{1+ \powa} \sum_{d\le X}d^{-1} (1+\log X)^{\powb-1} \le$$
	$$\le \frac1{1+ \powa} X^{1+ \powa} (1+\log X)^{\powb}.$$
Hence the statement holds for all $1\le \powb\in\NN$, $\powa >-1$ and $X\ge 1$.
	\end{proof}

\begin{lemma}
\label{Lma: basic counting lemma sqrt(2)}
	Let $\powa>-1$. Then there exists $0<C\in\RR$ s.t. for all $X\ge 1$,
	$$\left|\sum_{1\le n\le X}\mu^2(n)\cdot n^{\powa}\cdot 2^{\frac12\omega(n)}\right|\le C\cdot X^{1+ \powa}(1+\log X)^{\frac12}.$$
\end{lemma}
\begin{proof}
	For all $n$ we have 
	$$\mu^2(n)\cdot n^{\powa}\cdot 2^{\frac12\omega(n)}=
	\left(n^{\powa}\right)^{\frac12} \left(\mu^2(n)\cdot n^{\powa}\cdot 2^{\omega(n)}\right)^{\frac12}.$$
	By the Cauchy-Schwartz inequality,
	$$\sum_{n=1}^{\floor X}\mu^2(n)\cdot n^{\powa}\cdot2^{\frac12\omega(n)}\le
	\sqrt{\sum_{n=1}^{\floor X}n^{\powa}}\cdot\sqrt{\sum_{n=1}^{\floor X}\mu^2(n)\cdot n^{\powa}\cdot 2^{\omega(n)}}.$$
	Since $\powa >-1$, and following \cref{Lma: basic counting lemma 2}, this is bounded by
	$$C\cdot \sqrt{X^{1+ \powa}}\cdot\sqrt{X^{1+ \powa}(1+\log X)}=C\cdot X^{1 +\powa}\cdot (1+\log X)^{\frac12}.$$
\end{proof}

\begin{proposition}
\label{Prop: bound for T_s(X)}
	Let $s\in\NNN$. For all $1\le X\in\RR$, denote by
	$$T_s(X)=\sum_{\substack{a,b\ge 2\\ab\le X}}\mu_s^2(ab)\Jac ba.$$
	Then there exists a constant $C>0$ s.t. for all $1\le X\in\RR$,
	$$\left|T_s(X)\right|\le C\cdot 2^{\frac12\omega(s)}\cdot X^{\frac78}(1+\log X)^{\frac12}.$$
\end{proposition}
\begin{proof}
	We write
	$$T_s(X)=\sum_{\substack{a,b\ge 2\\ab\le X}}\mu_s^2(ab)\Jac ba=$$
	\begin{equation}
	\label{Eqn: Ts(X) split sum 2}
	=\left(\overbrace{\sum_{\substack{a,b\ge 2 \\ab\le X}}^{2\nmid a,\;2\nmid b}}^{(I)} +
	\overbrace{\sum_{\substack{a,b\ge 2 \\ab\le X}}^{2\nmid a ,\; 2|| b}}^{(II)}+
	\overbrace{\sum_{\substack{a,b\ge 2 \\ab\le X}}^{2||a ,\; 2\nmid b}}^{(III)}\right)
	\mu_s^2(ab)\Jac ba.\end{equation}
	We have
	$$(I)= \sum_{\substack{a,b\ge 2}}^{ab\le X}\mu_s^2(ab)\uno_2(ab)\Jac ba.$$
	$$(II)= \sum_{\substack{2b,a\ge 2}}^{2b\cdot a\le X}\mu_s^2(2b\cdot a)\uno_2(ab)\Jac{2b}a=\sum_{\substack{b\ge1,\;a\ge2}}^{ab\le X/2}\mu_s^2(ab) \uno_2(ab)\Jac{2b}a.$$
	$$(III)= \sum_{\substack{b,2a\ge 2}}^{b\cdot 2a\le X}\mu_s^2(b\cdot 2a)\uno_2(b)\Jac b{2a}= \sum_{\substack{b\ge 2,\;a\ge 1}}^{ab\le X/2}\mu_s^2(ab)\uno_2(ab) \Jac b{2a}.$$
	
	By propositions \ref{Prop: odd Kronecker is bi-char} and \ref{Prop: bi-char summation} there exists $C>0$ s.t.
	$$|(I)|\le C 2^{\frac12\omega(s)}X^{\frac78}(1+\log X)^{\frac12},$$
	$$
	|(II)|\le C 2^{\frac12\omega(s)}(X/2)^{\frac78}(1+\log (X/2))^{\frac12}\le C 2^{\frac12\omega(s)}X^{\frac78}(1+\log X)^{\frac12},$$
	and
	$$
	|(III)|\le C 2^{\frac12\omega(s)}(X/2)^{\frac78}(\log (1+X/2))^{\frac12}\le C 2^{\frac12\omega(s)}X^{\frac78}(1+\log X)^{\frac12}.$$
	Therefore
	$$\left|T_s(X)\right|
	\le \left|(I)\right| + \left|(II)\right| + \left|(III)\right| \le 3 C 2^{\frac12\omega(s)}X^{\frac78}(1+\log X)^{\frac12}.$$
\end{proof}

\section{Trivial Keis $\triv_a$}
\label{Section: trivial kei proof}
Recall that a kei $\k\in\Kei$ is trivial if $\qq xy=y$ for all $x,y\in\k$.
The following assertions about trivial keis are easily shown:
\begin{lemma}
${}$
\begin{enumerate}
	\item Let $\triv\in\Kei$ be trivial. Then every sub-kei $\triv'\le \triv$ is trivial.
	\item Let $\triv,\triv'\in\Kei$ be trivial. Then any function $f\colon \triv'\to \triv$ is a kei morphism.
\end{enumerate}
\end{lemma}

Recall in \cite{WEBSITE:DavisSchlankHilbert2023} the computation of the Hilbert polynomial of $\triv_a$:
$$P_{\triv_ \powa}(k)={\powa +k-1\choose \powa-1}=\frac{k^{\powa-1}}{(\powa-1)!}+O(k^{\powa-2}).$$

\subsection{Interpretation of $\triv_\powa$-colorings}

\begin{proposition}
\label{Prop: Trivial colorings factor through pi_0}
Let $0\le \powa\in\NN$ and let $n\in\NNN$. Then
$$\col_{\triv_\powa}(n)=\powa^{\omega(n)}.$$
\end{proposition}
\begin{proof}
This follows directly from \cref{Prop: Trivial colorings factor through pi_0 II}:
$$\col_{\triv_\powa}(n)= \left| \Hom_{\Kei}^\cont(\kei_n,\triv_\powa)\right| =\left|\triv_\powa^{ \{p|n\;\textnormal{prime}\}}\right|=\powa^{\
omega(n)}.$$
\end{proof}

\subsection{Proof of Main Conjecture for $\k=\triv_\powa$.}

\begin{lemma}
\label{Lma: Triv proof aux lemma}
	Let $\powa\in\NN$. The infinite product
	$$\prod_p \frac{(1+p^{-1})^\powa}{(1+ \powa p^{-1})}$$
	converges.
\end{lemma}
\begin{proof}
For $\powa =0,1$, the claim is obvious. Assume that $\powa\ge 2$. For all prime $p$,
$$1+ \powa p^{-1}\le (1+p^{-1})^ \powa =1+ \powa p^{-1}+\sum_{i=2}^ \powa{\powa\choose i}p^{-i}\le 1+ \powa p^{-1}+2^ \powa p^{-2}\le (1+ \powa p^{-1})(1+2^ \powa p^{-2}).$$
For $p\ge 4^ \powa $, all but finitely many, we have
$$1\le \frac{(1+p^{-1})^ \powa}{(1+ \powa p^{-1})}\le 1+2^ \powa p^{-2} \le 1+p^{-3/2}.$$
The convergence of
$$\prod_p(1+p^{-3/2})=\zeta(3)^{-1}\zeta(3/2)$$
guarantees the convergence of 
$$\prod_p \frac{(1+p^{-1})^ \powa}{(1+ \powa p^{-1})}.$$
\end{proof}

\begin{proposition}
\label{Prop: Main Conj for T_a}
Let $1\le \powa\in\NN$. Then $\col_{\triv_ \powa}$ has generic summatory type
$$\col_{\triv_ \powa}\in\GenType{\powa-1}{\textstyle{\frac1{(\powa-1)!}}}.$$
\end{proposition}
\begin{proof}
In \cref{Prop: Trivial colorings factor through pi_0} we saw that 
$$\col_{\triv_ \powa}(n)= \powa ^{\omega(n)}$$
for all $n\in\NNN$. Let $s\in\NNN$. The function
$\mu_s^2(n) \powa ^{\omega(n)}\colon\NN\to\RR$ is multiplicative, and satisfies
$$0\le \mu_s^2(n) \powa ^{\omega(n)} =\mu_s^2(n)\uno^{* \powa}(n)\le \uno^{* \powa}(n)$$
for all $n\in\NN$. 
By \cref{prop: Tenenbaum 3.12 II}, 
$$\lim_{X\to\infty}\frac{\counter_{\triv_ \powa,s}(X)}{X(\log X)^{\powa-1}}
=\frac1{(\powa-1)!} \prod_{p}(1-p^{-1})^ \powa \powa\sum_\nu\frac{\mu_s^2(p^{\nu}) \powa ^{\omega(p^\nu)}}{p^\nu}=$$
$$=\frac1{(\powa-1)!} \prod_{p|s}(1-p^{-1})^ \powa\cdot \prod_{p\nmid s}(1-p^{-1})^ \powa(1+ \powa p^{-1}).$$
Therefore 
$$c_s(\triv_ \powa)=\gamma_s^{1-\powa}\lim_{X\to\infty}\frac{\avg_{\triv_ \powa,s}(X)}{\log^{\powa-1}(X)}=
\gamma_s^{-\powa}\lim_{X\to\infty}\frac{\counter_{\triv_ \powa,s}(X)}{X\log^{\powa-1}(X)} =$$
$$=\frac{1}{(\powa-1)!}\cdot \prod_{p|s}\frac {(1-p^{-1})^ \powa}{(1-p^{-1})^ \powa}\cdot \prod_{p\nmid s} \frac{(1-p^{-1})^ \powa(1+ \powa p^{-1})}{(1-p^{-2})^{\powa}}=$$
$$=\frac1{(\powa-1)!}\prod_{p\nmid s}\frac{1+ \powa p^{-1}}{(1+p^{-1})^ \powa}.$$
By \cref{Lma: Triv proof aux lemma}, this product converges. Moreover,
it follows that
$$\lim_{s\in\NNN}c_s(\triv_ \powa)=\frac 1{(\powa-1)!}\lim_{s\in\NNN}\prod_{p\nmid s}\frac{1+ \powa p^{-1}}{(1+p^{-1})^ \powa}=\frac1{(\powa-1)!}.$$
Therefore $\col_{\triv_ \powa}\colon\NNN\to\RR$ has summatory type $(\powa-1,\frac1{(\powa-1)!})$.
\end{proof}

\section{The Joyce kei $\joyce$}
\label{Section: Joyce kei proof}
In this section we recall the kei $\joyce$ that was introduced in \cite[\textsection 6]{ARTICLE:Joyce1982}. We will address here not only colorings of $n\in\NNN$ by $\joyce$, but by $\joyce\sqcup \triv_ \powa $ where $\triv_\powa\in\Kei^\fin$ is the trivial kei on $\powa\in\NN$ elements.

\begin{definition}
\label{Def: Joyce kei}
	The kei $\joyce$ has underlying set and structure
	$$\joyce=\{x_+,x_-,y\}\;\;,\;\;\;\;\begin{tabular}{|cc||c|c|c|}
\hline
\multicolumn{2}{|c||}{\multirow{2}{*}{$\qq zw$}}&\multicolumn{3}{|c|}{$z$}\\
\cline{3-5}
&&$x_+$&$x_-$&$y$\\
\hline
\hline
\multirow{3}{*}{$w$}&\multicolumn{1}{|c||}{$x_+$}&
$x_+$&$x_+$&$x_-$\\
\cline{2-5}
& \multicolumn{1}{|c||}{$x_-$}&
$x_-$&$x_-$&$x_+$\\
\cline{2-5}
& \multicolumn{1}{|c||}{$y$}&
$y$&$y$&$y$\\
\hline
\end{tabular}\;\;.$$
The kei $\joyce$ satisfies
$$\Inn(\joyce)=\{\Id_\joyce,\varphi_y\}\simeq\ZZ/2\ZZ.$$
For all $z\in \joyce$, $\varphi_z\neq\Id$ iff $z=y$.
The proper sub-quandles of $\joyce$ are
$$\emptyset,\{x_+\}, \{x_-\} , \{y\} , \{x_+,x_-\},$$
all trivial.
\end{definition}

\begin{lemma}
\label{Lma: Joyce pi0 is morphism}
	Let $\{x_0,y_0\}=\triv_2$ be a trivial kei. Then the map
	$$\eta_0\colon \joyce\to \{x_0,y_0\}\;\;,\;\;\;\;x_\pm\mapsto x_0\;,\;\;y\mapsto y_0$$
	is a kei morphism.
\end{lemma}

Recall in \cite{WEBSITE:DavisSchlankHilbert2023} the computation of the Hilbert polynomial of $\joyce$:
$$P_{\joyce}(k)=2k+1,$$
and more generally, for $0\le \powa\in\NN$,
$$P_{\joyce\sqcup \triv_ \powa}(k)=2{k+ \powa +1\choose \powa +1}-{k+ \powa\choose \powa}=\frac2{(\powa +1)!}k^{\powa +1}+O(k^ \powa).$$

\subsection{Interpretation of $\joyce$-colorings}

\begin{proposition}
\label{Prop: arithmetic computation: Joyce colorings}
	Let $n\in\NNN$. Then
	$$\col_{\joyce}(n)=\sum_{ab|n}\Jac ba,$$
	where $\Jac ba$ is the Kronecker symbol.
\end{proposition}
\begin{proof}
Consider the kei morphism $\eta_0\colon \joyce\to \{x_0,y_0\}\simeq\triv_2$ from \cref{Lma: Joyce pi0 is morphism}. By \cref{Lem: coloring preserves limits} and \cref{Prop: Trivial colorings factor through pi_0 II} there is a map
$$\Col_{\joyce}(n)\xrightarrow{\eta_0\circ-}\Col_{\triv_2}(n)\simeq\triv_2^{\{p|n\;\textnormal{prime}\}}\simeq\{(n_x,n_y)\in\NN^2\;|\;n_xn_y=n\}.$$
For $n_1,n_2\in\NN$ s.t. $n_1n_2=n$, we denote by $\Col_{\joyce}(n)_{(n_1,n_2)}\subseteq\Col_{\joyce}(n)$ the fiber 
$$\Col_{\joyce}(n)_{(n_1,n_2)}=\left\{f\colon \kei_n\to \joyce\;|\; \forall \p\in\kei_n\;,\;\; f(\p)\in\begin{cases}
	\{x_\pm\}&\p|n_1\\
	\{y\}&\p|n_2
\end{cases}\;\;\;\;\right\}$$
over $(n_1,n_2)$, and by 
$$\col_{\joyce}(n)_{(n_1,n_2)}:=\left|\Col_{\joyce}(n)_{(n_1,n_2)}\right|\in\NN,$$
s.t.
$$\col_\joyce(n)=\sum_{n_1n_2=n}\col_\joyce(n)_{(n_1,n_2)}.$$
We show that for all $n_1,n_2\in\NN$ s.t. $n_1n_2=n$,
$$\col_\joyce(n)_{(n_1,n_2)}=\prod_{p|n_1}\left(1+\Jac p{n_2}\right)=
\begin{cases}
	0&,\; \exists\; p|n_1\; \Jac p{n_2}=-1\\
	2^{\omega(n_1)}&,\; \forall\; p|n_1\; \Jac p{n_2}=1
\end{cases}\;\;.$$
If $n_1=n$, then for all $p|n$ we have $\Jac p1=1$, and indeed
$$\col_\joyce(n)_{(n,1)}=\left|\Hom_{\Kei}^{\textnormal{cont}}(\kei_n,\{x_\pm\})\right|=\col_{\triv_2}(n)=2^{\omega(n)}.$$
If $n_1=1$, then $\Jac pn=1$ vacuously for all $p|1$, and
$$\col_\joyce(n)_{(1,n)}=\left|\Hom_{\Kei}^\cont(\kei_n,\{y\})\right|=1=2^{\omega(1)}.$$
Let $n_1,n_2\neq 1$ with $n_1n_2=n$. We denote by $\rho_{n,n_2}$ the map
$$\rho_{n,n_2}\colon \G_n=\Gal(\F_n/\QQ)\twoheadrightarrow \Gal(\L_{n_2}/\QQ)\simeq\ZZ/2\ZZ\simeq\Inn(\joyce).$$
where $\L_{n_2}=\QQ(\sqrt{n_2^*})$ is the unique quadratic field in $\F_n$ that ramifies at precisely $n_2$.
Any $f\in\Col_\joyce(n)_{(n_1,n_2)}$ is necessarily surjective, thus 
by \cref{Prop: pro-fin concise augkei lift surj kei morphism}, lifts uniquely to 
$$(f, \rho_{\A_n,f} )\colon (\kei_n,\G_n,\m_n)\to(\joyce,\Inn(\joyce),\varphi)$$
in $\Pro\AKei^\fin$.
For all $\p\in\kei_n$,
$$\rho_{\A_n,f}(\m_\p)=\varphi_{f(\p)}\neq 1\;\iff\; \rho_{\A_n,f}(\m_\p) =\varphi_{y}\;\iff\;\p\in f^{-1}(y)\;\iff\; \p|n_2.$$
Therefore $\rho_{\A_n,f}\colon \G_n\twoheadrightarrow\Inn(\joyce)$ corresponds to $\L_{n_2}\subseteq\F_n$:
$$\rho_{\A_n,f}=\rho_{n,n_2}
\;\;,\;\;\;\; \ker\rho_{\A_n,f}=\Gal(\F_n/\L_{n_2}).$$
It also follows from \cref{Prop: profinite augkei statements} that  every $f\in\Col_\joyce(n)_{(n_1,n_2)}$ factors through the quotient
$$\k :=\ker\rho_{n,n_2}\backslash\kei_n \simeq \Spec\left(\O_{\K_{n_2}}\otimes\ZZ/n\ZZ\right) \to \joyce$$
in $\Kei^\fin$. 
The kei structure on $\k$ is defined by the induced augmentation with $G:=\Gal(\K_{n_2}/\QQ)$, satisfying $\qq\q\p=\m_{\q}(\p)$ for all $\p,\q\in \k$, where $\m_\q=1_G$ iff $\q|n_1$.
Then for all $\p,\q\in \k$ over respective rational primes $p,q$,
$$\qq\q\p\neq\p\;\iff\;q|n_2,\;p|n_1,\;\textnormal{ and }\; \Jac{n_2^*}p=\Jac p{n_2}=1.$$
Suppose there exists a prime $p|n_1$ s.t. $\Jac p{n_2}=-1$. Let $\p,\q\in\k$ be s.t. $\p|p$ and $\q|n_2$ ($n_2\neq 1$). Then on  one hand, $\qq\q\p=\p$. On the other hand, for any  $f\in\Col_\joyce(n)_{(n_1,n_2)}$,
$$f(\p)\in\{x_\pm\}\;\;\Longrightarrow\;\;\qq{f(\q)}{f(\p)}=\qq y{f(\p)}\neq f(\p).$$
This is a contradiction, therefore $\Col_\joyce(n)_{(n_1,n_2)}=\emptyset$ in these conditions:
$$\exists\; p|n_1\; \Jac p{n_2}=-1\;\Longrightarrow\;\col_\joyce(n)_{(n_1,n_2)}=0.$$
If on the other hand $\Jac p{n_2}=1$ for all $p|n_1$, then $f$ is determined by bijections
$$\Spec(\O_{\K_{n_2}}\otimes\FF_p)\iso \{x_\pm\}$$
for all $p|n_1$, each independent of the other. Therefore
$$\forall\; p|n_1\; \Jac p{n_2}=1\;\Longrightarrow\;\col_\joyce(n)_{(n_1,n_2)}=2^{\omega(n_1)}.$$
It follows that for any $n_1,n_2$ s.t. $n_1n_2=n$,
$$\col_\joyce(n)_{(n_1,n_2)}=\prod_{p|n_1} \left(1+\Jac p{n_2}\right) =\sum_{b|n_1}\Jac b{n_2}.$$
In total,
$$\col_\joyce(n)=\sum_{n_1n_2=n}\col_\joyce(n)_{(n_1,n_2)}= \sum_{n_1n_2=n}\sum_{b|n_1}\Jac b{n_2}= \sum_{ab|n}\Jac ba.$$
\end{proof}

\begin{proposition}
\label{Prop: arithmetic computation: T sqcup Joyce colorings}
	Let $0\le \powa\in\NN$. Then for all $n\in\NNN$,
	$$\col_{\triv_\powa\sqcup \joyce}(n) =\sum_{abc=n} (\powa+1)^{\omega(c)}\Jac ba .$$
\end{proposition}
\begin{proof}
	Let $n\in\NNN$. By \cref{Prop: Trivial colorings factor through pi_0},
	$$\col_{\triv_\powa}=\powa^{\omega(n)}. $$
	By \cref{Prop: arithmetic coloring by disjoint union} and \cref{Lma: divisor arithmetic that doesnt really need wrapping},
$$\col_{\triv_\powa\sqcup \joyce}(n)=\sum_{n_1n_2=n}\col_{\triv_ \powa}(n_1) \col_ \joyce(n_2) =\sum_{n_1n_2=n} \powa ^{\omega(n_1)}\sum_{ab|n_2} \Jac ba =$$
$$= \sum_{abc=n} \Jac ba \sum_{n_1|c}\powa^{\omega(n_1)} =\sum_{abc=n} (\powa+1)^{\omega(c)}\Jac ba .$$
\end{proof}

\subsection{Proof of the Main Conjecture for $\k=\joyce\sqcup \triv_\powa$.}

\begin{proposition}
\label{Prop: J_+ to error term}
	Let $0\le \powa\in\NN$, let $s\in\NNN$ and let $X\ge 1$. Then
	$$\counter_{\joyce\sqcup \triv_\powa,s}(X)=
	2\counter_{\triv_{\powa+2},s}(X)-\counter_{\triv_{\powa+1},s}(X) +
	\sum_{n\le X}\mu_s^2(n)(\powa+1)^{\omega(n)}T_{sn}(X/n),$$
	where for $v\in\NNN$ and $1\le Y\in\RR$,
	$$T_{v}(Y)=\sum_{\substack{a,b\ge2\\ab\le Y}}\mu_{v}^2(ab) \Jac ba.$$
\end{proposition}
\begin{proof}
	Let $n\in\NNN$. By \cref{Lma: divisor arithmetic that doesnt really need wrapping}, we write
	$$\col_{\joyce\sqcup \triv_\powa}(n)=\sum_{abc=n}(\powa+1)^{\omega(c)}\Jac ba=$$
	$$=\left(\sum_{\substack{abc=n\\a=1}}+ \sum_{\substack{abc=n\\b=1}}-\sum_{\substack{abc=n\\a=b=1}}+ \sum_{\substack{abc=n\\a,b\ge2}}\right)(\powa+1)^{\omega(c)}\Jac ba=$$
$$=2\sum_{c|n}(\powa+2)^{\omega(c)}-(\powa+1)^{\omega(n)}+\sum_{\substack{abc=n\\a,b\ge2}}\powa^{\omega(c)}\Jac ba.$$
Therefore by \cref{Prop: Trivial colorings factor through pi_0},
$$\col_{\joyce\sqcup \triv_\powa}(n)= 2\col_{\triv_{\powa+2}}(n)-\col_{\triv_{\powa+1}}(n)+ \sum_{\substack{abc=n\\a,b\ge2}}(\powa+1)^{\omega(c)}\Jac ba. $$
Thus for $X\ge 1$, summing over $1\le n\le X$ we have
$$\counter_{\joyce\sqcup \triv_\powa,s}(X)- 2\counter_{\triv_{\powa+2},s}(X)+\counter_{\triv_{\powa+1},s}(X)=  \sum_{1\le n\le X}\mu_s^2(n)\sum_{\substack{abc=n\\a,b\ge2}}(\powa+1)^{\omega(c)}\Jac ba= $$
$$\sum_{\substack{abc\le X\\a,b\ge2}}\mu_s^2(abc)(\powa+1)^{\omega(c)}\Jac ba= \sum_{1\le c\le X}\mu_s^2(c)(\powa+1)^{\omega(c)} \sum_{\substack{ab\le X/c\\a,b\ge2}} \mu_{sc}^2(ab)\Jac ba
= $$
$$= \sum_{1\le c\le X}\mu_s^2(c)(\powa+1)^{\omega(c)}T_{sc}(X/c).$$
\end{proof}

\begin{proposition}
\label{Prop: Technical statement for J_+ proof}
	Let $0\le\powa\in\NN$ and let $s\in\NNN$. Then
	$$\counter_{\joyce\sqcup \triv_\powa,s}(X)- 2\counter_{\triv_{\powa+2},s}(X)+\counter_{\triv_{\powa+1},s}(X) = o(X\log ^{\powa+1}(X)).$$
\end{proposition}
\begin{proof}
	By \cref{Prop: J_+ to error term}, 
	$$\counter_{J\sqcup T_a,s}(X)-2\counter_{T_{a+2},s}(X)+\counter_{T_{a+1},s}(X) = \sum_{n\le X}\mu_s^2(n)(\powa+1)^{\omega(n)}T_{sn}(X/n),$$
	which equals
	\begin{equation}
	\label{Eqn: J-col split T_su(X/u) sum J_+ case}
		\left(\overbrace{\sum_{1\le n\le \frac XY}}^{(I)}+\overbrace{\sum_{\frac XY<n\le X}}^{(II)}\right) \mu_s^2(n)(\powa+1)^{\omega(n)}T_{sn}(X/n)
	\end{equation}
	for some $1<Y<X$ to be determined later. 
	By \cref{Prop: bound for T_s(X)}, there exists $C>0$ s.t. the first summand in (\ref{Eqn: J-col split T_su(X/u) sum J_+ case}) is bounded by
	$$|(I)|\le \sum_{1\le n\le \frac XY}\mu_s^2(n)(\powa+1)^{\omega(n)}|T_{sn}(X/n)|\le$$
	$$\le C\cdot \sum_{n\le \frac XY} \mu_s^2(n)(\powa+1)^{\omega(n)}\cdot 2^{\frac12\omega(sn)}(X/n)^{\frac78}(\log (X/n)+1)^{\frac12}\le  $$
	$$\le C\cdot 2^{\frac12\omega(s)} X^{\frac78}(\log X+1)^{\frac12} \sum_{n\le \frac XY} \mu_s^2(n)(\powa+1)^{\omega(n)}\cdot 2^{\frac12\omega(n)}n^{-\frac78}.$$
	Let $\powb=\lceil (\powa+1)\sqrt2\rceil$. Then by \cref{Lma: basic counting lemma 2},
	$$|(I)|\le C\cdot 2^{\frac12\omega(s)} X^{\frac78}(\log X+1)^{\frac12} \sum_{n\le \frac XY} \mu_s^2(n)\powb^{\omega(n)}n^{-\frac78} \le $$
	$$\le C\cdot 2^{\frac12\omega(s)} X^{\frac78}(\log X+1)^{\frac12}\cdot 8(X/Y)^{\frac18}(1+\log (X/Y))^{\powb-1} \le $$
		$$\le 8C\cdot 2^{\frac12\omega(s)} X \cdot Y^{-\frac18}(\log X+1)^{\powb}.$$
		For $X\gg 0$ we may take 
		$$Y=(\log X+1)^{8\powb}\le X,$$
		for which we then have
		$$|(I)|\le C\cdot 2^{\frac12\omega(s)}X.$$
	We turn to $(II)$, the second summand in (\ref{Eqn: J-col split T_su(X/u) sum J_+ case}):
	$$(II)= \sum_{\frac XY<n\le X}\mu_s(n)^2(\powa+1)^{\omega(n)}T_{sn}(X/n) =$$
	$$=\sum_{\frac XY<n\le X}\mu_s^2(n)\uno^{*(\powa+1)}(n)\sum_{\substack{a,b\ge 2\\ab\le X/n}}\mu_{sn}^2(ab)\Jac ba,$$
	which is bounded by
	$$|(II)|\le 
	\sum_{\frac XY<n\le X}\uno^{*(\powa+1)}(n)\sum_{\substack{a,b\ge 1\\ab\le X/n}}1\le\sum_{\substack{a,b\ge1\\ab\le Y}}\sum_{n\le \frac X{ab}}\uno^{*(\powa+1)}(n).$$
	By \cref{Lma: Piltz lazy} there exists $C'>0$ s.t.
	$$|(II)|\le C'\sum_{\substack{ab\le Y}}\frac X{ab}\left(\log^\powa(X/ab)+1\right)\le$$
	$$\le C'X\left(\log^\powa(X)+1\right) \sum_{\substack{ab\le Y}}\frac 1{ab}\le C'X\left(\log^\powa(X)+1\right) \cdot C''\log^2 (Y)=$$
	$$=C'C''X\left(\log^\powa (X)+1\right) (8\powb\log (\log X+1))^2.$$
	Hence
	$$\left| \counter_{\joyce\sqcup \triv_\powa,s}(X)- 2\counter_{\triv_{\powa+2},s}(X)+\counter_{\triv_{\powa+1},s}(X)\right|= |(\ref{Eqn: J-col split T_su(X/u) sum J_+ case})|\le |(I)|+|(II)|=$$
	$$=O_s(X)+O_\powa\left(X\log^\powa(X)(\log\log X)^2\right)=o(X\log^{\powa+1}(X)).$$
\end{proof}

\begin{proposition}
\label{Prop: Main Conj for J_+}
	Let $0\le \powa\in\NN$. Then $\col_{\joyce\sqcup \triv_\powa}$ is of generic summatory type $(\powa +1,\frac2{(\powa +1)!})$:
	$$\col_{\joyce\sqcup \triv_\powa}\in\GenType{\powa +1}{\textstyle{\frac2{(\powa +1)!}}}$$
\end{proposition}
\begin{proof}
	Let $s\in\NNN$. Following \cref{Prop: Technical statement for J_+ proof},
		$$\lim_{X\to\infty}\frac{\counter_{\joyce\sqcup \triv_\powa,s}(X)- 2\counter_{\triv_{\powa+2},s}(X)+\counter_{\triv_{\powa+1},s}(X)}{X\log^{\powa+1}(X)}=0.$$
By \cref{Prop: Main Conj for T_a}, $\col_{\triv_{\powa+1}}\in\GenType {\powa}{\frac1{\powa!}}$ and $\col_{\triv_{\powa+2}}\in\GenType {\powa+1}{\frac1{(\powa+1)!}}$. For all $s\in\NNN$ we therefore have $\counter_{\triv_{\powa+1},s}(X)=o\left(X\log^{\powa+1}(X)\right)$ and
	$$\lim_{X\to\infty}\frac{\counter_{\joyce\sqcup \triv_\powa,s}(X)}{X\log^{\powa+1}(X)}= 2\lim_{X\to\infty}\frac{\counter_{\triv_{\powa+2},s}(X)}{X\log^{\powa+1}(X)}.$$
	Therefore
	$$c_s(\joyce\sqcup \triv_\powa):= \gamma_s^{-\powa-1}\lim_{X\to\infty}\frac{\avg_{\joyce\sqcup \triv_\powa,s}(X)}{\log^{\powa+1}(X)}=$$
	$$=2\gamma_s^{-\powa-1}\lim_{X\to\infty}\frac{\avg_{\triv_{\powa+2},s}(X)}{\log^{\powa+1}(X)} = 2c_s(\triv_{\powa+2})>0,$$
	and
	$$\lim_{s\in\NNN}c_s(\joyce\sqcup \triv_\powa)= 2\lim_{s\in\NNN}c_s(\triv_{\powa+2})=\frac2{(\powa +1)!}.$$
	Thus
	$$\col_{\joyce\sqcup \triv_\powa}\in\GenType{\powa +1}{\textstyle{\frac 2{(\powa +1)!}}}.$$
\end{proof}

\section{The Dihedral kei $\RGB$}
\label{Section: dihedral 3 kei proof}
\begin{definition}
	Let $D_3\in\Grp$ denote the dihedral group of a triangle. and let $\sigma\colon D_3\to\{\pm1\}$ denote the homomorphism
	$$\sigma\colon D_3\simeq (\ZZ/3\ZZ)\rtimes\{\pm1\}\to\{\pm1\}.$$
	By $\RGB\subseteq D_3$ we denote the set of reflections in $D_3$:
	$$\RGB:=D_3\setminus\ker \sigma =
	\sigma^{-1}(-1)
	.$$
	The set $\RGB$ is comprised of involutions and is closed under conjugation. Therefore $\RGB$ is a kei with structure given by conjugation: $\qq gh:=ghg^{-1}$.
\end{definition}

Recall in \cite{WEBSITE:DavisSchlankHilbert2023} the computation of the Hilbert polynomial of $\RGB $:
$$P_{\RGB}(k)=6.$$

\subsection{Interpretation of $\RGB $-colorings}

\begin{lemma}
\label{Lma: Kei dom Col to Group dom morphisms}
	Let $\k\in\Kei^\fin$ be such that
	\begin{enumerate}
		\item the map $\varphi\colon\k\to\Inn(\k)$ is an injection, and
		\item for every proper sub-kei $\k'<\k$, the set $\{\varphi_x\;|\;x\in\k'\}\subseteq \Inn(\k)$ generates a proper subgroup of $\Inn(\k)$:
		$$\langle\{\varphi_{\k,x} \;|\;x\in\k'\}\rangle<\Inn(\k).$$
	\end{enumerate}
Then for all $n\in\NNN$ there is a bijection
$$\Col_\k^\dom(n)\simeq\{\rho\colon\G_n\twoheadrightarrow\Inn(\k)\;|\;(\rho\circ\m_n)(\kei_n)\subseteq \varphi(\k)\subseteq\Inn(\k)\}.$$
\end{lemma}
\begin{proof}
	Let $n\in\NNN$ and let $f\colon\kei_n\twoheadrightarrow \k$. Then by \cref{Prop: pro-fin concise augkei lift surj kei morphism}, $f$ extends uniquely to a morphism $(f,\rho)\colon \A_n\twoheadrightarrow \a_\k$.
	The map $\rho\colon\G_n\twoheadrightarrow\Inn(\k)$ satisfies
	$$\rho(\m_\p)=\varphi_{f(\p)}\in\varphi(\k).$$
	We therefore have a map
	$$\Col_\k^\dom(n)
	\to \{\rho\colon\G_n\twoheadrightarrow\Inn(\k)\;|\;\rho\circ\m_n=\varphi\circ f\}\subseteq \Hom_\Grp^\cont(\G_n,\Inn(\k))^\surj.$$
Conversely, let $\rho\colon\G_n\twoheadrightarrow\Inn(\k)$ be such that $(\rho\circ\m_n)(\kei_n)\subseteq \varphi(\k)\subseteq\Inn(\k)$. Since $\varphi$ is an embedding, there is a unique continuous map $f\colon \kei_n\to\k$ s.t.
$$\rho\circ \m_n=\varphi\circ f\colon \kei_n\to \Inn(\k).$$
	This $f$ is a kei morphism: for all $\p,\q\in\kei_n$,
	$$\varphi_{f(\qq \p\q)}= \rho(\m_{\qq \p\q})=\rho(\m_\p \m_\q \m_\p^{-1})=\rho(\m_\p) \rho(\m_\q) \rho(\m_\p)^{-1} =$$
	$$=\varphi_{f(\p)} \varphi_{f(\q)} \varphi_{f(\p)}^{-1} = \varphi_{\qq{f(\p)}{f(\q)}}.$$
	It follows that $f(\qq \p\q)=\qq{f(\p)}{f(\q)}$ because $\varphi$ is an embedding. This $f$ is also surjective: Let $\k'=f(\kei_n)\le \k$. Then
	$$\Inn(\k)=\rho(\G_n)=\langle\{\rho(\m_\p)\;|\;\p\in\kei_n\}\rangle= \langle\{\varphi_{f(\p)}\;|\;\p\in\kei_n\}\rangle= \langle\{\varphi_{\k,x}\;|\;x\in \k'\}\rangle.$$
	It follows from the assumption on $\k$ that $f(\kei_n) =\k$, that is, $f$ is surjective. Hence we obtain a map
	$$\{\Hom_\Grp^\cont(\G_n,\Inn(\k))^\surj\;|\;(\rho\circ\m_n)(\kei_n)\subseteq \varphi(\k)\}\to
	\Hom_\Kei^\cont(\kei_n,
	\k)^\surj=$$
	$$=\Col_\k^\dom(n).$$
	These constructions are inverse to one another.
\end{proof}

\begin{proposition}
\label{Prop: RGB with D3 is AK}
	Let $\iota\colon \RGB\hookrightarrow D_3$ denote the inclusion. Then $(\RGB,D_3,\iota)$ is an augmented kei, and
	$$(\RGB,D_3,\iota) \simeq \a_{\RGB}\in\AKei.$$
\end{proposition}
\begin{proof}
	 The set $\RGB$ is a conjugacy class of involutions in $D_3$, hence $\a:=(\RGB,D_3,\iota)\in\AKei$. This $\a$ is concise because $\langle\RGB\rangle=D_3$. By \cref{Prop: concise augkei lift surj kei morphism}, $\Id_{\RGB}$ lifts to a morphism
	$$(\Id_{\RGB},\rho)\colon \a\twoheadrightarrow \a_{\RGB}.$$
	Suppose $g\in\ker \rho$. Then $gxg^{-1}=x$ for all $x\in \RGB$. Since $\RGB$ generates $D_3$, it follows that $g\in Z(D_3)=\{1\}$. Hence $\rho$ is a group isomorphism, and
	$$(\Id_{\RGB},\rho)\colon(\RGB,D_3,\iota) \iso \a_{\RGB}.$$
\end{proof}

\begin{corollary}
\label{Cor: R_3 is nice}
The kei $\RGB$ satisfies the conditions of \cref{Lma: Kei dom Col to Group dom morphisms}
\end{corollary}
\begin{proof}
	Following \cref{Prop: RGB with D3 is AK}, the map $\varphi\colon\RGB\to\Inn(\RGB)$ is an embedding.
	Moreover, the non-empty proper sub-keis of $\RGB$ are the singletons, and for all $x\in R_3$, 
	$$|\langle \varphi_{\RGB,x}\rangle|=2<6= |\Inn(\RGB)|.$$
\end{proof}

An \emph{order} is an integral domain $R$ that is free as a $\ZZ$-module. For order $R$, the ring $R_\QQ=\QQ\otimes_\ZZ R$ is a number field. As such $R\subseteq \O_{R_\QQ}$, where $\O_L$, the ring of integers in a number field $L$, is the maximal order in $L$. The \emph{discriminant} $\Disc(R)$ of an order $R$ is the volume of $R$ as a lattice with the trace form:
$$\Disc(R)=\det([\textnormal{Tr}(x_ix_j)]_{i,j})\in\ZZ,$$
where $\{x_i\}$ is a basis for $R$ as a $\ZZ$-module, and $\textnormal{Tr}(x)\in\ZZ$ is the trace of multiplication by $x$. The discriminant of a number field $L$ is defined to be
$$\Disc(L):=\Disc(\O_L).$$
A \emph{fundamental discriminant} is the discriminant of a quadratic number field. A quadratic number field $L$ can be written uniquely as $L=\QQ(\sqrt{m})$ with $m\in\ZZ$ square-free. For such $L=\QQ(\sqrt{m})$,
$$\Disc(L)=\begin{cases}
	m	&,\;m=1\mod 3\\
	4m	&,\;m=2,3\mod 4
\end{cases}\;\;.$$
Every square-free integer $m\in\ZZ$ can be written uniquely as $m=u'n=u n^*$ with $n\in\NNodd$ and $u,u'\in\{\pm1,\pm2\}$. Hence every fundamental discriminant is of the form
$$\Delta=u\cdot n^*\;\;,\;\;\;\;n\in\NNodd\;,\;u\in\{1,-4, 8,-8\}.$$

\begin{definition}
	Let $0\neq D\in\ZZ$. By $H_3(D)$ we denote the set of isomorphism classes of cubic orders $R$ with discriminant $\Disc(R)=D$:
	$$H_3(D)=\left\{R/\ZZ\;|\;\substack{rk_\ZZ(R) =3\;\textnormal{ and }\\ \;\Disc(R)=D}\right\}/_\sim,$$
	and by $h_3(D) :=|H_3(D)|\in\NN$ its cardinality.
\end{definition}

 The following is a well-known fact about cubic number fields, due to \cite{ARTICLE:Hasse1930}:
\begin{lemma}
\label{Hasse 1930}
	Let $L$ be a cubic number field. Then $\Disc(L)=df^2$, where $d\in\ZZ$ is a fundamental discriminant, and $f\in\ZZ$ is such that for every prime $p$, $p|f$ iff $p$ is totally ramified in $L$.
\end{lemma}

\begin{proposition}
\label{Prop: cubic with quadratic discriminant are maximal orders}
	Let $\Delta\in\ZZ$ be a fundamental discriminant. Then
	$$H_3(\Delta)= \left\{\O_L/\ZZ\;|\;\substack{[L:\QQ] =3\;\textnormal{ and }\\ \;\Disc(L)=\Delta}\right\}/_\sim.$$
\end{proposition}
\begin{proof}
	Let $R\in H_3(\Delta)$. Denote by $L:=\QQ\otimes_\ZZ R$, a cubic number field. Let $\O_L$ be the maximal order in $L$. Then
	$$\Delta=\Disc(R)=[\O_L:R]^2\Disc(\O_L)= [\O_L:R]^2\Disc(L).$$
	There exists $f\in\ZZ$ s.t.
	$$\Disc(L)=df^2,$$
	where $d$ is a fundamental discriminant. Thus
	$$\Delta=[\O_L:R]^2f^2d.$$
	 No two fundamental discriminants differ by a perfect square, hence $R=\O_L$ is the maximal order in $L$, and
	 $\Disc(L)=\Delta$. Hence
	$$H_3(\Delta)= \left\{\O_L/\ZZ\;|\;\substack{[L:\QQ] =3\;\textnormal{ and }\\ \;\Disc(L)=\Delta}\right\}/_\sim.$$
\end{proof}

\begin{proposition}
\label{Prop: computation of RGB for n>1}
	Let $1<n\in\NNN$. Then
	$$\col_{\RGB}(n)=3 \left|Cl(\L_n)\otimes_\ZZ \FF_3\right| 
	=\begin{cases}
	3+6h_3(n^*)&,\;2\nmid n\\
	3+6h_3(4n^*)&,\;2|n
	\end{cases}$$
	$$=\begin{cases}
	3+6h_3(n)&,\;n=1\mod 4\\
	3+6h_3(-n)&,\;n=3\mod 4\\
	3+6h_3(4n)&,\;n=2\mod 8\\
	3+6h_3(-4n)&,\;n=6\mod 8
	\end{cases}$$
	where $Cl(\L_n)$ is the class group of $\L_n=\QQ(\sqrt{n^*})$.
\end{proposition}
\begin{proof}
	Since $n>1$, $\kei_n\neq\emptyset$. The image of a coloring $f\colon\kei_n\to \RGB $ is therefore a nonempty sub-kei of $\RGB $. The proper nonempty sub-keis of $\RGB$ are the three singletons in $\RGB$. Therefore
	$$\col_{\RGB}(n)=3\col_{\triv_1}(n)+\col_{\RGB}^{\dom}(n)=
	3+\col_{\RGB}^{\dom}(n).$$
	Let $\eta\colon\G_n\to\{\pm1\}$ denote the homomorphism
	$$\eta\colon \G_n=\Gal(\F_n/\QQ)\twoheadrightarrow\Gal(\L_n/\QQ)\iso\{\pm1\}.$$
	This is the unique morphism that maps $\m_\p\mapsto -1$ for all $\p\in\kei_n$.
	Since $\RGB=\sigma^{-1}(-1)\subseteq D_3$, then by \cref{Lma: Kei dom Col to Group dom morphisms} and \cref{Cor: R_3 is nice}, we have
	$$\Col_{\RGB}^\dom(n)\simeq\{\rho\colon\G_n\twoheadrightarrow D_3\;|\;(\rho\circ\m_n)(\kei_n)\subseteq \RGB\}= \Hom_{/\eta}(\G_n,D_3)^\surj $$
	where $\Hom_{/\eta}(\G_n,D_3)^\surj$ is the set of surjective lifts
	$$\begin{tikzcd}
	\G_n\ar[r,two heads, dashed,"\rho"]\ar[rd,"\eta"']&D_3\ar[d,"\sigma"]\\
	&\{\pm1\}
	\end{tikzcd}$$
	The $D_3$-action on $\Hom_{/\eta}(\G_n,D_3)^\surj$ via conjugation,
	$$g(\rho)(\gamma):= g\rho(\gamma)g^{-1},$$
	is free because $D_3$ has trivial center. Therefore
	$$\col_{R_3}(n)=
	3+6\cdot | D_3\backslash\Hom_{/\eta}(\G_n,D_3)^\surj|.$$
	The quotient $D_3\backslash\Hom_{/\eta}(\G_n,D_3)^\surj$ classifies $D_3$-Galois number fields $\widetilde L$ containing $\L_n$ s.t. $\widetilde L/\L_n$ is unramified. Since $\Gal(\L_n/\QQ)$ acts via sign on the class group $Cl(\L_n)$ in its entirety, 
	every unramified $\ZZ/3\ZZ$-extension $N/\L_n$ is Galois over $\QQ$ with
	$$\Gal(N/\QQ)\simeq\ZZ/3\ZZ\rtimes\ZZ/2\ZZ\simeq D_3.$$
	Hence these $\widetilde L$ are precisely the unramified $\ZZ/3\ZZ$-Galois extensions of $\L_n$, of which there are
	$$| D_3\backslash\Hom_{/\eta}(\G_n,D_3)^\surj |=\left|\mathbb{P}_{\FF_3}\left(Cl(\L_n)\otimes_\ZZ\FF_3\right)\right|=\frac{\left|Cl(\L_n)\otimes_\ZZ\FF_3\right|-1}2.$$
	Such extensions are classified by cubic fields $L$ with discriminant
	$$\Disc(L)=\Disc(\L_{n})=\begin{cases}
	n^* &,\;2\nmid n\\
	4n^* &,\;2|n
	\end{cases}\;\;.$$
	It follows by \cref{Prop: cubic with quadratic discriminant are maximal orders} that
	$$ D_3\backslash\Hom_{/\eta}(\G_n,D_3)^\surj \simeq H_3(\Disc(\L_n)).$$
	All in all,
	$$\col_{R_3}(n)=
	3+6\cdot | D_3\backslash\Hom_{/\eta}(\G_n,D_3)^\surj|=$$
	$$ =3+6\cdot \frac{\left|Cl(\L_n)\otimes_\ZZ\FF_3\right|-1}2 = 3\left|Cl(\L_n)\otimes_\ZZ\FF_3\right|= $$
	$$=3+6h_3(\Disc(\L_n)) =\begin{cases}
	3+6h_3(n^*)&,\;2\nmid n\\
	3+6h_3(4n^*)&,\;2|n
	\end{cases}\;\;.$$
\end{proof}

\subsection{Proof of the Main Conjecture for $\k=\RGB$}

\begin{lemma}
\label{Lma: local behavior of partially-split}
	Let $\Delta$ be a fundamental discriminant, let $p|\Delta$ be prime, and let $R/\ZZ$ be a cubic order with discriminant $\Delta$. Then
$$\ZZ_p\otimes_\ZZ R\simeq\ZZ_p\times \ZZ_p[{\sqrt{\Delta}}/2].$$
\end{lemma}
\begin{proof}
	Since $\Delta$ is a fundamental discriminant, then $R$ is the maximal order $R=\O_L$ in the non-Galois cubic number field $L=\QQ\otimes_\ZZ R$. Since $p|\Delta=\Disc(L)$, then $p$ is partially ramified in $L$, hence
	there is some ramified quadratic $M/\QQ_p$ s.t.
	$$\ZZ_p\otimes_\ZZ R\simeq \ZZ_p\times \O_M.$$
The normal closure $\widetilde L$ of $L$ contains $\QQ(\sqrt{\Delta})$. Which is to say that $\QQ(\sqrt{\Delta})$ is a sub-quotient of $L\otimes_\QQ L$. Therefore 
	$\QQ_p(\sqrt{\Delta}):=\QQ_p\otimes_\QQ\QQ(\sqrt{\Delta})$ is a sub-quotient of 
	$$(L\otimes_\QQ L)\otimes_\QQ\QQ_p\simeq(L\otimes_\QQ\QQ_p)\otimes_{\QQ_p} (L\otimes_\QQ\QQ_p)\simeq$$
	$$\simeq (\QQ_p\times M)\otimes_{\QQ_p}(\QQ_p\times M)\simeq \QQ_p\times M^{4}.$$
	Hence $M=\QQ_p(\sqrt{\Delta})$. For all $p|\Delta$,
	$$\O_M=\O_{\QQ_p(\sqrt{\Delta})}=\ZZ_p\otimes_\ZZ \ZZ[\textstyle{\frac{\Delta+\sqrt{\Delta}}2}] =\ZZ_p[\sqrt{\Delta}/2],$$
	and
	$$\ZZ_p\otimes_\ZZ R\simeq\ZZ_p\times\O_M\simeq\ZZ\times\ZZ_p[\sqrt{\Delta}/2].$$

\end{proof}

\begin{corollary}
\label{Lma: QQ2 quadratics}
	Let $R$ be a cubic order s.t. $\Delta:= \Disc(R)$ is a fundamental discriminant. Then there exists even $n\in\NNN$ s.t. $\Delta=4n^*$ iff
	$$\ZZ_2\otimes_\ZZ R\in\{
	\ZZ_2\times\ZZ_2[\sqrt{2}],
	\ZZ_2\times\ZZ_2[\sqrt{-6}]\}.$$
\end{corollary}
\begin{proof}
	The set of fundamental discriminants equals
	$$\{c\cdot m^*\;|\;m\in\NN^*_2,\;c=1,-4,-8,8\}.$$
	The proof of the corollary is straightforward under the assumption that $2\nmid \Delta$ since neither $\ZZ_2\times\ZZ_2[\sqrt{2}]$ nor $\ZZ_2\times\ZZ_2[\sqrt{-6}]$ is \'etale over $\ZZ_2$. Assuming therefore that $2|\Delta$, we have
	$$\ZZ_2\otimes_\ZZ R\simeq\ZZ_2\times\ZZ_2[\sqrt{\Delta}/2]$$
For $m\in\NN^*_2$ we have
$$-m^*\equiv3,-1\mod 8\;,\;\; -2m^*\equiv6,-2\mod 16\;,\;\; 2m^*\equiv2,-6\mod 16,$$
Hence
$$\begin{tabular}{|c|c|c|}
	\hline
	$\Delta$&$\ZZ_2\otimes_\ZZ R$ &$\ZZ_2[\sqrt{\Delta}/2]=$\\
	\hline
	$-4m^*$&$\ZZ_2\times\ZZ_2[\sqrt{-m^*}] $&$\ZZ_2[\sqrt{3}], \ZZ_2[\sqrt{-1}] $\\
	\hline
	$-8m^*$&$\ZZ_2\times\ZZ_2[\sqrt{-2m^*}] $&$\ZZ_2[\sqrt{6}], \ZZ_2[\sqrt{-2}] $\\
	\hline $8m^*$&$\ZZ_2\times\ZZ_2[\sqrt{2m^*}] $&$\ZZ_2[\sqrt{2}], \ZZ_2[\sqrt{-6}] $\\
	\hline
\end{tabular}$$
Having done a comprehensive sweep of all fundamental discriminants, we can verify that there exists even $n=2m\in\NNN$ s.t. $\Disc(R)=4n^*=8m^*$ iff
$$\ZZ_2\otimes_\ZZ R\in\{
	\ZZ_2\times\ZZ_2[\sqrt{2}],
	\ZZ_2\times\ZZ_2[\sqrt{-6}]\}.$$
\end{proof}

\begin{definition}
	For prime $p$ we define the following sets $S_p$ and $S_p^\circ$ of cubic $\ZZ_p$-algerbas up to isomorphism:
	$$S_p= \{W(\FF_{p^3})\}\cup\{\ZZ_p\times \O_{\QQ_p[x]/(x^2-\alpha)}\;|\;\alpha\in\QQ_p^\times/(\QQ_p^\times)^2\},$$
	- which for $p\neq 2$ equals
	$$S_p= \{W(\FF_{p^3})\;,\;\ZZ_p^3 \;,\; \ZZ_p\times \ZZ_p[\sqrt{\alpha_p}] \;,\;
	\ZZ_p\times\ZZ_p[\sqrt{p}] \;,\; \ZZ_p\times \ZZ_p[\sqrt{\alpha_p p}]\},$$
	where $\alpha_p$ is a quadratic non-residue modulo $p$ and $W$ denotes the ring of Witt vectors. The set $S_p^\circ$ consists of all \'etale cubic $\ZZ_p$-algebras:
	$$S_p^\circ= \{W(\FF_{p^3}),\ZZ_p^3,\ZZ_p\times W(\FF_{p^2})\} .$$
	For $p=2$ we also define the set $S_2'$ of $\ZZ_2$-algebras:
	$$S_2'= \{
	\ZZ_2\times\ZZ_2[\sqrt{2}], \ZZ_2\times \ZZ_2[\sqrt{-6}]\}.$$
	Let $s\in\NNN$ and let $p$ be prime. We define the sets $\Sigma_{s,p}$ and $\Sigma_{s,p}'$:
	$$\Sigma_{s,p}=\begin{cases}
	S_p^\circ &,\;p|2s\\
	S_p&,\;p\nmid 2s
	\end{cases}\;\;\;\;\;\;\;\;,\;\;\;\;\;\;\;\;\;\;\;\;\;\;\;\;\Sigma_{s,p}'=\begin{cases}
	S_p^\circ &,\;2\nmid s,\;p|s\\
	S_p&,\;2\nmid s,\;2\neq p\nmid s\\
	S_2'&,\;2\nmid s,\;p=2\\
	\emptyset &,\;2|s,\;p=2\\
	S_p^\circ &,\;2|s,\;2\neq p|s\\
	S_p &,\;2|s,\;p\nmid s
	\end{cases}$$
We also define for $p=\infty$, the following sets of $\RR$-algebras:
$$\Sigma_{s,\infty}=\Sigma_{s,\infty}'=\{\RR^3,\RR\times\CC\}.$$
\end{definition}

\begin{proposition}
\label{Prop: R3 form to local conditions}
	Let $s\in\NNN$ and let $R/\ZZ$ be a cubic order. Then 
	$$R\in \bigcup_{\substack{n\in\NNN\\n\textnormal{ odd}}}H_3(n^*)\;\;\;\;\;\;\textnormal{ - resp.}\;\; R\in \bigcup_{\substack{n\in\NNN\\n\textnormal{ even}}}H_3(4n^*) $$
	iff for all prime $p$,
	$$\ZZ_p\otimes_\ZZ R\in \Sigma_{s,p}\;\;\;\;\;\;\textnormal{ - resp. } \ZZ_p\otimes_\ZZ R\in\Sigma_{s,p}'.$$
\end{proposition}
\begin{proof}
	Following \cref{Hasse 1930} and \cref{Prop: cubic with quadratic discriminant are maximal orders}, the property that $\Disc(R)$ is a fundamental discriminant is completely determined locally: For prime $p$, let
	$$R_p:=\ZZ_p\otimes R.$$
	Then $R\in H_3(\Delta)$ for some fundamental discriminant $\Delta$ iff for all prime $p$, $R_p$ is a regular ring either isomorphic to the Witt vectors $W(\FF_{p^3})$, or of the form $R_p\simeq \ZZ_p\times \O_M$ for some degree-$2$ \'etale $\QQ_p$-algebra $M$:
	$$R_p\in \{W(\FF_{p^3})\}\cup\{\ZZ_p\times \O_{\QQ_p[x]/(x^2-\alpha)}\;|\;\alpha\in\QQ_p^\times/(\QQ_p^\times)^2\}=S_p.$$
Furthermore, for every prime $p$, $p\nmid \Delta$ iff $R_p/\ZZ_p$ is \'etale, which is to say that
\begin{equation}
\label{Eqn: Disc coprime local condition}
p\nmid \Delta\;\iff\;R_p\in S_p^{\circ}
.\end{equation}
The set $\{n^*\;|\; \textnormal{odd } n\in\NNN\}$ equals the set of odd fundamental discriminants, hence $R\in H_3(n^*)$ for some odd $n\in\NNN$ iff
$$R_p\in\begin{cases}
	S_p &,\;p\neq 2\\
	S_p^\circ &,\;p= 2
\end{cases}\;\;.$$
By \cref{Lma: QQ2 quadratics}, the set $\{4n^*\;|\; \textnormal{even } n\in\NNN\}= \{8m^*\;|\; m\in\NN^*_2\} $ equals the set of fundamental discriminants $\Delta$ for which
$$\QQ_2(\Delta)=\QQ_2(\sqrt{2}), \QQ_2(\sqrt{-6}).$$
Following \cref{Lma: local behavior of partially-split},
$R\in H_3(4n^*)$ for some even $n\in\NNN$ iff
$$R_p\in\begin{cases}
	S_p &,\;p\neq 2\\
	S_2' &,\;p= 2
\end{cases}\;\;.$$
By (\ref{Eqn: Disc coprime local condition}), placing further $s$-coprimality conditions on $\Disc(R)$ is tantamount to intersecting any prior local conditions on $R$ with $S_p^\circ$ for all $p|s$:
The cubic order $R$ is in $H_3(n^*)$ for some odd $n\in\NNs$ iff for every prime $p$,
$$R_p\in\begin{cases}
	S_p &,\;2\neq p\nmid s\\
	S_2^\circ &,\;2=p\nmid s\\
	S_p\cap S_p^\circ &,\;2\neq p|s\\
	S_2^\circ\cap S_p^\circ &,\;2=p|s
\end{cases}\;\;\;\;= \begin{cases}
	S_p &,\;p\nmid 2s\\
	S_p^\circ &,\;p| 2s
	\end{cases}\;\;\;\;=\Sigma_{s,p}.$$
	Similarly, $R$ is in $H_3(4n^*)$ for some even $n\in\NNN$ iff
	for every prime $p$,
$$R_p\in\begin{cases}
	S_p &,\;2\neq p\nmid s\\
	S_2' &,\;2=p\nmid s\\
	S_p\cap S_p^\circ &,\;2\neq p|s\\
	S_2'\cap S_p^\circ &,\;2=p|s
\end{cases}\;\;\;\;= \begin{cases}
	S_p &,\;2\neq p\nmid s\\
	S_2' &,\;2=p\nmid s\\
	S_p^\circ &,\;2\neq p|s\\
	\emptyset &,\;2=p|s
\end{cases}\;\;\;\;=\Sigma_{s,p}'.$$
\end{proof}

\begin{proposition}
\label{Prop: R3 form to interface Bhargava}
Let $s\in\NNN$. Then
$$\lim_{X\to\infty}\frac1X \left(\sum_{\substack{1\neq n\;\textnormal{odd}\\(n,s)=1}}^{n\le X}h_3(n^*) +\sum_{\substack{n \;\textnormal{even}\\(n,s)=1}}^{n\le X}h_3(4n^*)\right) =\frac{\gamma_s}3,$$
where
$$\gamma_s=\prod_{p|s}(1-p^{-1})\prod_{p\nmid s}(1-p^{-2}).$$
\end{proposition}
\begin{proof}
Let $X\ge 1$. By \cref{Prop: R3 form to local conditions}, the sum
$$ \sum_{\substack{1\neq n\;\textnormal{odd}\\(n,s)=1}}^{n\le X}h_3(n^*)$$
counts cubic orders $R/\ZZ$ for which $\ZZ_p\otimes R\in\Sigma_{s,p}$ for all primes $p$, and
$$|\Disc(R)|=|n^*|=n\le X.$$
Likewise, the sum
$$ \sum_{\substack{n\;\textnormal{even}\\(n,s)=1}}^{n\le X}h_3(4n^*)$$
counts cubic orders $R/\ZZ$ for which $\ZZ_p\otimes R\in\Sigma_{s,p}'$ for all primes $p$, and
$$|\Disc(R)|=|4n^*|=4n\le 4X.$$
By Theorem 8 in \cite{WEBSITE:BhargavaShankTsim2012}, 
$$\lim_{X\to\infty}\frac1X \sum_{\substack{1\neq n\;\textnormal{odd}\\(n,s)=1}}^{n\le X}h_3(n^*) =C(\Sigma_{s,\infty})\prod_pC(\Sigma_{s,p}),$$
$$... \lim_{X\to\infty}\frac1{4X} \sum_{\substack{n\;\textnormal{even}\\(n,s)=1}}^{n\le X}h_3(4n^*) = C(\Sigma_{s,\infty}')\prod_pC(\Sigma_{s,p}'),$$
where for $\widetilde S_p$ a set of isomorphism classes of cubic $\ZZ_p$-algebras, $C(\widetilde S_p)$ is a local density factor defined as
$$C(\widetilde S_p)=(1-p^{-1})\sum_{R\in \widetilde S_p}\frac1{\Disc_p(R)}\cdot\frac1{|\Aut(R)|},$$
and for $\widetilde S_\infty$ a set of cubic $\RR$-algebras,
$$C(\widetilde S_\infty)=\frac12\sum_{R\in \widetilde S_\infty}\frac1{|\Aut(R)|}.$$
Hence
$$\lim_{X\to\infty}\frac1X \left(\sum_{\substack{1\neq n\;\textnormal{odd}\\(n,s)=1}}^{n\le X}h_3(n^*) +\sum_{\substack{n \;\textnormal{even}\\(n,s)=1}}^{n\le X}h_3(4n^*)\right) =$$
$$= C(\Sigma_{s,\infty})\prod_pC(\Sigma_{s,p}) +4 C(\Sigma_{s,\infty}')\prod_pC(\Sigma_{s,p}').$$
For all every prime $p\neq 2$, $\Sigma_{s,p}=\Sigma_{s,p}'$ and
$$\sum_{R\in\Sigma_{s,p}}\frac1{\Disc_p(R)}\cdot \frac1{|\Aut(R)|}=\begin{cases}
	\frac16+ \frac12+ \frac13=1&,\;p|s\\
	\frac16+ \frac12+ \frac13+\frac1p\cdot\frac12+\frac1p\cdot\frac12=1+p^{-1}
	&,\;p\nmid s
\end{cases}\;\;.$$
Hence
$$C(\Sigma_{s,p})= C(\Sigma_{s,p}')= (1-p^{-1}) \sum_{R\in\Sigma_{s,p}}\frac1{\Disc_p(R)}\cdot \frac1{|\Aut(R)|}=\begin{cases}
1-p^{-1}&,\;p|s\\
1-p^{-2}&,\;p\nmid s
\end{cases}\;\;.$$
For the case at $\infty$,
$$C(\Sigma_{s,\infty})= C(\Sigma_{s,\infty}')= \frac12(\frac16+\frac12)=\frac13.$$
For $p=2$, we compute $C(\Sigma_{s,2})+4C(\Sigma_{s,2}')$. If $s$ is even, then
$$C(\Sigma_{s,2})+4C(\Sigma_{s,2}')=C(\Sigma_{s,2})=\frac12=1-2^{-1}$$
if $s$ is odd, then
$$C(\Sigma_{s,2})+4C(\Sigma_{s,2}')=\frac12+4\cdot\frac12\left(\frac18\cdot\frac12+ \frac18\cdot\frac12\right)=\frac34=1-2^{-2}.$$
Hence regardless of parity of $s$, we have
$$\lim_{X\to\infty}\frac1X \left(\sum_{\substack{1\neq n\;\textnormal{odd}\\(n,s)=1}}^{n\le X}h_3(n^*) +\sum_{\substack{n \;\textnormal{even}\\(n,s)=1}}^{n\le X}h_3(4n^*)\right) =$$
$$=C(\Sigma_{s,\infty})\prod_pC(\Sigma_{s,p}) +4 C(\Sigma_{s,\infty}')\prod_pC(\Sigma_{s,p}') =$$
$$=\frac13\prod_{p|s}(1-p^{-1})\prod_{p\nmid s}(1-p^{-2})=\frac{\gamma_s}3.$$

\end{proof}

\begin{proposition}
\label{Prop: Main Conj for RGB}
	The function $\col_{\RGB}$ has generic summatory type
	$$\col_{\RGB}\in\GenType 05.$$
\end{proposition}
\begin{proof}
	Let $s\in\NNN$. 
	For $n=1$ we have $\kei_n=\emptyset$, therefore
	$$\col_{\RGB}(n)=|\Hom_\Kei(\emptyset, \RGB)|=1.$$
	For all $X\ge 1$ we have
	$$\counter_{\RGB,s}(X)=1+\sum_{(n,s)=1}^{2\le n\le X}\col_{\RGB}(n)=1+ \sum_{\substack{1\neq n\;\textnormal{odd}\\(n,s)=1}}^{n\le X}(3+6h_3(n^*)) +\sum_{\substack{n \;\textnormal{even}\\(n,s)=1}}^{n\le X}(3+6h_3(4n^*))=$$
	$$-2+3\sum_{1\le n\le X}\mu_s^2(n)+6\left( \sum_{\substack{1\neq n\;\textnormal{odd}\\(n,s)=1}}^{n\le X}h_3(n^*) +\sum_{\substack{n \;\textnormal{even}\\(n,s)=1}}^{n\le X} h_3(4n^*)\right).$$
	Therefore
	$$\lim_{X\to\infty}\frac{\counter_{\RGB,s}(X)}X=$$
	$$=3 \lim_{X\to\infty}\frac{\counter_{\triv_1,s}(X)}X+6\lim_{X\to\infty}\left( \sum_{\substack{1\neq n\;\textnormal{odd}\\(n,s)=1}}^{n\le X}h_3(n^*) +\sum_{\substack{n \;\textnormal{even}\\(n,s)=1}}^{n\le X} h_3(4n^*)\right)=$$
	$$=3\gamma_s+6\cdot \frac{\gamma_s}3=5\gamma_s.$$
We find that
$$c_s(\RGB)= \lim_{X\to\infty}\frac{\avg_{\RGB,s}(X)}{\log ^0(X)} =\lim_{X\to\infty}\frac{\counter_{\RGB,s}(X)}{\gamma_sX}=5$$
for all $s\in\NNN$, and subseuqently
$$\lim_{s\in\NNN}c_s(\RGB)=5.$$
Hence
$$\col_{\RGB}\in\GenType 05.$$

\end{proof}

\bibliography{Ariel_Davis_bibliography.bib}
\bibliographystyle{alpha}
\end{document}